\documentclass[11pt,a4paper]{article}
\usepackage[utf8]{inputenc}
\usepackage{amsmath}
\usepackage{amssymb}
\usepackage{amsthm}
\usepackage[english]{babel}
\usepackage{booktabs}
\usepackage{graphicx}
\usepackage[colorlinks,allcolors=blue]{hyperref}
\usepackage{pgfplots}
\pgfplotsset{compat=1.14}
\usepackage{xcolor}
\usepackage[hmargin={30mm,30mm},vmargin={30mm,35mm}]{geometry}
\usepackage{caption}
\usepackage{subcaption}
\usepackage{algorithm}
\usepackage{algpseudocode} % To write the pseudocode
\usepackage{nicefrac} % To write in a better way the exponential fractions
\usepackage[affil-it]{authblk}

\usepackage{marginnote} % margin notes for revised version

\numberwithin{equation}{section}    %The number of the equations change according to the section

\usepackage{xpatch} % Name of theorems are bold
\makeatletter
   \xpatchcmd{\@thm}{\fontseries\mddefault\upshape}{}{}{} % same font as thm-header
\makeatother

\usepackage{newtxtext}
\usepackage{newtxmath}

\DeclareRobustCommand{\VEC}[1]{\boldsymbol{#1}}
\DeclareRobustCommand{\MAT}[1]{\boldsymbol{\mathbb{#1}}}
\tikzstyle{arrow} = [thick,->,>=stealth]

\definecolor{cadmiumgreen}{rgb}{0.0, 0.42, 0.24}
\definecolor{bole}{rgb}{0.47, 0.27, 0.23}
\definecolor{cyan(process)}{rgb}{0.0, 0.72, 0.92}

\definecolor{colorH1ad}{rgb}{0.803921568627451,0.36078431372549,0.36078431372549}
\definecolor{colorEnergyun}{rgb}{0.254901960784314,0.411764705882353,0.882352941176471}
\definecolor{colorEnergyad}{rgb}{1,0.647058823529412,0}
\definecolor{colorTEun}{rgb}{0.0980392156862745,0.0980392156862745,0.43921568627451}
\definecolor{colorUpad}{rgb}{0.647058823529412,0.164705882352941,0.164705882352941}
\definecolor{colorEffEnergy}{rgb}{0.729411764705882,0.333333333333333,0.827450980392157}
\definecolor{colorEffUp}{rgb}{0.580392156862745,0,0.827450980392157}
\definecolor{colorEstTot}{rgb}{0.12156862745098,0.466666666666667,0.705882352941177}
\definecolor{colorEstStress}{rgb}{1,0.498039215686275,0.0549019607843137}
\definecolor{colorEstLin}{rgb}{0.172549019607843,0.627450980392157,0.172549019607843}
\definecolor{colorEstContactn}{rgb}{0.580392156862745,0.403921568627451,0.741176470588235}
\definecolor{colorEstContactt}{rgb}{0.83921568627451,0.152941176470588,0.156862745098039}

\pgfarrowsdeclare{mylatex'}{mylatex'}
{
	\newdimen\len
	\len=\pgfgetarrowoptions{mylatex'}
	\pgfarrowsleftextend{-0.4\len}
	\pgfarrowsrightextend{0.6\len}
}
{
	\newdimen\len
	\len=\pgfgetarrowoptions{mylatex'}
	\pgfpathmoveto{\pgfqpoint{0.6\len}{0\len}}
	\pgfpathcurveto
	{\pgfqpoint{0.35\len}{0.05\len}}
	{\pgfqpoint{-0.1\len}{0.15\len}}
	{\pgfqpoint{-0.4\len}{0.375\len}}
	\pgfpathcurveto
	{\pgfqpoint{-0.15\len}{0.1\len}}
	{\pgfqpoint{-0.15\len}{-0.1\len}}
	{\pgfqpoint{-0.4\len}{-0.375\len}}
	\pgfpathcurveto
	{\pgfqpoint{-0.1\len}{-0.15\len}}
	{\pgfqpoint{0.35\len}{-0.05\len}}
	{\pgfqpoint{0.6\len}{0\len}}
	\pgfusepathqfill
}
\pgfsetarrowoptions{mylatex'}{8pt}
\pgfkeys{/tiplen/.default=8pt, /tiplen/.code={\pgfsetarrowoptions{mylatex'}{#1}}}

\newcommand{\fric}{\mu_{\rm Coul}}
\newcommand{\gamD}{\Gamma_{\rm D}}
\newcommand{\gamN}{\Gamma_{\rm N}}
\newcommand{\gamC}{\Gamma_{\rm C}}
\newcommand{\gN}{\VEC{g}_{\rm N}}
\newcommand{\HunoD}[1]{\VEC{H}^1_{\rm D}(#1)}
\newcommand{\faces}[2]{\mathcal{F}_{#1}^{\rm #2}}
\newcommand{\vertices}[2]{\mathcal{V}_{#1}^{\rm #2}}

\newcommand{\vvvert}{\vert\kern-0.25ex\vert\kern-0.25ex\vert}
\newcommand{\norm}[1]{\vvvert #1 \vvvert}
\newcommand{\normHuno}[1]{\left\lVert #1\right\rVert_{1,\Omega}}
\newcommand{\normGamma}[1]{\left\lvert #1\right\rvert_{{\rm C},h}}
\newcommand{\normT}[1]{\left\lvert #1\right\rvert_{{\rm C},T}}

\newcommand{\energynorm}[1]{\left\lVert #1\right\rVert_{\rm en}}
\newcommand{\localdualnormresidual}[1]{{\color{black}\norm{\mathcal{R}_{\mathcal{T}_T}(#1)}_{*,\tilde{\omega}_T}}}
\newcommand{\dualnormresidual}[1]{{\color{black}\norm{\mathcal{R}(#1)}_*}}

\newtheorem{theorem}{Theorem}

\newtheorem{lemma}[theorem]{Lemma}

\theoremstyle{remark}
\newtheorem{remark}[theorem]{Remark}
\theoremstyle{definition}
\theoremstyle{definition}
\newtheorem{definition}[theorem]{Definition}

\newtheorem*{Unilateral*}{Unilateral contact problem with no friction}

\theoremstyle{plain}

\newtheorem{construction}[theorem]{Construction}

\newtheorem{assumption}[theorem]{Assumption}

\DeclareMathOperator{\tr}{tr}

\DeclareCaptionLabelFormat{andtable}{#1~#2  \&  \tablename~\thetable}

\title{An a posteriori error analysis based on equilibrated stresses for finite element approximations of frictional contact}

\author[,1]{Ilaria Fontana \thanks{Corresponding author}}
\author[2]{Daniele A. Di Pietro}

\newcommand\affilcr{\protect\\ \protect\it}

\affil[1]{Department of Engineering Sciences and Applied Mathematics,\affilcr Northwestern University, Evanston, IL 60208, USA}
\affil[2]{IMAG, Univ Montpellier, CNRS, Montpellier, France}

\def\displaycancellation{}
\def\displaycomment{}

\usepackage[normalem]{ulem}
\newcounter{corr}
\definecolor{violet}{rgb}{0.580,0.,0.827}
\newcommand{\corr}[3]{\typeout{Warning : a correction remains in page \thepage}
	\stepcounter{corr}        
				      {\ifthenelse{\isundefined{\displaycancellation}}{}{\color{blue}\ifmmode\text{\,\sout{\ensuremath{#1}}\,}\else\sout{#1}\fi}}
              {\color{red}#2}
              {\ifthenelse{\isundefined{\displaycomment}}{}{\color{violet} #3}}
}

\date{}

\begin{document}

\maketitle

\begin{abstract}
    We consider the unilateral contact problem between an elastic body and a rigid foundation in a description that includes both Tresca and Coulomb friction conditions.
    For this problem, we present an a posteriori error analysis based on an equilibrated stress reconstruction in the Arnold--Falk--Winther space that includes a guaranteed upper bound distinguishing the different components of the error. This analysis is the starting point for the development of an adaptive algorithm including a stopping criterion for the generalized Newton method. This algorithm is then used to perform numerical simulations that validate the theoretical results.
\end{abstract}

\noindent
\textbf{Keywords:} frictional unilateral contact problem, weakly enforced contact conditions, a posteriori error estimate, equilibrated stress reconstruction, Arnold--Falk--Winther mixed finite element, adaptive algorithms
\medskip\\
\textbf{MSC2020 classification:} 74M15, 74S05, 65N15, 65N30, 65N50

%------------------------------------------------------------------------------%

%% \tableofcontents

%------------------------------------------------------------------------------%

\section{Introduction}

In the field of mechanical engineering, accounting for contact conditions with friction is indispensable for the accurate representation of the behavior of an elastic object. This is particularly crucial in structural studies, such as those involving dams, where ensuring safety demands a thorough analysis. The most frequently used friction conditions in mechanical problems are Tresca and Coulomb conditions.
In this work, we consider the Tresca or Coulomb unilateral contact problems between an elastic body and a rigid foundation.
Numerous techniques exist in the literature for approximating solutions to such problems, including penalty formulations \cite{Kikuchi1981}, mixed formulations \cite{Haslinger1996}, and methods based on the weak enforcement of contact conditions à la Nitsche \cite{Chouly-Mlika2017}.
Since the original work on Dirichlet boundary conditions \cite{Nitsche1971}, Nitsche's technique has been extended to general boundary conditions \cite{Juntunen_2009} and, more recently, to the unilateral contact problem without friction \cite{Chouly2013}; see \cite{Chouly-Mlika2017} for a review including subsequent developments. 
From the numerical standpoint, this technique is appealing as it does not require the introduction of Lagrange multipliers and results in an easily implementable formulation.
The starting point is the general formulation proposed in \cite{Araya2023}, which covers both Tresca and Coulomb friction cases.

The primary focus of this paper is the extension of the a posteriori error analysis via equilibrated stress reconstruction developed in \cite{DiPietro2022} to the frictional unilateral contact problem. This technique, presented in \cite{Ern2015} for the Poisson problem, is based on the Prager--Synge inequality \cite{Prager.Synge:47}. 
We provide a guaranteed upper bound of the dual norm of the error without using any saturation assumption, see Theorems~\ref{th:a posteriori} and \ref{th:a posteriori distinguishing} below. The presented upper bound is expressed in terms of fully computable a posteriori error estimators which distinguish the different sources of error.
Besides the data of the problem (volumetric and surface loadings), these estimators involve the approximate solution $\VEC{u}_h$ and an equilibrated stress reconstruction $\VEC{\sigma}_h$, i.e., an $\MAT{H}({\bf div})$-conforming correction of the stress tensor $\VEC{\sigma}(\VEC{u}_h)$ locally in equilibrium with the force source terms.

The paper is organized as follows. In Section~\ref{sec:problem description}, we introduce the frictional unilateral contact problem in both strong and weak forms, along with its discretization à la Nitsche.
Additionally, we provide the main space-related and mesh-related notations that we will use throughout the rest of the paper in Tables~\ref{tab:spaces notation} and \ref{tab - mesh notation}. Section~\ref{sec:a posteriori analysis} showcases the main results of the work: a basic a posteriori error estimate, a refined version distinguishing the components of the error, and an adaptive algorithm based on the latter. This algorithm includes an adaptive stopping criterion for the linearization iterations. Moreover, we compare the dual norm of the residual with the energy norm. In Section~\ref{sec:stress reconstruction} we explicitly propose how to construct an equilibrated stress reconstruction with the right properties by assembling the solutions of local problems on element patches around mesh vertices.
Section~\ref{sec:numerical results} validates the results of Section~\ref{sec:a posteriori analysis} with two numerical examples with Tresca and Coulomb friction conditions, respectively.
Finally, Section~\ref{sec:efficiency} concludes the work by presenting results of local and global efficiency.

\section{Setting}\label{sec:problem description}

\subsection{Continuous problem}

\begin{figure}[tb]
    \centering
    \includegraphics{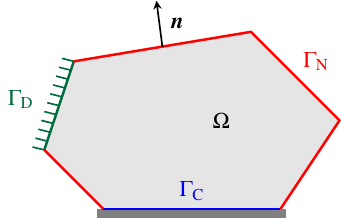}
    \caption{Example of domain $\Omega$ with $d=2$. The boundary $\partial\Omega$ is subdivided into $\gamD$ (in {\em green}), $\gamN$ (in {\em red}), and $\gamC$ (in {\em blue}).}
    \label{fig:domain_example}
\end{figure}

We consider a domain $\Omega\in\MAT{R}^d$, $d\in\{2,3\}$, which represents a body with elastic behavior, and we suppose for simplicity that $\Omega$ is a polygon if $d=2$ or a polyhedron if $d=3$, so that it can be covered exactly by a finite element mesh.
The boundary of the domain is denoted by $\partial\Omega$ and is subdivided into three nonoverlapping parts: $\gamD$, $\gamN$, and $\gamC$ (see Figure~\ref{fig:domain_example} for a two-dimensional example) such that $\lvert\gamD\rvert > 0$ and $\lvert\gamC\rvert > 0$, where $\lvert\,\cdot\,\rvert$ is the Hausdorff measure.
The setting of the problem is the following: the body is clamped at $\gamD$, it is subject to volumetric forces $\VEC{f}\in\VEC{L}^2(\Omega)$ in $\Omega$ and to surface forces $\gN\in\VEC{L}^2(\gamN)$ on the portion of the boundary $\gamN$, in the reference configuration it is in contact with a rigid foundation on $\gamC$, while in the deformed configuration the contact region is included in $\gamC$.
On the boundary $\partial\Omega$, we consider the outward unit normal vector $\VEC{n}$ that allows us to decompose any displacement field $\VEC{v}$ and any density of surface force $\VEC{\sigma}(\VEC{v}) \VEC{n}$ into their normal and tangential components:
\begin{equation}
    \VEC{v}=v^n\VEC{n}+\VEC{v}^{\VEC{t}} \qquad\qquad \text{and} \qquad\qquad \VEC{\sigma}(\VEC{v})\VEC{n}=\sigma^n(\VEC{v}) \VEC{n} + \VEC{\sigma}^{\VEC{t}}(\VEC{v}).
\end{equation}

The frictional unilateral contact problem then reads:
Find the displacement $\VEC{u}\colon \Omega\to \MAT{R}^d$ such that
\begin{subequations}\label{eq:unilateral problem}
    \begin{alignat}{3}
      %\VEC{\nabla}\cdot 
      \VEC{\rm div}\, 
      \VEC{\sigma}(\VEC{u})+\VEC{f}=\VEC{0}
      & \qquad & \qquad & \text{in $\Omega$}, \label{eq:unilateral equilibrium}
      \\ \label{eq:sigma}
      \VEC{\sigma}(\VEC{u}) = \lambda \tr\VEC{\varepsilon}(\VEC{u}) \VEC{I}_d + 2\mu\VEC{\varepsilon}(\VEC{u})
      & \qquad & \qquad & \text{in $\Omega$},
      \\
      \VEC{u} = \VEC{0}
      & \qquad & \qquad & \text{on $\gamD$}, \label{eq:Dirichlet condition}
      \\
      \VEC{\sigma}(\VEC{u}) \VEC{n} = \gN
      & \qquad & \qquad & \text{on $\gamN$}, \label{eq:unilateral Neumann}
      \\
      u^n\leq 0,\ \sigma^n(\VEC{u})\leq 0,\ \sigma^n(\VEC{u})\, u^n = 0
      & \qquad & \qquad & \text{on $\gamC$}, \label{eq:unilateral contact 1}
      \\
      \begin{cases}
          |\VEC{\sigma}^{\VEC{t}}(\VEC{u})| \leq S(\VEC{u}) & \qquad\text{if}\ \VEC{u}^{\VEC{t}} = \VEC{0}\\
          \VEC{\sigma}^{\VEC{t}}(\VEC{u}) = -S(\VEC{u}) \displaystyle\frac{\VEC{u}^{\VEC{t}}}{\left|\VEC{u}^{\VEC{t}}\right|} & \qquad \text{otherwise}
      \end{cases}
      & \qquad & \qquad & \text{on $\gamC$}. \label{eq:unilateral contact 2}
    \end{alignat}
\end{subequations}
Here, $\VEC{\varepsilon}(\VEC{v})\coloneqq\frac12(\VEC{\nabla v}+\VEC{\nabla v}^\top)$ is the strain tensor field, $\VEC{\sigma}(\VEC{v}) \in \mathbb{R}^{d\times d}_{\text{sym}}$ is the Cauchy stress tensor, $\VEC{\rm div}$ is the divergence operator acting row-wise on tensor valued functions, $\mu$ and $\lambda$ denote the Lam\'e parameters, and $\lvert\,\cdot\,\rvert$ in \eqref{eq:unilateral contact 2} is the Euclidian norm in $\MAT{R}^{d-1}$.

The first contact condition \eqref{eq:unilateral contact 1} is a complementary condition representing non-penetration ($u^n \leq 0$) and the absence of normal cohesive forces (if $u^n < 0$, then $\sigma^n(\VEC{u}) = 0$). 
The second contact condition represents the friction condition and makes it possible to include in this formulation both Tresca and Coulomb models (see also \cite{Araya2023}): $S(\VEC{u}) = s \in L^2(\gamC)$, $s\leq 0$, for the Tresca friction model, and $S(\VEC{u}) = -\fric\, \sigma^n(\VEC{u})$ for the Coulomb one, where $\fric \geq 0$ is the Coulomb friction coefficient.
We remark that, while Tresca friction is easier to implement and analyze since the friction threshold $s$ is a known function, Coulomb friction allows us to model the fact that there is no friction when the elastic body is not in contact with the rigid foundation in the deformed configuration.

\begin{table}[ht]
  \centering
  \begin{tabular}{p{0.25\textwidth}p{0.70\textwidth}}
    \toprule
    \multicolumn{1}{c}{\textbf{Notation}} & \multicolumn{1}{c}{\textbf{Definition}} \\
    \midrule
    $X$ & Measurable set of $\MAT{R}^d$ or $\MAT{R}^{d-1}$, typically: $\Omega$, a portion of $\partial\Omega$, or the union of a finite subset of mesh elements \\
    $H^s(X)$ & Sobolev space of index $s$ on $D$ \\ 
    $\VEC{H}^s(X)$ & Vector Sobolev space $[H^s(X)]^d$ \\ 
    $\MAT{H}^s(X)$ & Tensor Sobolev space $[H^s(X)]^{d\times d}$ \\ 
    $L^2(X)$ & Sobolev space $H^0(D)$ of square-integrable functions on $X$ \\
    $\VEC{L}^2(X)$ & Vector space $[L^2(X)]^d$ or $[L^2(X)]^{d-1}$ \\
    $\MAT{L}^2(X)$ & Tensor space $[L^2(X)]^{d\times d}$ \\
    $\lVert\,\cdot\,\rVert_{s,X}$ & Norm of $H^s(X)$ or $\VEC{H}^s(X)$ according to the argument \\
    $\lVert\,\cdot\,\rVert_X$ & Norm of $L^2(X)$, $\VEC{L}^2(X)$ or $\MAT{L}^2(X)$ according to the argument \\
    $\lVert\,\cdot\,\rVert$ & Norm of $L^2(\Omega)$, $\VEC{L}^2(\Omega)$ or $\MAT{L}^2(\Omega)$ according to the argument \\
    $(\,\cdot\, ,\,\cdot\,)_X$ & Inner product of $L^2(X)$, $\VEC{L}^2(X)$ or $\MAT{L}^2(X)$ according to the argument \\ 
    $(\,\cdot\, , \,\cdot\,)$ & Inner product of $L^2(\Omega)$, $\VEC{L}^2(\Omega)$ or $\MAT{L}^2(\Omega)$ according to the argument \\
    $\MAT{H}(\textbf{div},\Omega$) & Space spanned by functions of $\MAT{L}^2(\Omega)$ with weak (row-wise) divergence in $\VEC{L}^2(\Omega)$ \\
    \bottomrule
  \end{tabular}
  \caption{Space-related notations.}
  \label{tab:spaces notation}
\end{table}

Table~\ref{tab:spaces notation} summarizes the main space-related notations used in the paper. 
In addition, we define the space $\HunoD{\Omega}$ as the space of functions of $\VEC{H}^1(\Omega)$ satisfying the homogeneous Dirichlet boundary condition \eqref{eq:Dirichlet condition}, and the space $\VEC{K}$ of admissible displacement:
\begin{equation*}
    \HunoD{\Omega} \coloneqq \left\{\VEC{v}\in \VEC{H}^1(\Omega)\ :\ \VEC{v}=\VEC{0} \ \text{on}\ \gamD\right\},
    \qquad
    \VEC{K} \coloneqq \left\{\VEC{v}\in \HunoD{\Omega}\ :\ v^n\leq 0 \ \text{on}\ \gamC\right\}.
\end{equation*}
The weak formulation of the problem \eqref{eq:unilateral problem} is the following variational inequality (see, e.g. \cite{Haslinger1996, Chouly-Mlika2017}): Find $\VEC{u}\in\VEC{K}$ such that
\begin{equation}\label{eq:weak formulation}
    a(\VEC{u},\VEC{v}-\VEC{u}) + j(\VEC{u}; \VEC{v}) - j(\VEC{u}; \VEC{u}) \geq L(\VEC{v}-\VEC{u}) \qquad \forall\VEC{v}\in\VEC{K}, 
\end{equation}
where 
\begin{gather}
    a(\VEC{w},\VEC{v}) \coloneqq (\VEC{\sigma}(\VEC{u}),\VEC{\varepsilon}(\VEC{v})),
    \qquad
    L(\VEC{v})\coloneqq(\VEC{f},\VEC{v}) + (\gN,\VEC{v})_{\gamN}, \label{eq:definition a and L}\\ 
    j(\VEC{u}; \VEC{v}) \coloneqq (S(\VEC{u}), \lvert\VEC{v}^{\VEC{t}}\rvert)_{\gamC}.
\end{gather}
for all $(\VEC{u},\VEC{v})\in\VEC{H}^1(\Omega)\times\VEC{H}^1(\Omega)$.
It is known that this formulation has a unique solution for Tresca friction, while for Coulomb friction the analysis is more intricate. We refer to \cite{Araya2023} and the references therein for a discussion.

\subsection{Discrete problem}

We consider now a family $\{\mathcal{T}_h\}_h$ of conforming triangulations of $\Omega$, indexed by the mesh size $h\coloneqq\max_{T\in\mathcal{T}_h}h_T$, where $h_T$ is the diameter of the element $T$.
This family is assumed to be regular in the classical sense; see, e.g., \cite[Eq. (3.1.43)]{Ciarlet:02}.
Furthermore, each triangulation is conformal to the subdivision of the boundary into $\gamD$, $\gamN$, and $\gamC$ in the sense that the interior of a boundary edge (if $d=2$) or face (if $d=3$) cannot have a non-empty intersection with more than one part of the subdivision.
Mesh-related notations that will be used in the a posteriori error analysis are collected in Table~\ref{tab - mesh notation}.
For the sake of simplicity, from this point on we adopt the three-dimensional terminology and speak of faces instead of edges also in dimension $d=2$.

\begin{table}[ht]
  \centering
  \begin{tabular}{p{0.25\textwidth}p{0.70\textwidth}}
    \toprule
    \multicolumn{1}{c}{\textbf{Notation}} & \multicolumn{1}{c}{\textbf{Definition}} \\
    \midrule
    $\mathcal{F}_h$ & Set of faces of $\mathcal{T}_h$ \\
    $\faces{h}{b}$ & Set of boundary faces, i.e.,  $\{F\in\mathcal{F}_h\ :\ F\subset \partial\Omega\}$ \\
    $\faces{h}{D}\cup\faces{h}{N}\cup\faces{h}{C}$ & Partition of $\faces{h}{b}$ induced by the boundary and contact conditions \\
    $\faces{h}{i}$ & Set of interior faces, i.e., $\mathcal{F}_h\setminus \faces{h}{b}$ \\
    $\mathcal{F}_T$ & Set of faces of the element $T\in\mathcal{T}_h$, i.e., $\{F\in\mathcal{F}_h\ :\ F\subset\partial T\}$ \\
    $\mathcal{F}_T^\bullet$, $\bullet\in\{{\rm b}, {\rm D}, {\rm N}, {\rm C}\}$ & $\mathcal{F}_T\cap\mathcal{F}_h^\bullet$, $T\in\mathcal{T}_h$ \\
    $\mathcal{V}_h$ & Set of all the vertices of $\mathcal{T}_h$ \\
    $\vertices{h}{b}$ & Set of boundary vertices, i.e., $\{\VEC{a}\in\mathcal{V}_h\ :\ \VEC{a}\in\partial\Omega\}$ \\
    $\vertices{h}{i}$ & Set of interior vertices, i.e., $\mathcal{V}_h\setminus\vertices{h}{b}$ \\
    $\mathcal{V}_T$ & Set of vertices of the element $T\in\mathcal{T}_h$, i.e., $\{\VEC{a}\in\mathcal{V}_h\ :\ \VEC{a}\in\partial T\}$ \\
    $\mathcal{V}_F$ & Set of vertices of the mesh face $F\in\mathcal{F}_h$, i.e., $\{\VEC{a}\in\mathcal{V}_h\ :\ \VEC{a}\in\partial F\}$ \\
    $\vertices{h}{D}$ & Set of Dirichlet boundary vertices, i.e., $\{\VEC{a}\in\mathcal{V}_h\ :\ \VEC{a}\in F, F\in\faces{h}{D}\}$ \\
    $\omega_{\VEC{a}}$ & Union of the elements sharing the vertex $\VEC{a}\in\mathcal{V}_h$, i.e., $\displaystyle\bigcup_{T\in\mathcal{T}_h,\,\VEC{a}\in \partial T} T$ \\
    \bottomrule
  \end{tabular}
  \caption{Mesh-related notations.}
  \label{tab - mesh notation}
\end{table}

For any $X\in\mathcal{T}_h\cup\mathcal{F}_h$ mesh element or face, $\mathcal{P}^n(X)$ is the space of $d$-variate polynomials of total degree $\leq n$ defined on $X$, and we set $\VEC{\mathcal{P}}^n(X)\coloneqq[\mathcal{P}^n(X)]^d$ and $\MAT{P}^n(X) \coloneqq [\mathcal{P}^n(X)]^{d\times d}$.
We will seek the approximate solution in the standard Lagrange finite element space of degree $p\geq 1$ with strongly enforced boundary condition on $\gamD$:
\begin{equation}
    \VEC{V}_h \coloneqq \left\{ \VEC{v}_h\in \HunoD{\Omega}\ :\ \VEC{v}_h|_T\in \VEC{\mathcal{P}}^p(T) \ \text{for any}\ T\in\mathcal{T}_h \right\}.
\end{equation}

The key idea of the method we focus on consists in rewriting the contact boundary conditions \eqref{eq:unilateral contact 1} and \eqref{eq:unilateral contact 2} in a compact way and enforcing them à la Nitsche.
For this purpose, we introduce the projector $[x]_{\MAT{R}^-} = \frac{1}{2} (x-\lvert x\rvert)$ on the half-line of negative numbers $\MAT{R}^-$, and the orthogonal projector $[\VEC{x}]_{\alpha} \colon \MAT{R}^{d-1} \to \MAT{R}^{d-1}$ on the $(d-1)$-dimensional ball $B(\VEC{0},\alpha)$ centered in $\VEC{0}$ with radius $\alpha > 0$, i.e., 
\begin{equation}
    [\VEC{x}]_{\alpha} = \begin{cases}
        \VEC{x} & \qquad\text{if}\ \lvert\VEC{x}\rvert \leq \alpha, \\
        \alpha \displaystyle\frac{\VEC{x}}{\lvert\VEC{x}\rvert} & \qquad\text{otherwise}.
    \end{cases}
\end{equation}
In addition, for every real number $\theta$ and every positive bounded function $\gamma\colon \gamC\to\mathbb{R}^+$, we define the following linear operators \cite{Chouly-Mlika2017}:
\begin{equation}
    \begin{aligned}[t]
        P_{\theta,\gamma}^n \colon \VEC{W} &\to L^2(\gamC) \\
        \VEC{v} &\mapsto \theta\sigma^n(\VEC{v})-\gamma v^n,
    \end{aligned}
    \qquad\qquad\text{and}\qquad\qquad
    \begin{aligned}[t]
        \VEC{P}_{\theta,\gamma}^{\VEC{t}} \colon \VEC{W} &\to \VEC{L}^2(\gamC) \\
        \VEC{v} &\mapsto \theta\VEC{\sigma}^{\VEC{t}}(\VEC{v})-\gamma \VEC{v}^{\VEC{t}},
    \end{aligned}
\end{equation}
where $\VEC{W}\coloneqq\left\{\VEC{v}\in\VEC{H}^1(\Omega)\ :\ \VEC{\sigma}(\VEC{v})\VEC{n}|_{\gamC}\in\VEC{L}^2(\gamC)\right\}$ (notice that $\VEC{V}_h \subset \VEC{W}$).
Assuming that $\VEC{u}\in\VEC{W}$, the two contact conditions \eqref{eq:unilateral contact 1} and \eqref{eq:unilateral contact 2} can be rewritten as follows (see \cite{Curnier1988, Chouly2013, Chouly2022}):
\begin{align}
    \sigma^n(\VEC{u}) =& \left[\sigma^n(\VEC{u})-\gamma u^n\right]_{\mathbb{R}^-} = \left[P_{1,\gamma}^n(\VEC{u})\right]_{\mathbb{R}^-}, \label{eq:normal contact P} \\
    \VEC{\sigma}^{\VEC{t}}(\VEC{u}) =& \left[\VEC{\sigma}^{\VEC{t}}(\VEC{u}) - \gamma\VEC{u}^{\VEC{t}}\right]_{S(\VEC{u})} = \left[\VEC{P}^{\VEC{t}}_{1,\gamma}(\VEC{u})\right]_{S(\VEC{u})}
\end{align}

From now on, we assume that $\gamma$ is the positive piecewise constant function on $\gamC$ which satisfies:
For all $T\in\mathcal{T}_h$ such that $\lvert \partial T\cap\gamC\rvert >0$,
\begin{equation*}
    \gamma|_{\partial T\cap \gamC} = \frac{\gamma_0}{h_T},
\end{equation*}
where $\gamma_0 > 0$ is a fixed \emph{Nitsche parameter}.
Finally, we approximate the problem \eqref{eq:unilateral problem} with the following method \cite{Chouly2014}:
Find $\VEC{u}_h\in\VEC{V}_h$ such that
\begin{equation}\label{eq:Nitsche-based_method}
  a(\VEC{u}_h,\VEC{v}_h)
  - \left(\left[P_{1,\gamma}^n(\VEC{u}_h)\right]_{\mathbb{R}^-}, v_h^n \right)_{\gamC} - \left(\left[\VEC{P}_{1,\gamma}^{\VEC{t}}(\VEC{u}_h)\right]_{S_h(\VEC{u}_h)}, \VEC{v}_h^{\VEC{t}}\right)_{\gamC}
  = L(\VEC{v}_h) \qquad \forall \VEC{v}_h\in\VEC{V}_h,
\end{equation}
where $S_h(\VEC{u}_h)$ depends on the choice of $S(\VEC{u})$: for the Tresca case $S_h(\VEC{u}_h) = s$ while, for the Coulomb case, using again \eqref{eq:normal contact P},  $S_h(\VEC{u}_h) = -\fric\, [P_{1,\gamma}^n(\VEC{u}_h)]_{\MAT{R}^-}$, see \cite{Araya2023}.
For further results concerning the existence and uniqueness of a solution we refer to \cite{Chouly2014, Chouly2022}.
Notice that \eqref{eq:Nitsche-based_method} is the non-symmetric variant of the Nitsche method corresponding to the choice $\theta = 0$ in the above references.

\begin{remark}[Choice of normal contact conditions]
    In this work, we will focus on the unilateral contact problem with no jump on $\gamC$ between $\Omega$ and the rigid foundation, but it is also possible to adapt the present analysis to other cases by replacing the contact condition \eqref{eq:unilateral contact 2}:
    \begin{itemize}
        \item for the bilateral contact problem, we simply replace it with the condition $u^n = 0$;
        \item for the unilateral contact problem with normal gap $g$ on $\gamC$ in the reference configuration, we replace it with the three conditions
        \begin{equation}
            u^n - g\leq 0,\ \sigma^n(\VEC{u})\leq 0,\ \sigma^n(\VEC{u}) (u^n-g) = 0.
        \end{equation}
    \end{itemize}
    In the first case, we do not need to introduce the projection operator $[\,\cdot\,]_{\MAT{R}^-}$ and equation \eqref{eq:normal contact P} becomes $\sigma^n(\VEC{u}) = P_{1,\gamma}^n(\VEC{u})$ while, in the second case, the definition of $P_{\theta,\gamma}^n$ has to be modified by replacing $v^n$ with $v^n-g$.
\end{remark}

\section{A posteriori error analysis}\label{sec:a posteriori analysis}

The goal of this section is to present an a posteriori error estimate based on the notion of equilibrated stress reconstruction.
In this framework, following the approach of \cite{DiPietro2022}, we measure the error associated with the approximate solution $\VEC{u}_h$ using the dual norm of a residual operator.

\subsection{Basic a posteriori error estimate}

Starting from the discrete problem \eqref{eq:Nitsche-based_method} and denoting by $(\HunoD{\Omega})^*$ the dual space of $\HunoD{\Omega}$, for all discrete function $\VEC{w}_h\in\VEC{V}_h$, we define the {\em residual} $\mathcal{R}(\VEC{w}_h)\in (\HunoD{\Omega})^*$ by its action on the space $\HunoD{\Omega}$:
\begin{equation}\label{eq:residual definition}
    \begin{split}
        \left\langle \mathcal{R}(\VEC{w}_h),\VEC{v}\right\rangle &\coloneqq 
        L(\VEC{v})-a(\VEC{w}_h,\VEC{v})+\left(\left[P_{1,\gamma}^n(\VEC{w}_h)\right]_{\mathbb{R}^-},v^n\right)_{\gamC} + \left(\left[\VEC{P}_{1,\gamma}^{\VEC{t}}(\VEC{w}_h)\right]_{S_h(\VEC{w}_h)}, \VEC{v}_h^{\VEC{t}}\right)_{\gamC}
    \end{split}
\end{equation}
for all $\VEC{v}\in \HunoD{\Omega}$. Here, $\langle\,\cdot\,,\,\cdot\,\rangle$ denotes the duality pairing between $\HunoD{\Omega}$ and $(\HunoD{\Omega})^*$.
Let
\begin{equation}\label{eq:triple norm}
    \norm{\VEC{v}}^2 \coloneqq \lVert \VEC{\nabla}\VEC{v}\rVert^2 + \normGamma{\VEC{v}}^2 \qquad \forall \VEC{v}\in \HunoD{\Omega},
\end{equation}
with
\begin{equation}\label{eq:seminorm C}
    \normGamma{\VEC{v}}^2 \coloneqq \sum_{F\in \faces{h}{C}} \frac{1}{h_F} \lVert\VEC{v}\rVert_F^2  \qquad \forall \VEC{v}\in \HunoD{\Omega}.
\end{equation}
Given $\VEC{w}_h \in \VEC{V}_h$, the dual norm of the residual $\mathcal{R}(\VEC{w}_h)$ on the normed space $(\HunoD{\Omega}, \norm{\,\cdot\,})$ is given by
\begin{equation}\label{eq:dual norm definition}
  \dualnormresidual{\VEC{w}_h}
  \coloneqq \sup_{\substack{\VEC{v}\in \HunoD{\Omega},\, \norm{\VEC{v}} = 1}}
  \left\langle \mathcal{R}(\VEC{w}_h),\VEC{v}\right\rangle,
\end{equation}
and the quantity $\dualnormresidual{\VEC{u}_h}$ can be used as a measure of the error committed approximating the exact solution $\VEC{u}$ with $\VEC{u}_h$.

\begin{definition}[Equilibrated stress reconstruction]\label{def:equilibrated stress reconstruction}
    We will call \emph{equilibrated stress reconstruction} any tensor-valued field $\VEC{\sigma}_h:\Omega\mapsto\mathbb{R}^{d\times d}$ such that:
    \begin{enumerate}
        \item $\VEC{\sigma}_h\in \mathbb{H}(\textbf{div},\Omega)$,
        \item $\left(\VEC{\rm div}\, \VEC{\sigma}_h+\VEC{f},\VEC{v}\right)_T=0$ for every $\VEC{v}\in \VEC{\mathcal{P}}^0(T)$ and every $T\in\mathcal{T}_h$,
        \item $(\VEC{\sigma}_h\VEC{n})|_F\in \VEC{L}^2(F)$ for every $F\in \faces{h}{N} \cup \faces{h}{C}$ and $\left(\VEC{\sigma}_h\VEC{n},\VEC{v}\right)_F = \left(\gN,\VEC{v}\right)_F$ for every $\VEC{v}\in \VEC{\mathcal{P}}^0(F)$ and  every $F\in \faces{h}{N}$.
    \end{enumerate}
\end{definition}

Given an equilibrated stress reconstruction $\VEC{\sigma}_h$, for every element $T\in \mathcal{T}_h$, we define the following local error estimators:
\[
\begin{aligned}
    \eta_{\text{osc},T} &\coloneqq \frac{h_T}{\pi} \lVert \VEC{f}+\VEC{\rm div}\, %\VEC{\nabla}\cdot
    \VEC{\sigma}_h \rVert_T,
    &\qquad&\text{(oscillation)}
    \\
    \eta_{\text{str},T} &\coloneqq \lVert \VEC{\sigma}_h-\VEC{\sigma}(\VEC{u}_h) \rVert_T,
    &\qquad&\text{(stress)}
    \\
    \eta_{\text{Neu},T} &\coloneqq \sum_{F\in \faces{T}{N}} C_{t,T,F}  h_F^{\nicefrac{1}{2}} \lVert \gN-\VEC{\sigma}_h \VEC{n} \rVert_F,
    &\qquad&\text{(Neumann)}
    \\
    \eta_{\text{cnt},T} &\coloneqq \sum_{F\in \faces{T}{C}} h_F^{\nicefrac{1}{2}} \left\lVert \left[P_{1,\gamma}^n(\VEC{u}_h)\right]_{\mathbb{R}^-}-\sigma^n_h \right\rVert_F.
    &\qquad&\text{(normal contact)} 
    \\
    \eta_{\text{frc},T} &\coloneqq \sum_{F\in \faces{T}{C}} h_F^{\nicefrac{1}{2}} \left\lVert \left[\VEC{P}_{1,\gamma}^{\VEC{t}}(\VEC{u}_h)\right]_{S_h(\VEC{u}_h)} - \VEC{\sigma}^{\VEC{t}}_h\right\rVert_F,
    &\qquad&\text{(friction)}
\end{aligned}
\]
where, $C_{t,T,F}$ is the constant of the trace inequality $\lVert\VEC{v}-\overline{\VEC{v}}_F\rVert_F \leq C_{t,T,F} h_F^{\nicefrac{1}{2}} \lVert\VEC{\nabla v}\rVert_T$ with $\overline{\VEC{v}}_F\coloneq\frac{1}{|F|}\int_F\VEC{v}$, valid for every $\VEC{v}\in \VEC{H}^1(T)$ and any $F\in\mathcal{F}_T$ (see \cite[Theorem 4.6.3]{Vohralik2015} or \cite[Section 1.4]{Di-Pietro.Ern:12}).

The estimator $\eta_{\text{osc},T}$ represents the residual of the volumetric force balance equation \eqref{eq:unilateral equilibrium} inside the element $T$, $\eta_{\text{str},T}$ the difference between the Cauchy stress tensor computed from $\VEC{u}_h$ and the equilibrated stress reconstruction, $\eta_{\text{Neu},T}$ the residual of the Neumann boundary condition \eqref{eq:unilateral Neumann}, $\eta_{\text{cnt},T}$ the residual of the normal condition \eqref{eq:unilateral contact 1} on the contact boundary, and $\eta_{\text{frc},T}$ the residual of the friction condition.

The following result shows a guaranteed upper bound of the dual norm of the residual \eqref{eq:residual definition} based on these local estimators.

\begin{theorem}[A posteriori error estimate for the dual norm of the residual]\label{th:a posteriori}
    Let $\VEC{u}_h$ be the solution of \eqref{eq:Nitsche-based_method}, $\mathcal{R}(\VEC{u}_h)$ the residual defined by \eqref{eq:residual definition}, and $\VEC{\sigma}_h$ an equilibrated stress reconstruction in the sense of Definition \ref{def:equilibrated stress reconstruction}. Then,
    \begin{equation*}
        \dualnormresidual{\VEC{u}_h}
        \leq \Biggl(\sum_{T\in\mathcal{T}_h} \Bigl((\eta_{\emph{osc},T}+\eta_{\emph{str},T}+\eta_{\emph{Neu},T})^2 + (\eta_{\emph{cnt},T}+\eta_{\emph{frc},T}
        )^2 \Bigr) \Biggr)^{\nicefrac{1}{2}}.
    \end{equation*}
\end{theorem}

\begin{proof}
    Using the regularity of $\VEC{\sigma}_h$ and of its normal trace established by Properties 1. and 3. in Definition \ref{def:equilibrated stress reconstruction}, the following integration by parts formula holds:
    \begin{equation}\label{eq:Green's formula}
      (\VEC{\sigma}_h, \VEC{\nabla v})
      + \left(\VEC{\rm div}\, \VEC{\sigma}_h, \VEC{v}\right)
      - (\VEC{\sigma}_h\VEC{n}, \VEC{v})_{\gamN}
      - (\sigma_h^n, v^n)_{\gamC}
      -(\VEC{\sigma}_h^{\VEC{t}}, \VEC{v^t})_{\gamC}
      = 0
      \qquad \forall \VEC{v}\in \HunoD{\Omega}.
    \end{equation}
    Now, fix $\VEC{v}\in \HunoD{\Omega}$ such that $\norm{\VEC{v}}^2 = \lVert\nabla\VEC{v}\rVert^2 + \normGamma{\VEC{v}}^2 = 1$ and consider the argument of the supremum in the definition \eqref{eq:dual norm definition} of the dual norm of the residual.
    Expanding $L(\cdot)$ and $a(\cdot,\cdot)$ according to \eqref{eq:definition a and L} in the definition \eqref{eq:residual definition} of the residual written for $\VEC{v}_h = \VEC{u}_h$ and summing \eqref{eq:Green's formula} to the resulting expression, we obtain
    \begin{equation*}
    \begin{split}
        \langle \mathcal{R}(\VEC{u}_h), \VEC{v}\rangle
        %% &= (\VEC{f},\VEC{v}) + (\gN,\VEC{v})_{\gamN} - \left(\VEC{\sigma}(\VEC{u}_h),\VEC{\varepsilon} (\VEC{v})\right) + \left(\left[P_{1,\gamma}^n(\VEC{u}_h)\right]_{\mathbb{R}^-}, v^n\right)_{\gamC} \corr{+}{}{[DDP]}
        %% \\
        %% &\quad
        %% + \left(\left[\VEC{P}_{1,\gamma}^{\VEC{t}}(\VEC{u}_h)\right]_{S_h(\VEC{u}_h)}, \VEC{v}^{\VEC{t}}\right)_{\gamC} + (\VEC{\sigma}_h, \VEC{\nabla}\VEC{v}) - (\VEC{\sigma}_h, \VEC{\nabla}\VEC{v})
        %% \\
        &= (\VEC{f} + \VEC{\rm div}\, \VEC{\sigma}_h,\VEC{v})
        + (\VEC{\sigma}_h-\VEC{\sigma}(\VEC{u}_h),\VEC{\nabla}\VEC{v})
        + (\gN-\VEC{\sigma}_h\VEC{n},\VEC{v})_{\gamN} \\
        &\quad
        + \left( \left[P_{1,\gamma}^n(\VEC{u}_h) \right]_{\mathbb{R}^-} - \sigma^n_h,v^n \right)_{\gamC}
        + \left(\left[\VEC{P}_{1,\gamma}^{\VEC{t}}(\VEC{u}_h)\right]_{S_h(\VEC{u}_h)} - \VEC{\sigma}^{\VEC{t}}_h,\VEC{v^t}\right)_{\gamC}
        \eqcolon\mathfrak{T}_1+\cdots+\mathfrak{T}_5,
    \end{split}
    \end{equation*}
    where we have additionally used the symmetry of $\VEC{\sigma}_h$ to replace $\VEC{\varepsilon}(\VEC{v})$ with $\VEC{\nabla}\VEC{v}$ in the second term.
    The first four terms can be treated as in \cite[Theorem 4]{DiPietro2022}, obtaining
    \begin{equation*}
        \begin{aligned}
            \mathfrak{T}_1 \leq \sum_{T\in \mathcal{T}_h} \eta_{\text{osc},T} \lVert \VEC{\nabla v}\rVert_T, 
            \qquad\qquad&\qquad\qquad
            \mathfrak{T}_2 \leq \sum_{T\in\mathcal{T}_h} \eta_{\text{str},T} \lVert \VEC{\nabla v}\rVert_T, \\
            \mathfrak{T}_3 \leq \sum_{T\in\mathcal{T}_h} \eta_{\text{Neu},T} \lVert\VEC{\nabla v}\rVert_T,
            \qquad\qquad&\qquad\qquad
            \mathfrak{T}_4 \leq \sum_{T\in\mathcal{T}_h} \eta_{\text{cnt},T} \normT{\VEC{v}},
        \end{aligned}
    \end{equation*}
    where $\normT{\,\cdot\,}$ is the local counterpart of the seminorm \eqref{eq:seminorm C} on the element $T\in\mathcal{T}_h$ obtained replacing $\faces{h}{C}$ with $\faces{T}{C}$ in the sum.
    %% :
    %% \begin{equation*}
    %%     \normT{\VEC{v}}^2\coloneqq\sum_{F\in \faces{T}{C}} \frac{1}{h_F} \lVert\VEC{v}\rVert_F^2
    %%     \qquad\qquad \forall \VEC{v}\in \HunoD{\Omega}.
    %% \end{equation*}
    Let us consider the fifth term $\mathfrak{T}_5$. Using the Cauchy-Schwarz inequality, we get
    \begin{equation*}
        \begin{aligned}
            \mathfrak{T}_5 &\leq \sum_{T\in \mathcal{T}_h} \sum_{F\in \faces{T}{C}} \left\lVert \left[\VEC{P}_{1,\gamma}^{\VEC{t}}(\VEC{u}_h) \right]_{S_h(\VEC{u}_h)} - \VEC{\sigma}^{\VEC{t}}_h\right\rVert_F \lVert \VEC{v}^{\VEC{t}}\rVert_F  \\
            &\leq \sum_{T\in\mathcal{T}_h} \sum_{F\in \faces{T}{C}} h_F^{\nicefrac{1}{2}} \left\lVert \left[\VEC{P}_{1,\gamma}^{\VEC{t}}(\VEC{u}_h) \right]_{S_h(\VEC{u}_h)} - \VEC{\sigma}^{\VEC{t}}_h\right\rVert_F \normT{\VEC{v}}  = \sum_{T\in\mathcal{T}_h} \eta_{\text{frc},T} \normT{\VEC{v}}.
        \end{aligned}
    \end{equation*}

    Combining the above results and using the Cauchy-Schwarz inequality and the definition of the norm $\norm{\,\cdot\,}$ \eqref{eq:triple norm}, we conclude
    \begin{equation*}
        \begin{aligned}
            \dualnormresidual{\VEC{u}_h}
            &\leq \sup_{\substack{\VEC{v}\in \HunoD{\Omega},\, \norm{\VEC{v}} = 1}} \Biggl\{\sum_{T\in\mathcal{T}_h} \Bigl( \eta_{a,T} \lVert \VEC{\nabla}\VEC{v}\rVert_T + \eta_{b,T}
            \normT{\VEC{v}}\Bigr)\Biggr\}  \\
            &\leq \sup_{\substack{\VEC{v}\in \HunoD{\Omega},\, \norm{\VEC{v}} = 1}} \Biggl\{\Biggl(\sum_{T\in\mathcal{T}_h}\left((\eta_{a,T})^2 + (\eta_{b,T})^2 \right)\Biggr)^{\nicefrac{1}{2}} \Biggl(\sum_{T\in\mathcal{T}_h} \left(\lVert \VEC{\nabla}\VEC{v}\rVert_T^2 + \normT{\VEC{v}}^2 \right) \Biggr)^{\nicefrac{1}{2}} \Biggr\}  \\
            &= \Biggl(\sum_{T\in\mathcal{T}_h}\Bigl((\eta_{\text{osc},T}+\eta_{\text{str},T}+\eta_{\text{Neu},T})^2 + (\eta_{\text{cnt},T}+\eta_{\text{frc},T})^2 \Bigr) \Biggr)^{\nicefrac{1}{2}},
        \end{aligned}
    \end{equation*}
    where, for sake of brevity, $\eta_{a,T}\coloneqq \eta_{\text{osc},T}+\eta_{\text{str},T}+\eta_{\text{Neu},T}$ and $\eta_{b,T} \coloneqq \eta_{\text{cnt},T}+\eta_{\text{frc},T}$ for all $T\in\mathcal{T}_h$.
\end{proof}

\subsection{Separating the error components}\label{subsec:a posteriori distinguishing}

In order to find numerically the approximate solution $\VEC{u}_h$, we apply an iterative method to the nonlinear problem \eqref{eq:Nitsche-based_method}. In particular, in this work we consider the ``generalized Newton method'', also employed in \cite{Chouly-Mlika2017}, where no special treatment is done to account for the fact the projection operators $[\,\cdot\,]_{\MAT{R}^-}$ and $[\,\cdot\,]_{S_h(\VEC{u}_h)}$ are not Gateaux-differentiable.
At each Newton iteration $k \geq 1$, we have to solve the linear problem: 
Find $\VEC{u}_h^k\in\VEC{V}_h$ such that
\begin{equation}\label{eq:linearized problem}
  a(\VEC{u}_h^k,\VEC{v}_h) - \left(P_{\text{lin}}^{n,k-1}(\VEC{u}_h^k),v_h^n\right)_{\gamC} - \left(\VEC{P}_{\text{lin}}^{\VEC{t},k-1}(\VEC{u}_h^k),\VEC{v}_h^{\VEC{t}}\right)_{\gamC} = L(\VEC{v}_h) \qquad \forall\VEC{v}_h\in\VEC{V}_h,
\end{equation}
where the linearized operators are obtained setting
\begin{equation}\label{eq:P lin n}
    \begin{split}
        P_{\text{lin}}^{n,k-1}(\VEC{w}_h) \coloneqq& \left[P_{1,\gamma}^n(\VEC{u}_h^{k-1}) \right]_{\MAT{R}^-} + \frac{\partial \left[P_{1,\gamma}^n (\VEC{v}) \right]_{\MAT{R}^-}}{\partial \VEC{v}} \Bigg|_{\VEC{v}=\VEC{u}_h^{k-1}} \cdot (\VEC{w}_h-\VEC{u}_h^{k-1})  \\
        %=& \left[P_{1,\gamma}^n(\VEC{u}_h^{k-1}) \right]_{\MAT{R}^-} + \frac{\mathrm{d} \left[x\right]_{\MAT{R}^-}}{\mathrm{d}x} \Bigg|_{x=P_{1,\gamma}^n(\VEC{u}_h^{k-1})} \left(P_{1,\gamma}^n(\VEC{w}_h)-P_{1,\gamma}^n(\VEC{u}_h^{k-1})\right) \\
        =& \begin{cases}
            0 &\qquad \text{if}\ P_{1,\gamma}^n(\VEC{u}_h^{k-1}) \leq 0 \\
            P_{1,\gamma}^n(\VEC{w}_h) &\qquad \text{otherwise}
        \end{cases}
    \end{split}
\end{equation}
and
\begin{equation}\label{eq:P lin t}
    \begin{split}
        \VEC{P}_{\text{lin}}^{\VEC{t},k-1}(\VEC{w}_h) \coloneqq& \left[\VEC{P}_{1,\gamma}^{\VEC{t}}(\VEC{u}_h^{k-1}) \right]_{S_h(\VEC{u}_h^{k-1})} + \frac{\partial \left[\VEC{P}_{1,\gamma}^{\VEC{t}} (\VEC{v}) \right]_{S_h(\VEC{u}_h^{k-1})}}{\partial \VEC{v}} \Bigg|_{\VEC{v}=\VEC{u}_h^{k-1}} \cdot (\VEC{w}_h-\VEC{u}_h^{k-1})  \\
        =& \left[\VEC{P}_{1,\gamma}^{\VEC{t}}(\VEC{u}_h^{k-1}) \right]_{S_h(\VEC{u}_h^{k-1})} + \frac{\mathrm{d} \left[\VEC{x}\right]_{S_h(\VEC{u}_h^{k-1})}}{\mathrm{d}\VEC{x}} \Bigg|_{\VEC{x}=\VEC{P}_{1,\gamma}^{\VEC{t}}(\VEC{u}_h^{k-1})} \left(\VEC{P}_{1,\gamma}^{\VEC{t}}(\VEC{w}_h)-\VEC{P}_{1,\gamma}^{\VEC{t}}(\VEC{u}_h^{k-1})\right).
    \end{split}
\end{equation}

\begin{figure}[tb]
	\centering
	\begin{subfigure}{0.49\textwidth}
		\centering
            \includegraphics{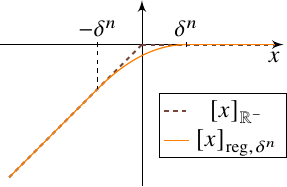}
	\end{subfigure}
	\hfill
	\begin{subfigure}{0.49\textwidth}
		\centering
		\includegraphics{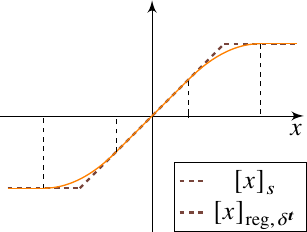}
	\end{subfigure}
    \caption{Regularized operators for $d=2$ and constant Tresca friction.}
    \label{fig:regularization}
\end{figure}

\begin{remark}[Possible regularization of projection operators]
    Another possible iterative approach consists in first regularizing the projection operators $[\,\cdot\,]_{\MAT{R}^-}$ and $[\,\cdot\,]_{S_h(\VEC{u}_h^{k-1})}$ and then applying the standard Newton method. For example, introducing two regularization parameters $\delta^n > 0$ and $\delta^{\VEC{t}} > 0$, we can define the regularized differentiable operators represented in Figure \ref{fig:regularization} for the case $d=2$. In this work, for the sake of simplicity, we only consider the analysis without regularization and refer to \cite{DiPietro2022} for a detailed treatment in the frictionless case.
\end{remark}

\begin{assumption}[Decomposition of the stress reconstruction]\label{assumption stress reconstruction}
  Let $\VEC{\sigma}_h^k$ be an equilibrated stress reconstruction in the sense of Definition~\ref{def:equilibrated stress reconstruction}.
  Then, $\VEC{\sigma}_h^k$ %
  can be decomposed into two parts 
  \begin{equation}\label{eq:sigmah decomposition}
    \VEC{\sigma}_h^k = \VEC{\sigma}_{h,\rm dis}^k + \VEC{\sigma}_{h,\rm lin}^k,
  \end{equation}
  where $\VEC{\sigma}_{h,\rm dis}^k$ represents \emph{discretization} and $\VEC{\sigma}_{h,\rm lin}^k$ represents \emph{linearization}.
\end{assumption}

For an example of reconstruction that satisfies Assumption~\ref{assumption stress reconstruction} we refer to Section~\ref{sec:stress reconstruction}. Now, we introduce the following local estimators that depend on the stress reconstruction and use its decomposition fixed by Assumption~\ref{assumption stress reconstruction}: For any mesh element $T\in\mathcal{T}_h$
\begin{subequations}\label{eq:local estimators distinguishing}
  \begin{align}
    &\eta_{\text{osc},T}^k
    \coloneqq \frac{h_T}{\pi} \left\lVert\VEC{f} + \VEC{\rm div}\, %\VEC{\nabla}\cdot
    \VEC{\sigma}_h^k \right\rVert_T,
    &\qquad& \text{(oscillation)} \label{eq:oscillation estimator uhk}
    \\
    &\eta_{\text{str},T}^k
    \coloneqq \lVert \VEC{\sigma}_{h,\rm dis}^k-\VEC{\sigma}(\VEC{u}_h^k)\rVert_T,
    &\qquad& \text{(stress)} \label{eq:stress estimator uhk}
    \\ 
    &\begin{gathered}
       \eta_{\text{lin1},T}^k \coloneqq \lVert\VEC{\sigma}_{h,\rm lin}^k\rVert_T,
       \quad
       \eta_{\text{lin2n},T}^k \coloneqq \sum_{F\in\faces{T}{C}} h_F^{\nicefrac{1}{2}} \bigl\lVert \sigma_{h,\rm lin}^{k,n} \bigr\rVert_F,
       \\
       \eta_{\text{lin2t},T}^k \coloneqq \sum_{F\in\faces{T}{C}} h_F^{\nicefrac{1}{2}} \bigl\lVert \VEC{\sigma}_{h,\rm lin}^{k,\VEC{t}} \bigr\rVert_F
       \end{gathered}
    &\qquad& \text{(linearization)}
    \\
    &\eta_{\text{Neu},T}^k
    \coloneqq \sum_{F\in\faces{T}{N}} C_{t,T,F} h_F^{\nicefrac{1}{2}} \left\lVert \gN-\VEC{\sigma}_h^k \VEC{n} \right\rVert_F,
    &\qquad& \text{(Neumann)} \label{eq:Neumann estimator uhk}
    \allowdisplaybreaks\\
    &\eta_{\rm cnt,T}^k
    \coloneqq \sum_{F\in\faces{T}{C}} h_F^{\nicefrac{1}{2}} \left\lVert \left[P_{1,\gamma}^n(\VEC{u}_h^k)\right]_{\mathbb{R}^-} - \sigma^{k,n}_{h,\rm dis}\right\rVert_F.
    &\qquad& \text{(contact)} \label{eq:contact estimator uhk}
    \\
    &\eta_{\text{frc},T}^k
    \coloneqq \sum_{F\in \faces{T}{C}} h_F^{\nicefrac{1}{2}} \left\lVert \left[\VEC{P}_{1,\gamma}^{\VEC{t}}(\VEC{u}_h^k)\right]_{S_h(\VEC{u}_h^k)} - \VEC{\sigma}^{k,\VEC{t}}_{h,\rm dis}\right\rVert_F,
    &\qquad&\text{(friction)} \label{eq:friction estimator uhk}
  \end{align}
\end{subequations}
The corresponding global error estimators are defined by
\begin{equation}\label{eq:global estimators distinguishing}
  \eta_{\bullet}^k \coloneqq \left[
    \sum_{T\in\mathcal{T}_h} \left(\eta_{\bullet,T}^k\right)^2
    \right]^{\nicefrac{1}{2}}.
\end{equation}

\begin{theorem}[A posteriori error estimate distinguishing the error components]\label{th:a posteriori distinguishing}
    Let $\VEC{u}_h^k\in\VEC{V}_h$ be the solution of the linearized problem \eqref{eq:linearized problem} with $P_{\rm lin}^{n,k-1}(\,\cdot\,)$ and $\VEC{P}^{\VEC{t},k-1}_{\rm lin}$ defined by \eqref{eq:P lin n} and \eqref{eq:P lin t}, respectively, and let $\mathcal{R}(\VEC{u}_h^k)$ be the residual of $\VEC{u}_h^k$ defined by \eqref{eq:residual definition}.
    Then, under Assumption~\ref{assumption stress reconstruction}, it holds
    \begin{multline}\label{eq:a posteriori local error estimate distinguishing}
        \dualnormresidual{\VEC{u}_h^k}
        \\
        \leq \Biggl[
          \sum_{T\in\mathcal{T}_h} \Bigl((\eta_{\rm osc,T}^k + \eta_{\rm str,T}^k + \eta_{\rm lin1,T}^k + \eta_{\rm Neu,T}^k)^2
          + (\eta_{\rm cnt,T}^k + \eta_{\rm frc,T}^k + \eta_{\rm lin2n,T}^k + \eta_{\rm lin2t,T}^k)^2 \Bigr)
          \Biggr]^{\nicefrac{1}{2}}
    \end{multline}
    and, as a result,
    \begin{equation}\label{eq:a posteriori global error estimate distinguishing}
        \dualnormresidual{\VEC{u}_h^k}
        \leq \Bigl[
          (\eta_{\rm osc}^k + \eta_{\rm str}^k + + \eta_{\rm lin1}^k + \eta_{\rm Neu}^k)^2 + (\eta_{\rm cnt}^k + \eta_{\rm frc}^k + \eta_{\rm lin2n}^k + \eta_{\rm lin2t}^k)^2
          \Bigr]^{\nicefrac{1}{2}}.
    \end{equation}
\end{theorem}

\begin{proof}
    Proceeding as in the proof of Theorem~\ref{th:a posteriori}, we obtain
    \begin{multline*}
        \dualnormresidual{\VEC{u}_h^k}
        \\
        \leq \Biggl\{
        \sum_{T\in\mathcal{T}_h} \Biggl[
          \bigl(\eta_{\text{osc},T}^k + \bigl\lVert\VEC{\sigma}_h^k - \VEC{\sigma}(\VEC{u}_h^k)\bigr\rVert_T + \eta_{\text{Neu},T}^k \bigr)^2
          + \Biggr(
          \sum_{F\in\faces{T}{C}} h_F^{\nicefrac{1}{2}} \left\lVert \left[P_{1,\gamma}^n(\VEC{u}_h^k)\right]_{\mathbb{R}^-} - \sigma_h^{k,n}\right\rVert_F
          \Biggr)^2
          \Biggr]
        \Biggr\}^{\nicefrac{1}{2}}.
    \end{multline*}
    Then, decomposing $\VEC{\sigma}_h^k$ into its discretization and linearization part according to \eqref{eq:sigmah decomposition}, using the triangle inequality and the definition of the local estimators \eqref{eq:local estimators distinguishing}, we get \eqref{eq:a posteriori local error estimate distinguishing}.
    Finally, \eqref{eq:a posteriori global error estimate distinguishing} is obtained from \eqref{eq:a posteriori local error estimate distinguishing} applying twice the inequality $\sum_{T\in\mathcal{T}_h}\left(\sum_{i=1}^m a_{i,T}\right)^2\le\left(\sum_{i=1}^m a_i\right)^2$ valid for all families of nonnegative real numbers $(a_{i,T})_{1\le i\le m,\, T\in\mathcal{T}_h}$ with $a_i\coloneq\left(\sum_{T\in\mathcal{T}_h}a_{i,T}^2\right)^{\nicefrac12}$ for all $1\le i\le m$.
\end{proof}

We close this section by introducing a fully adaptive algorithm for the refinement of an initial coarse mesh with a stopping criterion that automatically adjusts the number of Newton iterations at each mesh refinement iteration. With this goal, we fix a user-dependent parameter $\gamma_{\rm lin}\in (0,1)$ representing the relative magnitude of the linearization error with respect to the total error and define the linearization estimators
\[
  \text{%
    $\eta_{\rm lin, T}^k \coloneqq \eta_{\rm lin1, T}^k + \sqrt{\bigl(\eta_{\rm lin2n,T}^k\bigr)^2 + \bigl(\eta_{\rm lin2t,T}^k\bigr)^2}$ for all $T \in \mathcal{T}_h$ and
    $\eta_{\rm lin}^k \coloneqq \left[\sum_{T\in\mathcal{T}_h} \left(\eta_{\rm lin,T}^k\right)^2\right]^{\nicefrac{1}{2}}$,
    }
\]
and, for $T \in \mathcal{T}_h$, the total estimator
\begin{equation}\label{eq:local total estimator}
    \eta_{\rm tot,T}^k \coloneqq \left[ (\eta_{\rm osc,T}^k + \eta_{\rm str,T}^k + \eta_{\rm lin1,T}^k + \eta_{\rm Neu,T}^k)^2 + (\eta_{\rm cnt,T}^k + \eta_{\rm frc,T}^k + \eta_{\rm lin2n,T}^k + \eta_{\rm lin2t,T}^k)^2 \right]^{\nicefrac{1}{2}}.
\end{equation}

\begin{algorithm}[H]
    \caption{Adaptive algorithm}\label{algorithm}
    \begin{algorithmic}[1]
        \State {\bf choose} an initial displacement $\VEC{u}_h^0 \in\VEC{V}_h$ and fix $\gamma_{\rm lin} \in (0,1)$
        \Repeat \ \{mesh refinement\}
            \State {\bf set} $k = 0$
            \Repeat \ \{Newton algorithm\}
                \State {\bf set} $k = k+1$
                \State  \textbf{setup} the operators $P_{\text{lin},\delta}^{n, k-1}$ and $P_{\text{lin},\delta}^{\VEC{t}, k-1}$ and the linear system \eqref{eq:linearized problem}
                \State \textbf{compute} $\VEC{u}_h^{k}$, $\VEC{\sigma}_h^{k}$, and the estimators \eqref{eq:local estimators distinguishing}--\eqref{eq:global estimators distinguishing}
            \Until {$\eta_{\rm lin}^{k} \leq \gamma_{\rm lin} \left(\eta_{\rm osc}^{k} + \eta_{\rm str}^{k} + \eta_{\rm Neu}^{k} + \eta_{\rm cnt}^{k} + \eta_{\rm frc}^k\right)$} \label{alg:global stopping criterion}
            \State \textbf{refine} the elements of the mesh where $\eta_{\rm tot, T}^k$ is higher
            %%\State \textbf{update} data 
        \Until {$\eta_{\rm tot, T}^k$ is distributed evenly over the mesh}
    \end{algorithmic}
\end{algorithm}

\begin{remark}[Local stopping criterion]
  In the proposed algorithm, the stopping criterion for the number of Newton iterations is enforced in a global sense by comparing the size of the global linearization estimator with the sum of the other global estimators. It is also possible to introduce instead a local stopping criterion that has to be verified on all elements of the mesh:
    \begin{equation}\label{eq:local stopping criterion}
        \eta_{\rm lin, T}^{k} \leq \gamma_{\rm lin, T} (\eta_{\rm osc, T}^{k} + \eta_{\rm str, T}^{k} + \eta_{\rm Neu, T}^{k} + \eta_{\rm cnt, T}^{k} + \eta_{\rm frc, T}^k) \qquad \forall T \in\mathcal{T}_h,
    \end{equation}
    with $\gamma_{\rm lin,T}\in (0,1)$ for all $T\in\mathcal{T}_h$.
    This criterion will be used in Section \ref{sec:efficiency} to prove the local efficiency of the estimators \eqref{eq:local estimators distinguishing}.
\end{remark}

\subsection{Comparison with the energy norm}

This subsection is devoted to comparing the dual norm of the residual $\dualnormresidual{\VEC{u}_h}$ with the energy norm of the error $\energynorm{\VEC{u}-\VEC{u}_h}$ defined in a standard way as
\begin{equation}\label{eq:energy norm}
    \energynorm{\VEC{v}}^2 \coloneqq a(\VEC{v},\VEC{v}) = \left(\VEC{\sigma}(\VEC{v}), \VEC{\varepsilon}(\VEC{v})\right) \qquad \forall \VEC{v}\in \HunoD{\Omega}.
\end{equation}
In the following theorems, the notation $a\lesssim b$, $a,b\in\MAT{R}$ will stand for $a\leq C b$ where $C > 0$ is a constant independent of the mesh size $h$ and of the Nitsche parameter $\gamma_0$.

\begin{theorem}[Control of the energy norm]\label{th:control energy norm}
    Assume that the solution $\VEC{u}$ of the continuous problem \eqref{eq:unilateral problem} belongs to $\VEC{H}^{\frac{3}{2}+\nu}(\Omega)$ for some $\nu >0$, and let $\VEC{u}_h\in\VEC{V}_h$ be the solution of the discrete problem \eqref{eq:Nitsche-based_method}. 
    Then,
    \begin{equation}\label{eq:control energy norm}
        \begin{split}
            \alpha^{\nicefrac{1}{2}} \energynorm{\VEC{u}-\VEC{u}_h}
            \lesssim&
            \dualnormresidual{\VEC{u}_h}
            + \Biggl(\sum_{F\in \faces{h}{C}} \frac{1}{h_F} \left\lVert \sigma^n(\VEC{u})-\left[P_{1,\gamma}^n(\VEC{u}_h)\right]_{\mathbb{R}^-} \right\rVert_{F}^2 \Biggr)^{\nicefrac{1}{2}}\\
            & + \Biggl(\sum_{F\in \faces{h}{C}} \frac{1}{h_F} \left\lVert \VEC{\sigma}^{\VEC{t}}(\VEC{u})-\left[\VEC{P}_{1,\gamma}^{\VEC{t}}(\VEC{u}_h)\right]_{S_h(\VEC{u}_h)} \right\rVert_{F}^2\Biggr)^{\nicefrac{1}{2}},
        \end{split}
    \end{equation}
    where $\alpha$ is the coercitivity constant of the bilinear form $a$ (cf. \eqref{eq:definition a and L}) such that $\alpha \normHuno{\VEC{v}}^2 \leq \energynorm{\VEC{v}}^2$ for any $\VEC{v}\in\HunoD{\Omega}$.
\end{theorem}

\begin{proof}
  Proceeding as in \cite[Theorem 7]{DiPietro2022}, it is possible to show that
  \begin{equation}\label{eq:energy upper partial}
    \begin{aligned}
      \energynorm{\VEC{u}-\VEC{u}_h}^2
      &= \langle\mathcal{R}(\VEC{u}_h), \VEC{u}-\VEC{v}_h\rangle + \left( \sigma^n(\VEC{u})-\left[P_{1,\gamma}^n(\VEC{u}_h)\right]_{\mathbb{R}^-}, u^n-u^n_h \right)_{\gamC} \\
      &\quad
      + \left( \VEC{\sigma}^{\VEC{t}}(\VEC{u})-\left[\VEC{P}_{1,\gamma}^{\VEC{t}}(\VEC{u}_h)\right]_{S_h(\VEC{u}_h)}, \VEC{u}^{\VEC{t}}-\VEC{u}^{\VEC{t}}_h \right)_{\gamC}
      \eqqcolon \mathfrak{T}_1 + \mathfrak{T}_2 + \mathfrak{T}_3.
    \end{aligned}
  \end{equation}
  For the first two terms we use, respectively, the definition \eqref{eq:dual norm definition} of the dual norm of the residual along with the coercivity of $a$ and the Cauchy--Schwarz inequality followed by the definition \eqref{eq:energy norm} of $\energynorm{\cdot}$ to obtain
  \begin{gather}\label{eq:first term}
    \mathfrak{T}_1 \lesssim \alpha^{-\nicefrac{1}{2}}
    \dualnormresidual{\VEC{u}_h}
    \energynorm{\VEC{u}-\VEC{u}_h} 
    \\ \label{eq:second term}
    \mathfrak{T}_2 \lesssim \alpha^{-\nicefrac{1}{2}}
    \Biggl(
    \sum_{F\in\faces{h}{C}} \frac{1}{h_F} \left\lVert \sigma^n(\VEC{u}) - \left[P_{1,\gamma}^n(\VEC{u}_h)\right]_{\mathbb{R}^-} \right\rVert_{F}^2
    \Biggr)^{\nicefrac{1}{2}} \energynorm{\VEC{u}-\VEC{u}_h}.
  \end{gather}
  For the remaining term, we apply Cauchy--Schwarz and trace inequalities to write:
  \begin{equation}\label{eq:third term}
    \begin{aligned}
      \mathfrak{T}_3
      %%          &=
      %%             \sum_{F\in\faces{h}{C}}\left( \VEC{\sigma}^{\VEC{t}}(\VEC{u}) - \left[\VEC{P}_{1,\gamma}^{\VEC{t}}(\VEC{u}_h)\right]_{S_h(\VEC{u}_h}, \VEC{u}^{\VEC{t}}-\VEC{u}^{\VEC{t}}_h \right)_F
      %%             \\
      &\lesssim\Biggl(
      \sum_{F\in \faces{h}{C}} \frac{1}{h_F} \left\lVert \VEC{\sigma}^{\VEC{t}}(\VEC{u}) - \left[\VEC{P}_{1,\gamma}^{\VEC{t}}(\VEC{u}_h)\right]_{S_h(\VEC{u}_h)} \right\rVert_{F}^2
      \Biggr)^{\nicefrac12}\normHuno{\VEC{u}-\VEC{u}_h}  \\
      &\lesssim \alpha^{-\nicefrac{1}{2}}
      \Biggl(
      \sum_{F\in\faces{h}{C}} \frac{1}{h_F} \left\lVert \VEC{\sigma}^{\VEC{t}}(\VEC{u}) - \left[\VEC{P}_{1,\gamma}^{\VEC{t}}(\VEC{u}_h)\right]_{S_h(\VEC{u}_h)} \right\rVert_{F}^2
      \Biggr)^{\nicefrac{1}{2}} \energynorm{\VEC{u}-\VEC{u}_h}.
    \end{aligned}
  \end{equation}
  We conclude by inserting the estimates \eqref{eq:first term}, \eqref{eq:second term}, and \eqref{eq:third term} into \eqref{eq:energy upper partial}.
\end{proof}

\begin{theorem}[Control of the dual norm of the residual]\label{th:control dual norm}
    Assume that the solution $\VEC{u}$ of the continuous problem \eqref{eq:unilateral problem} belongs to $\VEC{H}^{\frac{3}{2}+\nu}(\Omega)$ for some $\nu>0$, and let $\VEC{u}_h\in\VEC{V}_h$ be the solution of the discrete problem \eqref{eq:Nitsche-based_method}. 
    Then, it holds
    \begin{equation}\label{eq:control dual norm}
        \begin{aligned}
          \dualnormresidual{\VEC{u}_h}
          &\leq (d \lambda + 4\mu)^{\nicefrac{1}{2}} \energynorm{\VEC{u}-\VEC{u}_h} + \Biggl(\sum_{F\in\faces{h}{C}} h_F \left\lVert \sigma^n(\VEC{u})-\left[P_{1,\gamma}^n(\VEC{u}_h)\right]_{\mathbb{R}^-}\right\rVert_F^2\Biggr)^{\nicefrac{1}{2}}
          \\
          &\quad          
          + \Biggl(\sum_{F\in\faces{h}{C}} h_F \left\lVert \VEC{\sigma}^{\VEC{t}}(\VEC{u})-\left[\VEC{P}_{1,\gamma}^{\VEC{t}}(\VEC{u}_h)\right]_{S_h(\VEC{u}_h)}\right\rVert_F^2\Biggr)^{\nicefrac{1}{2}}.
        \end{aligned}
    \end{equation}
\end{theorem}

\begin{proof}
  This result can be proved by using the definition of the residual \eqref{eq:residual definition}, an integration by parts, the symmetry of the stress tensor, and Cauchy-Schwarz inequalities.
  Here, we report only the main steps and refer to \cite[Theorem 10]{DiPietro2022} for the remaining details.
  Expanding $a$ and $L$ according to \eqref{eq:definition a and L} into the definition \eqref{eq:residual definition} of the residual written for $\VEC{w}_h = \VEC{u}_h$, we get
    \begin{equation*}
        \begin{split}
            \langle\mathcal{R}(\VEC{u}_h),\VEC{v}\rangle_{} =&\ (\VEC{\sigma}(\VEC{u}-\VEC{u}_h),\VEC{\varepsilon}(\VEC{v})) - \left(\sigma^n(\VEC{u})-\left[P_{1,\gamma}^n(\VEC{u}_h)\right]_{\mathbb{R}^-},v^n\right)_{\gamC} \\
            &- \left(\VEC{\sigma}^{\VEC{t}}(\VEC{u})-\left[\VEC{P}_{1,\gamma}^{\VEC{t}}(\VEC{u}_h)\right]_{S_h(\VEC{u}_h)}, \VEC{v}^{\VEC{t}}\right)_{\gamC}\\
            \leq&\ \left\lVert\VEC{\sigma}(\VEC{u}-\VEC{u}_h)\right\rVert \left\lVert\VEC{\nabla}\VEC{v}\right\rVert + \sum_{F\in\faces{h}{C}} h_F^{\nicefrac{1}{2}} \left\lVert \sigma^n(\VEC{u})-\left[P_{1,\gamma}^n(\VEC{u}_h)\right]_{\mathbb{R}^-}\right\rVert_F \frac{1}{h_F^{\nicefrac{1}{2}}}\left\lVert \VEC{v}\right\rVert_F  \\
            &+ \sum_{F\in\faces{h}{C}} h_F^{\nicefrac{1}{2}} \left\lVert \VEC{\sigma}^{\VEC{t}}(\VEC{u})-\left[P_{1,\gamma}^n(\VEC{u}_h)\right]_{\mathbb{R}^-}\right\rVert_F \frac{1}{h_F^{\nicefrac{1}{2}}}\left\lVert \VEC{v}\right\rVert_F \\
            \leq&\ \Biggl[
              (d \lambda + 4\mu)^{\nicefrac{1}{2}} \energynorm{\VEC{u}-\VEC{u}_h}
              + \Bigg(\sum_{F\in\faces{h}{C}} h_F\left\lVert \sigma^n(\VEC{u})-\left[P_{1,\gamma}^n(\VEC{u}_h)\right]_{\mathbb{R}^-}\right\rVert_F^2
              \Bigg)^{\nicefrac{1}{2}} \\
              &+ \Bigg(
              \sum_{F\in\faces{h}{C}} h_F\left\lVert \VEC{\sigma}^{\VEC{t}}(\VEC{u})-\left[\VEC{P}_{1,\gamma}^{\VEC{t}}(\VEC{u}_h) \right]_{S_h(\VEC{u}_h)}\right\rVert_F^2
              \bigg)^{\nicefrac{1}{2}} \Biggr] \norm{\VEC{v}},
        \end{split}
    \end{equation*}
    Here, we have invoked the symmetry of $\VEC{\sigma}(\VEC{u} - \VEC{u}_h)$ to replace $\VEC{\varepsilon}(\VEC{v})$ with $\VEC{\nabla}\VEC{u}$ in the first term and then used Cauchy--Schwarz inequalities in the second step, and recalled the definitions \eqref{eq:sigma} of $\VEC{\sigma}$ and \eqref{eq:triple norm} of $\norm{\cdot}$ and used a Cauchy--Schwarz inequality on the sum to conclude.
    We obtain \eqref{eq:control dual norm} applying the definition of dual norm \eqref{eq:dual norm definition}.
\end{proof}

\begin{remark}[Terms depending on $\VEC{u}$]
    The comparison results \eqref{eq:control energy norm} and \eqref{eq:control dual norm} contain two terms that depend on the exact solution $\VEC{u}$. It should be possible to obtain similar bounds not containing these terms proceeding like in some recent work \cite{Capatina2021, Gustafsson2020}.
    In this paper, the results of Theorems~\ref{th:control energy norm} and \ref{th:control dual norm} will be used in Section~\ref{sec:numerical results} for defining the two quantities \eqref{eq:lower bound} and \eqref{eq:upper bound} used as lower and upper bounds of the total estimator, respectively.
\end{remark}

\section{Equilibrated stress reconstruction}\label{sec:stress reconstruction}

This section is devoted to describing the procedure to construct an equilibrated stress reconstruction satisfying the decomposition Assumption~\ref{assumption stress reconstruction}. In particular, $\sigma_h^k$ is obtained working on patches of elements around the mesh vertices using the Arnold--Falk--Winther mixed finite element spaces \cite{Arnold2007}. We adapt the approach of \cite{DiPietro2022} to the frictional contact problem modifying the definition of one of the spaces involved in the stress reconstruction.

\begin{figure}[tb]
    \centering
    \includegraphics{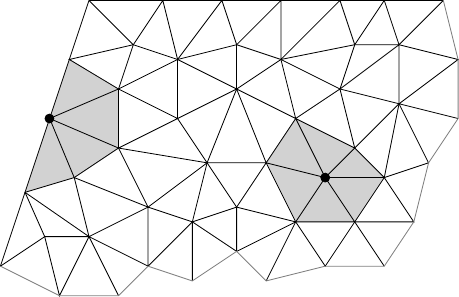}
    \caption{Illustration of a patch $\omega_{\VEC{a}}$ around an inner node $\VEC{a}\in\vertices{h}{i}$ and around a boundary node $\VEC{a}\in\vertices{h}{b}$.}
    \label{fig:patch}
\end{figure}

First, we define, at the local level for any element $T\in\mathcal{T}_h$, the spaces
\begin{equation*}
    \VEC{\Sigma}_T \coloneqq \MAT{P}^{p}(T),\qquad
    \VEC{U}_T \coloneqq \VEC{\mathcal{P}}^{p-1}(T),\qquad
    \VEC{\Lambda}_T \coloneqq \left\{\VEC{\mu}\in\MAT{P}^{p-1}(T) \ :\ \VEC{\mu}=-\VEC{\mu}^T \right\}.
\end{equation*}
The corresponding global spaces are
\begin{equation*}
    \begin{aligned}
        \VEC{\Sigma}_h &\coloneqq \left\{\VEC{\tau}_h \in \MAT{H}(\textbf{div},\Omega) \ :\ \VEC{\tau}_h|_{T} \in \VEC{\Sigma}_T \ \text{for any}\ T\in\mathcal{T}_h\right\},\\
        \VEC{U}_h &\coloneqq \left\{\VEC{v}_h \in \VEC{L}^2(\Omega) \ :\ \VEC{v}_h|_T\in\VEC{U}_T \ \text{for any}\ T\in\mathcal{T}_h \right\},\\
        \VEC{\Lambda}_h &\coloneqq \left\{\VEC{\mu}_h \in \MAT{L}^2(\Omega) \ :\ \VEC{\mu}_h|_T \in \VEC{\Lambda}_T \ \text{for any}\ T\in\mathcal{T}_h \right\}.
    \end{aligned}
\end{equation*}
Then, for any vertex $\VEC{a}\in\mathcal{V}_h$ of the mesh, we consider the patch $\omega_{\VEC{a}}$, see Figure~\ref{fig:patch}, and we denote by $\VEC{n}_{\omega_{\VEC{a}}}$ the outward normal unit vector on its boundary, with $\psi_{\VEC{a}}$ the hat function associated with $\VEC{a}$ and $\VEC{\Sigma}_h(\omega_{\VEC{a}})$, $\VEC{U}_h(\omega_{\VEC{a}})$, and $\VEC{\Lambda}_h(\omega_{\VEC{a}})$ the restrictions of the spaces $\VEC{\Sigma}_h$, $\VEC{U}_h$ and $\VEC{\Lambda}_h$ to the patch $\omega_{\VEC{a}}$.
At the patch level, we set
\begin{align}
  \VEC{\Sigma}_h^{\VEC{a}} &\coloneqq  
  \begin{cases}
    \big\{
    \VEC{\tau}_h \in\VEC{\Sigma}_h(\omega_{\VEC{a}}) \ :\  
    \text{$\VEC{\tau}_h \VEC{n}_{\omega_{\VEC{a}}} = \VEC{0}$ on $\partial\omega_{\VEC{a}}\setminus\gamD$}
    \big\}
    &\text{if $\VEC{a}\in\vertices{h}{b}$},
    \\
    \big\{
    \VEC{\tau}_h \in\VEC{\Sigma}_h(\omega_{\VEC{a}})\ :\
    \text{%
      $\VEC{\tau}_h \VEC{n}_{\omega_{\VEC{a}}} = \VEC{0}$ on $\partial\omega_{\VEC{a}}$
    }
    \big\}
    & \text{otherwise},
  \end{cases}\nonumber
  \\
  \VEC{\Sigma}_{h,{\rm N},{\rm C},\bullet}^{\VEC{a}} &\coloneqq \begin{cases}
    \begin{aligned}
      &\big\{\VEC{\tau}_h\in\VEC{\Sigma}_h(\omega_{\VEC{a}}) \ :\ \VEC{\tau}_h \VEC{n}_{\omega_{\VEC{a}}} = \VEC{0} \ \text{on}\ \partial\omega_{\VEC{a}}\setminus\partial\Omega,\\
      &\quad \VEC{\tau}_h \VEC{n}_{\omega_{\VEC{a}}} = \VEC{g}_{\bullet} \ \text{on}\ \partial\omega_{\VEC{a}}\cap\gamN, \ \text{and}\\
      &\qquad \VEC{\tau}_h \VEC{n}_{\omega_{\VEC{a}}} = \Pi_{\VEC{\Sigma}_h \VEC{n}_{\omega_{\VEC{a}}}} \bigl(\psi_{\VEC{a}} \VEC{P}_{\bullet}(\VEC{u}_h^k)\bigr) \ \text{on}\ \partial\omega_{\VEC{a}}\cap\gamC \big\}
    \end{aligned} &
    \text{if $\VEC{a}\in\vertices{h}{b}$},
    \\
    \VEC{\Sigma}_h^{\VEC{a}} & \text{otherwise}
    \end{cases} \label{eq:definition Sigma h,N,C}
        \\
        \VEC{U}_h^{\VEC{a}} &\coloneqq \begin{cases}
        \VEC{U}_h(\omega_{\VEC{a}}) &
        \text{if $\VEC{a} \in \vertices{h}{D}$},
        \\
        \left\{\VEC{v}_h \in\VEC{U}_h(\omega_{\VEC{a}}) \ :\ (\VEC{v}_h,\VEC{z})_{\omega_{\VEC{a}}} = 0 \ \text{for any}\ \VEC{z}\in\VEC{RM}^d\right\} &
        \text{otherwise},
    \end{cases}\nonumber\\
    \VEC{\Lambda}_h^{\VEC{a}} &\coloneqq \VEC{\Lambda}_h(\omega_{\VEC{a}})\nonumber,
\end{align}
where $\bullet \in\{\rm dis, \rm lin\}$ and
\begin{gather}
    \hspace{-2cm}\VEC{g}_{\rm dis} = \Pi_{\VEC{\Sigma}_h \VEC{n}_{\omega_{\VEC{a}}}} \left(\psi_{\VEC{a}} \gN \right)
    \qquad\text{and}\qquad
    \VEC{g}_{\rm lin} = \VEC{0},\nonumber\\
    \VEC{P}_{\rm dis}(\VEC{u}_h^k) \coloneqq \left[P_{1,\gamma}^n(\VEC{u}_h^k)\right]_{\mathbb{R}^-} \VEC{n} + \left[\VEC{P}_{1,\gamma}^{\VEC{t}}(\VEC{u}_h^k)\right]_{S_h(\VEC{u}_h^k)}, \label{eq:Pdis definition}\\
    \VEC{P}_{\rm lin}(\VEC{u}_h^k) \coloneqq \left(P_{\rm lin}^{n,k-1}(\VEC{u}_h^k) - \left[P_{1,\gamma}^n(\VEC{u}_h^k)\right]_{\mathbb{R}^-}\right) \VEC{n} + \VEC{P}_{\rm lin}^{\VEC{t},k-1}(\VEC{u}_h^k) - \left[\VEC{P}_{1,\gamma}^{\VEC{t}}(\VEC{u}_h^k)\right]_{S_h(\VEC{u}_h^k)}, \label{eq:Plin definition}
\end{gather}
and $\VEC{RM}^d$ is the space of rigid-body motions, i.e., $\VEC{RM}^2 \coloneqq \left\{\VEC{b}+c(x_2,-x_1)^{\top} \ :\ \VEC{b}\in\mathbb{R}^2, c\in\mathbb{R}\right\}$ and $\VEC{RM}^3 \coloneqq \left\{\VEC{b}+\VEC{c}\times\VEC{x} \ :\ \VEC{b},\VEC{c}\in\mathbb{R}^3\right\}$.
Additionally, let $\VEC{y}^{\VEC{a},k}\in \VEC{RM}^d$ be defined by
\begin{equation*}
    \begin{aligned}
        (\VEC{y}^{\VEC{a},k}, \VEC{z})_{\omega_{\VEC{a}}}
        &= (-\psi_{\VEC{a}} \VEC{f} + \VEC{\sigma}(\VEC{u}_h^k)\VEC{\nabla}\psi_{\VEC{a}}, \VEC{z})_{\omega_{\VEC{a}}} - \left(\Pi_{\VEC{\Sigma}_h \VEC{n}_{\omega_{\VEC{a}}}} \left(\psi_{\VEC{a}} \gN \right),\VEC{z}\right)_{\partial\omega_{\VEC{a}}\cap \gamN} 
        \\
        &\quad - \Bigl(\Pi_{\VEC{\Sigma}_h \VEC{n}_{\omega_{\VEC{a}}}} \Bigl(\psi_{\VEC{a}} \Bigl(\left[P_{1,\gamma}^n(\VEC{u}_h^k)\right]_{\mathbb{R}^-} \VEC{n} + \left[\VEC{P}_{1,\gamma}^{\VEC{t}}(\VEC{u}_h^k)\right]_{S_h(\VEC{u}_h^k)} \Bigr)\Bigr), \VEC{z}\Bigr)_{\partial\omega_{\VEC{a}}\cap \gamC},
    \end{aligned}
\end{equation*}
for all $\VEC{z}\in\VEC{RM}^d$ if $\VEC{a}\in\vertices{h}{b}$, and $\VEC{y}^{\VEC{a},k} = \VEC{0}$ if $\VEC{a}\in\vertices{h}{i}$.

\begin{remark}[Friction contact condition]
    Comparing this description with that provided in \cite{DiPietro2022}, we have modified the terms $\VEC{P}_{\rm dis}(\VEC{u}_h^k)$ and $\VEC{P}_{\rm lin}(\VEC{u}_h^k)$ defining the local spaces $\VEC{\Sigma}_{h, {\rm N},{\rm C},{\rm dis}}^{\VEC{a}}$ and $\VEC{\Sigma}_{h, {\rm N},{\rm C},{\rm lin}}^{\VEC{a}}$, respectively, to account for the friction contact conditions.
    These modifications will enable us to recover the properties of the stress reconstruction outlined in point 4. of Lemma \ref{lem:reconstruction sigmahk properties}, as will be discussed later.  Additionally, they will facilitate the rewriting of the contact and friction estimator as \eqref{eq:rewrite contact estimator} and \eqref{eq:rewrite friction estimator}, respectively.
    %The friction contact condition appear in the terms $\VEC{P}_{\rm dis}(\VEC{u}_h^k)$ and $\VEC{P}_{\rm lin}(\VEC{u}_h^k)$ defining the local spaces $\VEC{\Sigma}_{h, {\rm N},{\rm C},{\rm dis}}^{\VEC{a}}$ and $\VEC{\Sigma}_{h, {\rm N},{\rm C},{\rm lin}}^{\VEC{a}}$, respectively.
\end{remark}

\begin{construction}[Equilibrated stress reconstruction distinguishing the error components]\label{construction sigmahk}
  Let, for $\bullet \in \{\rm dis, lin\}$ and any vertex $\VEC{a}\in\mathcal{V}_h$, $(\VEC{\sigma}_{h,\bullet}^{\VEC{a},k},\VEC{r}_{h,\bullet}^{\VEC{a},k},\VEC{\lambda}_{h,\bullet}^{\VEC{a},k}) \in \VEC{\Sigma}_{h,{\rm N},{\rm C},\bullet}^{\VEC{a},k} \times \VEC{U}_h^{\VEC{a}} \times \VEC{\Lambda}_h^{\VEC{a}}$ be the solution to the following problem:
  \[
  \begin{alignedat}{4}
    (\VEC{\sigma}_{h,\bullet}^{\VEC{a},k},\VEC{\tau}_h)_{\omega_{\VEC{a}}} + (\VEC{r}_{h,\bullet}^{\VEC{a},k},\VEC{\rm div}\,
    \VEC{\tau}_h)_{\omega_{\VEC{a}}} + (\VEC{\lambda}_{h,\bullet}^{\VEC{a},k},\VEC{\tau}_h)_{\omega_{\VEC{a}}} &= (\VEC{\tau}_{h,\bullet}^{\VEC{a},k},\VEC{\tau}_h)_{\omega_{\VEC{a}}}
    & \qquad & \forall\VEC{\tau}_h\in\VEC{\Sigma}_h^{\VEC{a}},
    \\ 
    (\VEC{\rm div}\,
    \VEC{\sigma}_{h,\bullet}^{\VEC{a},k},\VEC{v}_h)_{\omega_{\VEC{a}}} &= (\VEC{v}_{h,\bullet}^{\VEC{a},k}, \VEC{v}_h)_{\omega_{\VEC{a}}}
    & \qquad & \forall\VEC{v}_h\in\VEC{U}_h^{\VEC{a}},
    \\ 
    (\VEC{\sigma}_{h,\bullet}^{\VEC{a},k},\VEC{\mu}_h)_{\omega_{\VEC{a}}} &= 0
    & \qquad & \forall\VEC{\mu}_h \in \VEC{\Lambda}_h^{\VEC{a}},
  \end{alignedat}
  \]
  where
  \begin{equation*}
    \VEC{\tau}_{h,\bullet}^{\VEC{a},k} \coloneqq \begin{cases}
      \psi_{\VEC{a}} \VEC{\sigma}(\VEC{u}_h^k) & \text{if $\bullet = \rm dis$}, \\
      0 &  \text{if $\bullet = \rm lin$},
    \end{cases}
    \qquad
    \VEC{v}_{h,\bullet}^{\VEC{a},k} \coloneqq \begin{cases}
      -\psi_{\VEC{a}} \VEC{f} + \VEC{\sigma}(\VEC{u}_h^k)\VEC{\nabla}\psi_{\VEC{a}} - \VEC{y}^{\VEC{a},k} &  \text{if $\bullet = \rm dis$},\\
      \VEC{y}^{\VEC{a},k} & \text{if $\bullet = \rm lin$}.
    \end{cases}
  \end{equation*}
  Extending $\VEC{\sigma}_{h,\bullet}^{\VEC{a},k}$ by zero outside the patch $\omega_{\VEC{a}}$, we set $\VEC{\sigma}_{h,\bullet}^k\coloneqq\sum_{\VEC{a}\in\mathcal{V}_h} \VEC{\sigma}_{h,\bullet}^{\VEC{a},k}$, and we define $\VEC{\sigma}_h^k \coloneqq \VEC{\sigma}_{h,\rm dis}^k + \VEC{\sigma}_{h,\rm lin}^k$.
\end{construction}

By definition, $\VEC{y}^{\VEC{a},k}$ ensures that the forcing terms $\VEC{v}_{h,\bullet}^{\VEC{a},k}$ satisfy the following compatibility conditions for $\VEC{a}\in\vertices{h}{b}\setminus \vertices{h}{D}$:
\begin{gather*}
    (\VEC{v}_{h,\rm dis}^{\VEC{a},k}, \VEC{z})_{\omega_{\VEC{a}}} = \left(\Pi_{\VEC{\Sigma}_h \VEC{n}_{\omega_{\VEC{a}}}} \left(\psi_{\VEC{a}} \gN \right),\VEC{z}\right)_{\partial\omega_{\VEC{a}}\cap \gamN} + \Bigl(\Pi_{\VEC{\Sigma}_h \VEC{n}_{\omega_{\VEC{a}}}} \psi_{\VEC{a}} \VEC{P}_{\rm dis}(\VEC{u}_h^k), \VEC{z}\Bigr)_{\partial\omega_{\VEC{a}}\cap \gamC},
    \\
    (\VEC{v}_{h,\rm lin}^{\VEC{a},k}, \VEC{z})_{\omega_{\VEC{a}}} = \left(\Pi_{\VEC{\Sigma}_h \VEC{n}_{\omega_{\VEC{a}}}} \psi_{\VEC{a}} \VEC{P}_{\rm lin}(\VEC{u}_h^k), \VEC{z}\right)_{\partial\omega_{\VEC{a}}\cap\gamC}
\end{gather*}
for any $\VEC{z}\in\VEC{RM}^d$, recalling that $\VEC{P}_{\rm dis}(\VEC{u}_h^k)$ and $\VEC{P}_{\rm lin}(\VEC{u}_h^k)$ are defined by \eqref{eq:Pdis definition} and \eqref{eq:Plin definition}, respectively.
The obtained tensor $\VEC{\sigma}_h^k$ is an equilibrated stress reconstruction in the sense of Definition \ref{def:equilibrated stress reconstruction} as stated by the following lemma.

\begin{lemma}[Properties of $\VEC{\sigma}_h^k$]\label{lem:reconstruction sigmahk properties}
    Let $\VEC{\sigma}_h^k$ be defined by Construction~\ref{construction sigmahk}. Then
    \begin{enumerate}
        \item $\VEC{\sigma}_{h,\rm dis}^k, \VEC{\sigma}_{h,\rm lin}^k, \VEC{\sigma}_h^k\in\MAT{H}(\emph{\textbf{div}},\Omega)$;
        \item For every $T\in\mathcal{T}_h$ and every $\VEC{v}_T\in \VEC{\mathcal{P}}^{p-1}(T)$, $(\VEC{\rm div}\,         \VEC{\sigma}_h^k+ \VEC{f},\VEC{v}_T)_T=0$;
        \item For every $F\in\faces{h}{N}$ and every $\VEC{v}_F\in\VEC{\mathcal{P}}^p(F)$, $(\VEC{\sigma}_h^k\VEC{n},\VEC{v}_F)_F=(\gN,\VEC{v}_F)_F$; 
        \item For every $F\in\faces{h}{C}$ and every $\VEC{v}_F\in\VEC{\mathcal{P}}^p(F)$, 
        \begin{gather*}
            (\sigma_{h,\rm dis}^{k,n},v_F^n)_F = \left(\left[P_{1,\gamma}^n(\VEC{u}_h^k)\right]_{\mathbb{R}^-}, v_F^n\right)_F, 
            \qquad
            (\VEC{\sigma}_{h,\rm dis}^{k,\VEC{t}},\VEC{v}_F^{\VEC{t}})_F = \left(\left[\VEC{P}_{1,\gamma}^{\VEC{t}}(\VEC{u}_h^k)\right]_{S_h(\VEC{u}_h^k)}, \VEC{v}_F^{\VEC{t}}\right)_F, \\
            (\sigma_{h,\rm lin}^{k,n}, v_F^n)_F = \left(P_{\rm lin}^{n, k-1}(\VEC{u}_h^k) - \left[P_{1,\gamma}^n(\VEC{u}_h^k)\right]_{\MAT{R}^-}, v_F^n\right)_F,
        \end{gather*}
        and 
        \begin{equation*}
            (\VEC{\sigma}_{h,\rm lin}^{k,\VEC{t}},\VEC{v}_F^{\VEC{t}})_F = \left(P_{\rm lin}^{k-1,\VEC{t}}(\VEC{u}_h^k) - \left[\VEC{P}_{1,\gamma}^{\VEC{t}}(\VEC{u}_h^k)\right]_{S_h(\VEC{u}_h^k)}, \VEC{v}_F^{\VEC{t}}\right)_F.
        \end{equation*}
    \end{enumerate}
\end{lemma}

\begin{proof}
  For the proof of 1.--3., the arguments of \cite[Lemma 16]{DiPietro2022} can be easily adapted here. We focus on the proof of 4.
  Let $F\in\faces{h}{C}$ and let $\VEC{v}_F\in \mathcal{P}^p(F) = (\VEC{\Sigma}_h\VEC{n})|_F$. Then, applying the definition \eqref{eq:definition Sigma h,N,C} of $\VEC{\Sigma}_{h, {\rm N}, {\rm C}, \bullet}^{\VEC{a}}$, we get, for $\bullet\in \{{\rm dis}, {\rm lin}\}$,
  \begin{equation*}
    (\VEC{\sigma}_{h,\bullet}^{k} \VEC{n}, \VEC{v}_F)_F = \sum_{\VEC{a} \in\mathcal{V}_F} (\VEC{\sigma}_{h,\bullet}^{\VEC{a}, k} \VEC{n}, \VEC{v}_F)_F = \sum_{\VEC{a}\in\mathcal{V}_F} \left(\psi_{\VEC{a}} \VEC{P}_{\bullet} (\VEC{u}_h^k), \VEC{v}_F\right)_F = \left(\VEC{P}_{\bullet}(\VEC{u}_h^k), \VEC{v}_F\right)_F,
  \end{equation*}
  where we have used the fact that $\sum_{\VEC{a}} \in \mathcal{V}_F \psi_{\VEC{a}}(\VEC{x}) = 1$ for any $\VEC{x} \in F$ to conclude.
  Point 4. follows using the decomposition into normal and tangential components and observing that, by the definitions \eqref{eq:Pdis definition} of $\VEC{P}_{\rm dis}$ and \eqref{eq:Plin definition} of $\VEC{P}_{\rm lin}$,
  \begin{gather*}
    P_{\rm dis}^n(\VEC{u}_h^k) = \left[P_{1,\gamma}^n(\VEC{u}_h^k)\right]_{\mathbb{R}^-},
    \qquad
    \VEC{P}_{\rm dis}^{\VEC{t}}(\VEC{u}_h^k) = \left[\VEC{P}_{1,\gamma}^{\VEC{t}}(\VEC{u}_h^k)\right]_{S_h(\VEC{u}_h^k)},
    \\
    P_{\rm lin}^n(\VEC{u}_h^k) = P_{\rm lin}^{n,k-1}(\VEC{u}_h^k) - \left[P_{1,\gamma}^n(\VEC{u}_h^k)\right]_{\mathbb{R}^-},
    \qquad
    \VEC{P}_{\rm lin}^{\VEC{t}}(\VEC{u}_h^k) = \VEC{P}_{\rm lin}^{\VEC{t},k-1}(\VEC{u}_h^k) - \left[\VEC{P}_{1,\gamma}^{\VEC{t}}(\VEC{u}_h^k)\right]_{S_h(\VEC{u}_h^k)}.
    \qedhere
  \end{gather*}
\end{proof}

\begin{remark}[Alternative expressions of local estimators]
    Thanks to Lemma~\ref{lem:reconstruction sigmahk properties}, we can rewrite the oscillation \eqref{eq:oscillation estimator uhk}, Neumann \eqref{eq:Neumann estimator uhk}, contact \eqref{eq:contact estimator uhk}, and friction \eqref{eq:friction estimator uhk} estimators as follows:
    \begin{subequations}
        \begin{gather}
            \eta_{\rm osc,T}^k = \frac{h_T}{\pi} \left\lVert \VEC{f}-\VEC{\Pi}_T^{p-1} \VEC{f} \right\rVert_T, \label{eq:rewrite oscillation estimator}\\
            \eta_{\rm Neu,T}^k = \sum_{F\in \faces{T}{C}} C_{t,T,F}\, h_F^{\nicefrac{1}{2}} \left\lVert \gN-\VEC{\Pi}_F^p \gN\right\rVert_F, \label{eq:rewrite Neumann estimator} \allowdisplaybreaks\\
            \eta_{\rm cnt,T}^k = \sum_{F\in\faces{T}{C}} h_F^{\nicefrac{1}{2}} \left\lVert \left[P_{1,\gamma}^n(\VEC{u}_h^k)\right]_{\mathbb{R}^-}
            - \Pi_F^p \left[P_{1,\gamma}^n(\VEC{u}_h^k)\right]_{\mathbb{R}^-} \right\rVert_F, \label{eq:rewrite contact estimator}\\
            \eta_{\rm frc,T}^k = \sum_{F\in\faces{T}{C}} h_F^{\nicefrac{1}{2}} \left\lVert \left[\VEC{P}_{1,\gamma}^{\VEC{t}}(\VEC{u}_h^k)\right]_{S_h(\VEC{u}_h^k)}
            - \VEC{\Pi}_F^p \left[\VEC{P}_{1,\gamma}^{\VEC{t}}(\VEC{u}_h^k)\right]_{S_h(\VEC{u}_h^k)} \right\rVert_F, \label{eq:rewrite friction estimator}
        \end{gather}
    \end{subequations}
    where $\VEC{\Pi}_T^{p-1}$, $\VEC{\Pi}_F^p$, and $\Pi_F^p$ denote the $L^2$-orthogonal projectors on the polynomial spaces $\VEC{\mathcal{P}}^{p-1}(T)$, $\VEC{\mathcal{P}}^p(F)$, and $\mathcal{P}^p(F)$, respectively.
    Here, $\VEC{\mathcal{P}}^p(F)$ is either $[\mathcal{P}^p(F)]^d$ or $[\mathcal{P}^p(F)]^{d-1}$ depending on the context.
\end{remark}

\section{Numerical results}\label{sec:numerical results}

In this section, we present a panel of numerical results obtained applying Algorithm~\ref{algorithm}, using the open source finite element library FreeFem++ (see \cite{FreeFEM} and visit \url{https://freefem.org/} for details). 
With this flexible tool, we are able to implement the discrete problem \eqref{eq:Nitsche-based_method} and compute the estimators \eqref{eq:local estimators distinguishing}--\eqref{eq:global estimators distinguishing} in a manner closely resembling their mathematical description. 
Specifically, the command {\tt trunc} has been used to implement the local problems of Construction~\ref{construction sigmahk} in combination with the definition of hat function $\psi_{\VEC{a}}$. This command is also used to obtain sequences of uniformly refined meshes.
Adaptive mesh refinement is, on the other hand, obtained using the {\tt splitmesh}, ensuring that we automatically generate a conformal mesh satisfying the regularity requirements.

\subsection{Tresca friction}\label{sec:numerical results:tresca rectangular}

\begin{figure}[!tb]
    \centering
    \includegraphics{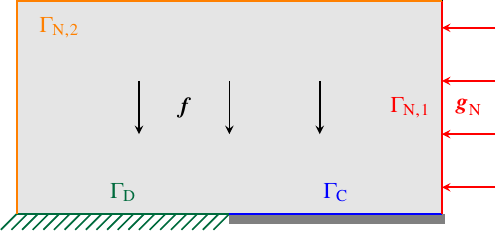}
    \caption{Rectangular domain of the numerical cases of Section~\ref{sec:numerical results:tresca rectangular} and \ref{sec:numerical results:coulomb rectangular domain} with representation of internal and lateral forces, and division of the domain's boundary.
    In particular, a uniform load $\gN$ is enforced on $\Gamma_{{\rm N},1}$, while homogeneous Neumann conditions are enforced on $\Gamma_{{\rm N},2}$.
    The portion of the boundary $\gamD$ is fixed, while contact is possible on $\gamC$.
    }
    \label{fig:domain illustration}
\end{figure}

\begin{figure}[!tb]
  \centering
  \begin{subfigure}{0.45\textwidth}
    \centering
    \includegraphics[width=\textwidth]{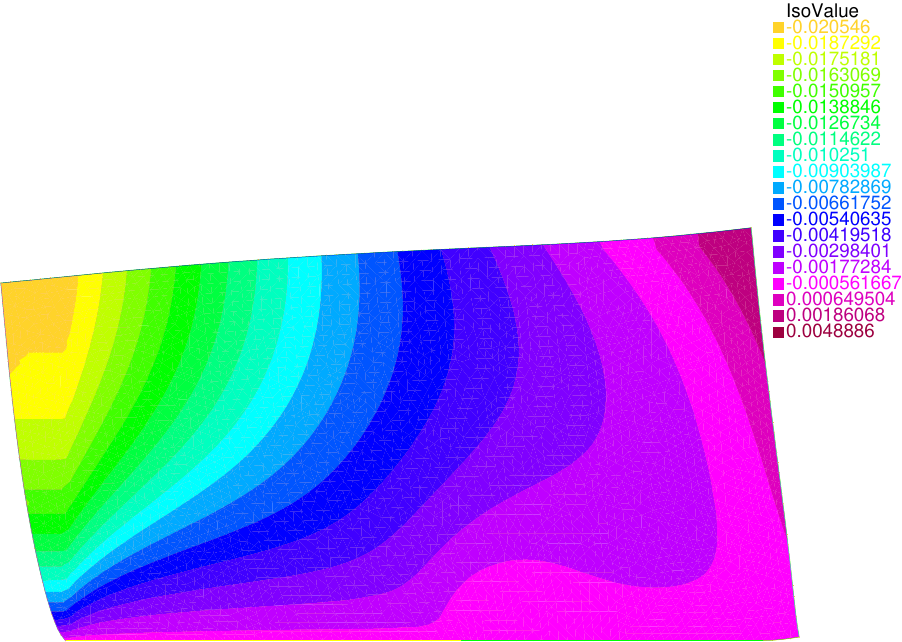}
    \subcaption{Vertical displacement in the deformed domain (amplification factor = 5).}
    \label{fig:Rnorm of displacement}
  \end{subfigure}
  \hfill
  \begin{subfigure}{0.45\textwidth}
    \centering
    \includegraphics[width=\textwidth]{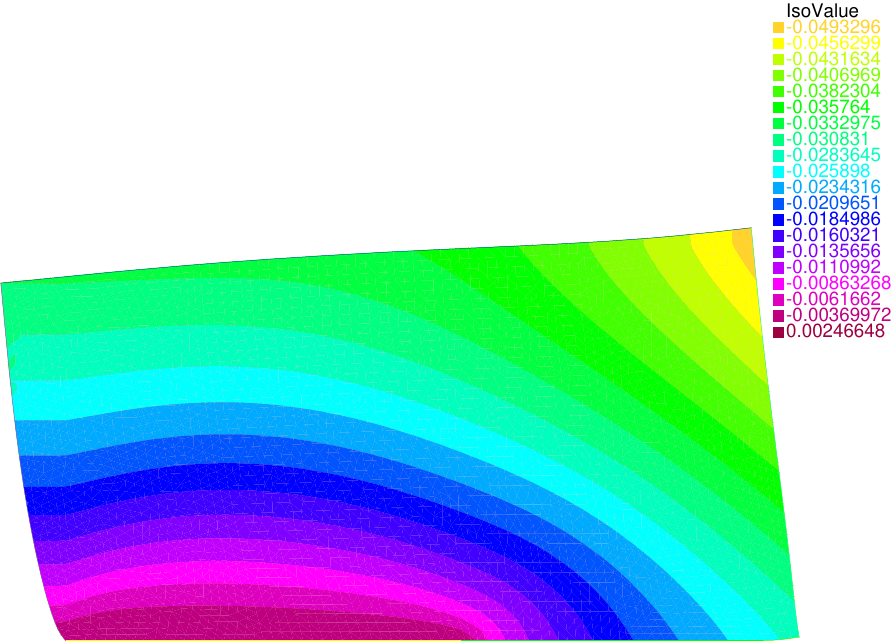}
    \subcaption{Horizontal displacement in the deformed domain (amplification factor = 5).}
    \label{fig:Rhorizontal displacement}
  \end{subfigure}
  \caption{Vertical (\emph{left}) and horizontal displacement (\emph{right}) in the deformed configuration for the Tresca test case of Section~\ref{sec:numerical results:tresca rectangular}.}
  \label{fig:Rdisplacement}
\end{figure}

Let $\Omega = (-1,1)\times (0,1)$ be a rectangular domain with the configuration represented in Figure~\ref{fig:domain illustration}, characterized by the following parameters: Young modulus $E=1$ and Poisson ratio $\nu = 0.3$ (resulting in the Lamé coefficients $\mu\approx 0.385$ and $\lambda\approx 0.577$), a weight force $\VEC{f} = (0,-0.02)$, a horizontal surface loading $\gN = (-0.028,0)$ on $\Gamma_{N,1}$ and $\gN=\VEC{0}$ on $\Gamma_{N,2}$. The Nitsche parameter is set to $\gamma_0 = 10 E$, and Tresca friction conditions are enforced on the contact boundary portion $\gamC$, i.e.,
\begin{equation*}
     \left[\VEC{P}^{\VEC{t}}_{1,\gamma}(\VEC{u}_h)\right]_{S_h(\VEC{u}_h)} = \begin{cases}
         \VEC{P}^{\VEC{t}}_{1,\gamma}(\VEC{u}_h) & \quad {\rm if}\ \bigl\lvert \VEC{P}^{\VEC{t}}_{1,\gamma}(\VEC{u}_h) \bigr\rvert \leq s, \bigskip\\
         s \displaystyle\frac{\VEC{P}^{\VEC{t}}_{1,\gamma}(\VEC{u}_h)}{\bigl\lvert \VEC{P}^{\VEC{t}}_{1,\gamma}(\VEC{u}_h) \bigr\rvert} & \quad \rm otherwise,
     \end{cases}
\end{equation*}
with constant friction function defined as $s = 5\cdot 10^{-3}$ on all $\gamC$.
For this problem, a closed-form solution is not available. Therefore, we adopt as reference solution the solution $\bar{\VEC{u}}_h$ of the discrete problem \eqref{eq:Nitsche-based_method} obtained using $\mathcal{P}^2$ Lagrange finite elements on a fine mesh with mesh size $h\approx 8.34\cdot 10^{-3}$. The approximate solution $\VEC{u}_h$ is obtained using $\mathcal{P}^1$ Lagrange finite elements, and we employ the Newton method outlined in Subsection~\ref{subsec:a posteriori distinguishing} for its computation. We remark that, although $\mathcal{P}^1$ finite elements are known to lock in the quasi-incompressible limit, it is admissible for the set of parameters considered here and it aligns with the use of the lowest-order mixed finite elements available in FreeFem++ for the computation of the equilibrated stress reconstructions described in Section~\ref{sec:stress reconstruction}.
Figure~\ref{fig:Rhorizontal displacement} shows the vertical and horizontal displacement in the deformed configuration with an amplification factor equal to 5. In this configuration, the domain is in contact with the rigid foundation $y=0$ in a non-empty interval $I_C$ which is approximately $(0.035,0.844)$.

\begin{figure}[!tb]
    \centering
    \begin{subfigure}{0.49\textwidth}
        \centering
        \includegraphics[width=0.95\textwidth]{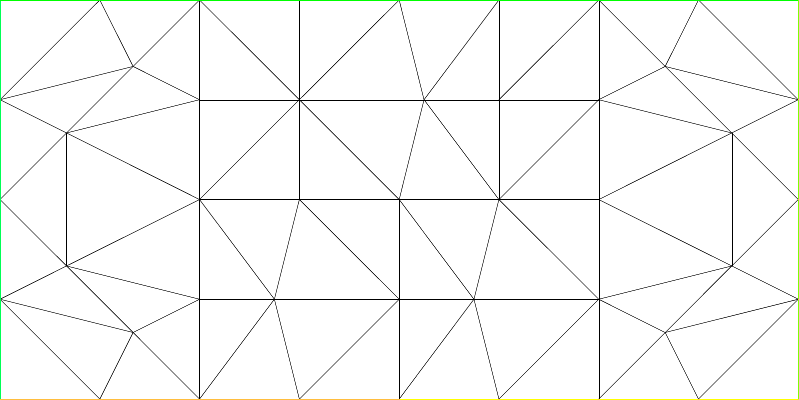}
        \subcaption{Initial mesh}
        \label{fig:initial mesh}
    \end{subfigure}
    \hfill
    \begin{subfigure}{0.49\textwidth}
        \centering
        \includegraphics[width=0.95\textwidth]{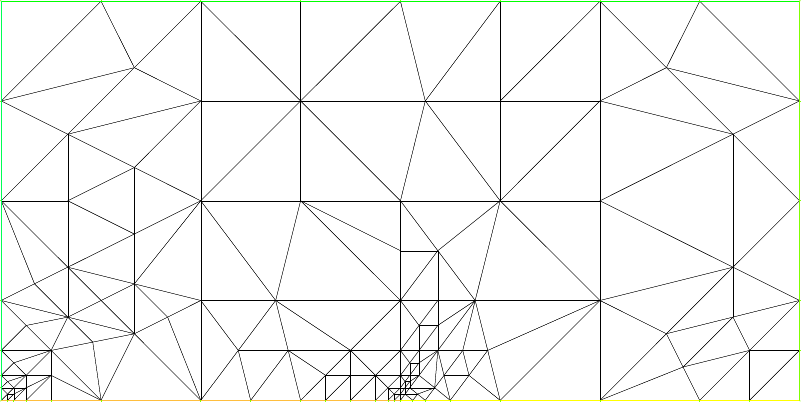}
        \subcaption{$4^{\rm rd}$ mesh refinement iteration}
        \label{fig:3rd spatial iteration}
    \end{subfigure}\vspace{0.3cm}\\
    \begin{subfigure}{0.49\textwidth}
        \centering
        \includegraphics[width=0.95\textwidth]{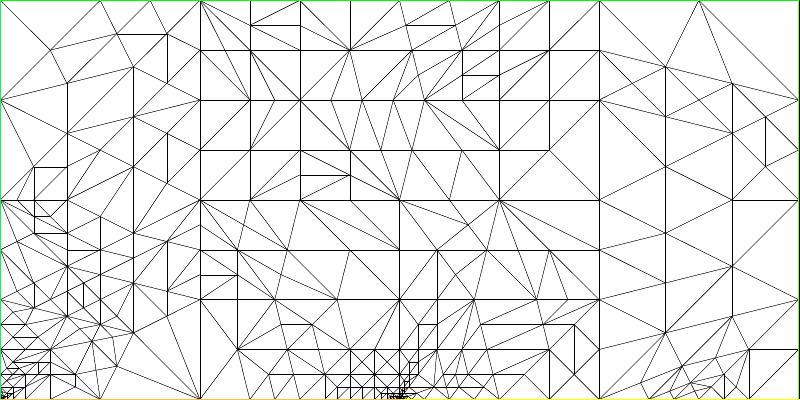}
        \subcaption{$8^{\rm th}$ mesh refinement iteration}
        \label{fig:7th spatial iteration}
    \end{subfigure}
    \hfill
    \begin{subfigure}{0.49\textwidth}
        \centering
        \includegraphics[width=0.95\textwidth]{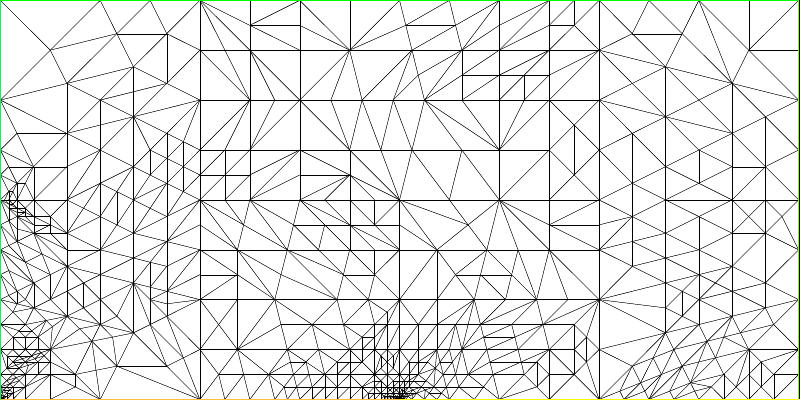}
        \subcaption{$10^{\rm th}$ mesh refinement iteration}
        \label{fig:10th spatial iterations}
    \end{subfigure}
    \caption{Initial mesh and adaptively refined mesh after 4, 8, and 10 remeshing steps, respectively, for the Tresca test case of Section~\ref{sec:numerical results:tresca rectangular}.}
    \label{fig: mesh adaptive refinement}
\end{figure}

\begin{figure}[!tb]
  \centering
  \begin{subfigure}{0.49\textwidth}
    \centering
    \includegraphics[scale=0.73]{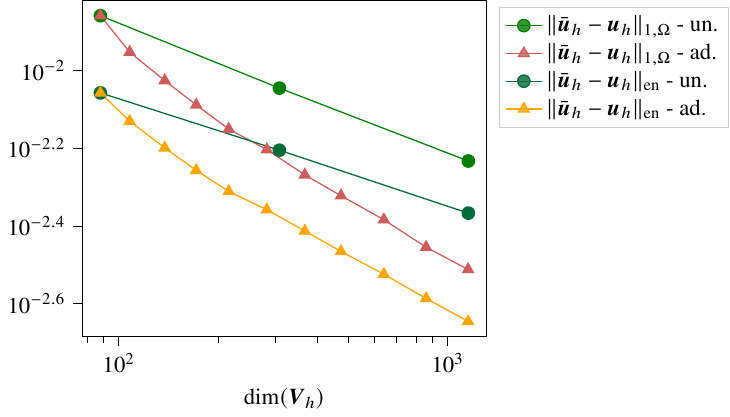}
    \subcaption{$H^1$-norm $\normHuno{\bar{\VEC{u}}_h-\VEC{u}_h}$ and energy norm $\energynorm{\bar{\VEC{u}}_h-\VEC{u}_h}$.}
    \label{fig:H1 norm and energy norm}
  \end{subfigure}
  \vspace{6mm}
  \hfill \\
  \begin{subfigure}{0.49\textwidth}
    \centering
    \includegraphics[scale=0.73]{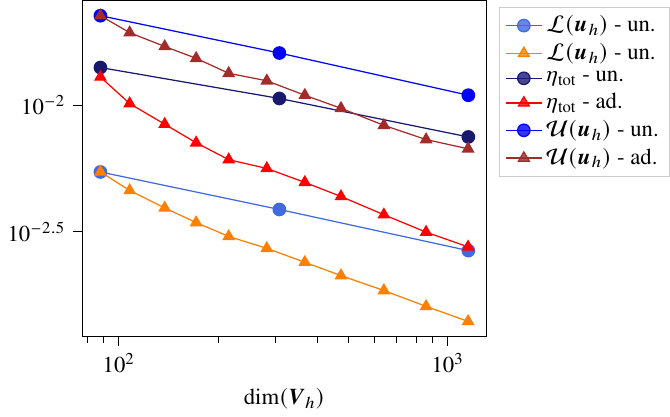}
    \subcaption{Global total estimator $\eta_{\rm tot}$, lower bound $\mathcal{L}(\VEC{u}_h)$, and upper bound $\mathcal{U}(\VEC{u}_h$).}
    \label{fig:lower, total estimator, upper}
  \end{subfigure}
  \hfill
  \begin{subfigure}{0.49\textwidth}
    \centering
    \includegraphics[scale=0.73]{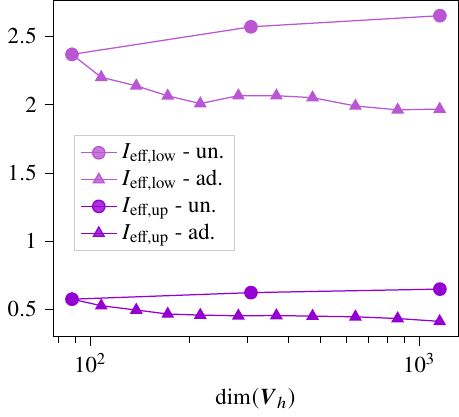}
    \subcaption{Effectivity indices of the lower bound $I_{\rm eff, low}$ and the upper bound $I_{\rm eff, up}$.}
    \label{fig:effectivity index}
  \end{subfigure}
  \caption{Comparison between uniform and adaptive refinement (circles and triangles, respectively) for the Tresca test case of Section~\ref{sec:numerical results:tresca rectangular}.}
\end{figure}

\begin{figure}[!tb]
    \medskip
  \centering
  \begin{subfigure}{0.49\textwidth}
    \centering
    \includegraphics[scale=0.8]{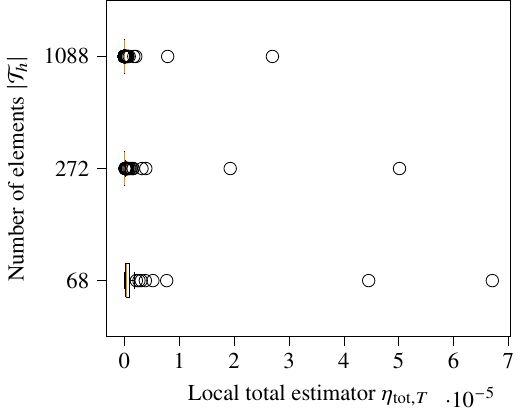}
  \end{subfigure}
  \hfill
  \begin{subfigure}{0.49\textwidth}
    \centering
    \includegraphics[scale=0.8]{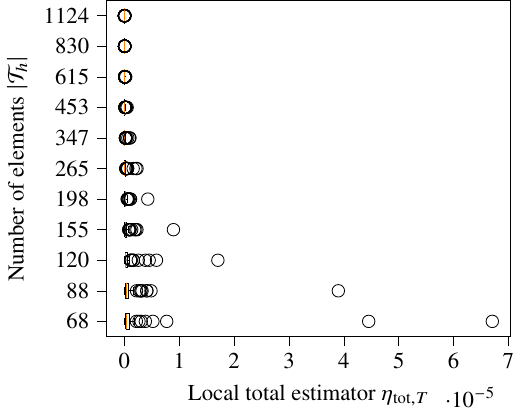}
  \end{subfigure}
  \caption{Distribution of the local total estimator $\eta_{\rm tot,T}$ for each spatial step with uniform ({\it left}) and adaptive ({\it right}) mesh refinement for the Tresca test case of Section~\ref{sec:numerical results:tresca rectangular}.}
  \label{fig:distribution total estimator}
\end{figure}

In our adaptive approach, we base the refinement of the mesh on the distribution of the total local estimator $\eta_{\rm tot,T}$ \eqref{eq:local total estimator}, as stated by Algorithm~\ref{algorithm}. Here, for the sake of simplicity, we omit the superscript $k$.
Starting with the initial coarse mesh in Figure~\ref{fig:initial mesh}, after 4, 8, and 10 steps of adaptive spatial remeshing, we get the meshes shown in Figures~\ref{fig:3rd spatial iteration}, \ref{fig:7th spatial iteration}, and \ref{fig:10th spatial iterations}, respectively. In this example, at least 6.2\% of the elements are refined at each refinement iteration. 
Figure~\ref{fig:10th spatial iterations} shows that most of the refinement is along the contact boundary part $\gamC$ and at the endpoints of $\gamD$, where two singularities arise due to the homogenous Dirichlet boundary conditions.
In Figure~\ref{fig:H1 norm and energy norm}, we compare uniform and adaptive convergence focusing on the $H^1$-norm $\normHuno{\bar{\VEC{u}}-\VEC{u}_h}$ and energy norm $\energynorm{\bar{\VEC{u}}-\VEC{u}_h}$ defined by \eqref{eq:energy norm}.
As expected, the rate of convergence with respect to the number of degrees of freedom $\rm dim(\VEC{V}_h)$ is better using the adaptive approach. Specifically, the asymptotic rates of convergence for the $H^1$-norm and energy norms are approximately 0.328 and 0.282 in the uniform case, and 0.450 and 0.463 in the adaptive one.
In Figure~\ref{fig:lower, total estimator, upper}, we visualize the value of the global total estimator constructed from the definition of the local total estimator \eqref{eq:local total estimator} as
\begin{equation*}
    \eta_{\rm tot} \coloneqq \left(\sum_{T\in\mathcal{T}_h} \left(\eta_{{\rm tot},T}\right)^2\right)^{\nicefrac{1}{2}},
\end{equation*}
and we compare it with the following quantities defined from Theorem~\ref{th:control energy norm} and Theorem~\ref{th:control dual norm}, respectively:
\begin{equation}\label{eq:lower bound}
    \mathcal{L}(\VEC{u}_h) \coloneqq \mu^{\nicefrac{1}{2}} \energynorm{\bar{\VEC{u}}_h-\VEC{u}_h}
\end{equation}
and
\begin{equation}\label{eq:upper bound}
    \begin{split}
        \mathcal{U}(\VEC{u}_h) \coloneqq& (d \lambda + 4\mu)^{\nicefrac{1}{2}} \energynorm{\bar{\VEC{u}}_h-\VEC{u}_h} + \left(\sum_{F\in\faces{h}{C}} h_F \left\lVert \sigma^n(\bar{\VEC{u}}_h)-\left[P_{1,\gamma}^n(\VEC{u}_h)\right]_{\mathbb{R}^-}\right\rVert_F^2\right)^{\nicefrac{1}{2}} \\
        &+ \left(\sum_{F\in\faces{h}{C}} h_F \left\lVert \VEC{\sigma^t}(\bar{\VEC{u}}_h)-\left[\VEC{P}_{1,\gamma}^{\VEC{t}}(\VEC{u}_h)\right]_{S_h(\VEC{u}_h)}\right\rVert_F^2\right)^{\nicefrac{1}{2}}.
    \end{split}
\end{equation}
The corresponding effectivity indices shown by Figure~\ref{fig:effectivity index} are defined in the usual way:
\begin{equation}\label{eq:effectivity indices}
    I_{\text{eff,low}} \coloneqq \frac{\eta_{\text{tot}}}{\mathcal{L}(\VEC{u}_h)} = \frac{\eta_{\rm tot}}{\mu^{\nicefrac{1}{2}} \energynorm{\bar{\VEC{u}}_h-\VEC{u}_h}}
    \qquad \text{and} \qquad
    I_{\text{eff,up}} \coloneqq \frac{\eta_{\text{tot}}}{\mathcal{U}(\VEC{u}_h)}.
\end{equation}
Notice that, for both the uniform and adaptive approaches, at the end of each mesh refinement iteration we get $\mathcal{L}(\VEC{u}_h) < \eta_{\rm tot} < \mathcal{U}(\VEC{u}_h)$ or, equivalently $I_{\text{eff,low}} > 1$ and $I_{\text{eff,up}} < 1$, validating the results \eqref{eq:control energy norm} and \eqref{eq:control dual norm}.
Figure~\ref{fig:distribution total estimator} displays the evolution of the distribution of the local total estimator when refining the mesh with the uniform approach ({\it left}) and with the adaptive one ({\it right}). In particular, with the adaptive approach, the interval containing all the local estimators $\{\eta_{\rm tot,T}\}_{T\in\mathcal{T}_h}$ progressively narrows with each refinement step, and the maximum value decreases significantly faster than with the uniform approach. Consequently, the distribution of the values of $\eta_{\rm tot,T}$ becomes more uniform in the adaptive case.

\begin{table}[!tb]
  \centering
	\begin{tabular}{|c|ccccccccccc|}
		\toprule
		& Initial & 1\textsuperscript{st} & 2\textsuperscript{nd} & 3\textsuperscript{rd} & 4\textsuperscript{th} & 5\textsuperscript{th} & 6\textsuperscript{th} & 7\textsuperscript{th} & 8\textsuperscript{th} & 9\textsuperscript{th} & 10\textsuperscript{th} \\
		\midrule
		$N_{\text{lin}}$ & 3 & 3 & 3 & 3 & 4 & 4 & 4 & 5 & 5 & 5 & 5 \\
		\bottomrule
	\end{tabular}
        \caption{Number of Newton iterations at each refinement step of Algorithm~\ref{algorithm} for the Tresca test case of Section~\ref{sec:numerical results:tresca rectangular}.}
        \label{tab:Newton steps}
\end{table}

\begin{figure}[!tb]
  \centering
  \begin{subfigure}{0.98\textwidth}
    \centering
    \hspace{1cm}
    \includegraphics[width=0.48\textwidth]{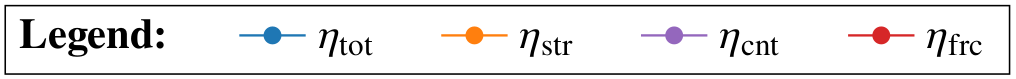}
    \bigskip\\
    \includegraphics[scale=0.8]{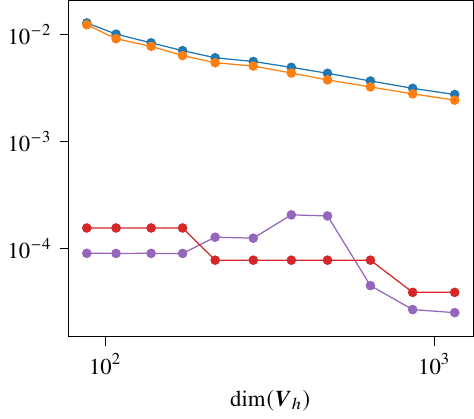}
    \caption{Global estimators}
    \label{fig:global estimators}
  \end{subfigure}
  \vspace{6mm}
  \hfill \\
  \begin{subfigure}{0.49\textwidth}
    \centering
    \includegraphics[scale=0.8]{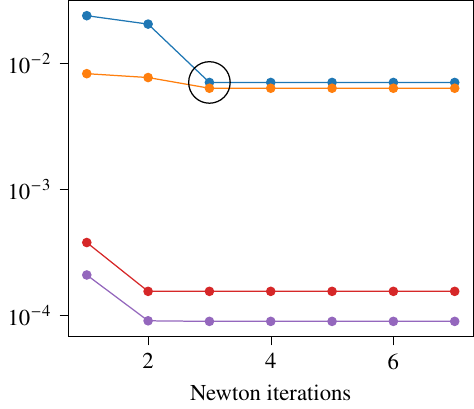}
    \caption{2\textsuperscript{nd} adaptively refined mesh}
    \label{fig:2nd adaptive estimators}
  \end{subfigure}
  \begin{subfigure}{0.49\textwidth}
    \centering
    \includegraphics[scale=0.8]{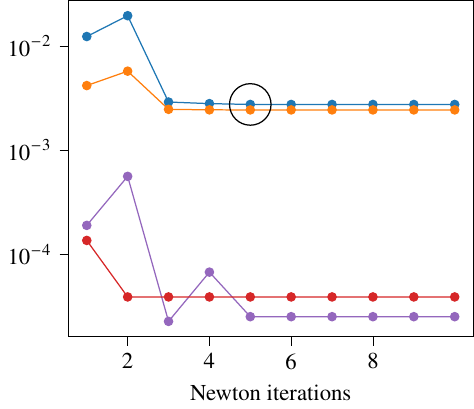}
    \caption{10\textsuperscript{th} adaptively refined mesh}
    \label{fig:9th adaptive estimators}
  \end{subfigure}
  \caption{Evolution of the global estimators $\eta_{\rm tot}$, $\eta_{\rm str}$, $\eta_{\rm cnt}$ and $\eta_{\rm frc}$ using Algorithm \ref{algorithm} with respect to the number of degrees of freedom ({\it top}), and with respect to the number of Newton iterations for the 2\textsuperscript{nd} and 10\textsuperscript{th} adaptively refined mesh for the Tresca test case of Section~\ref{sec:numerical results:tresca rectangular}. The circle indicates the Newton iteration at which the convergence criterion has been reached.
  }
\end{figure}

Finally, Table~\ref{tab:Newton steps} shows the number of Newton iterations \eqref{eq:linearized problem} required to satisfy the stopping criterion of Line~\ref{alg:global stopping criterion} of the fully adaptive Algorithm~\ref{algorithm} with $\gamma_{\rm lin} = 0.01$.
Figure~\ref{fig:global estimators} illustrates the evolution of the global estimators $\eta_{\rm tot}$, $\eta_{\rm str}$, $\eta_{\rm cnt}$ and $\eta_{\rm frc}$ as functions of the number of degrees of freedom. Additionally, the same estimators are represented as functions of the number of Newton iterations by Figure~\ref{fig:2nd adaptive estimators} and \ref{fig:9th adaptive estimators} for the 2\textsuperscript{nd} and 10\textsuperscript{th} adaptively refined meshes, respectively.

\subsection{Coulomb friction}\label{sec:numerical results:coulomb rectangular domain}

\begin{figure}[!tb]
  \centering
  \begin{subfigure}{0.49\textwidth}
    \centering
    \includegraphics[scale=0.8]{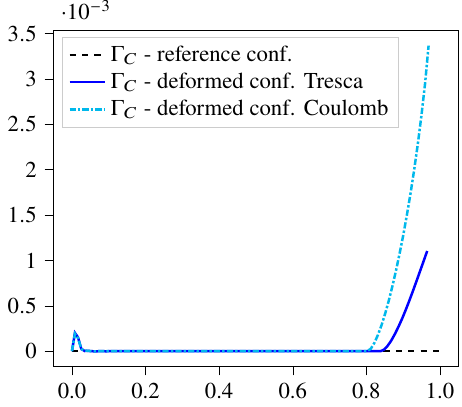}
    \subcaption{Displacement on $\gamC$ in the deformed configuration with Tresca and Coulomb friction.}
    \label{fig:C-displacement}
  \end{subfigure}
  \hfill
  \begin{subfigure}{0.49\textwidth}
    \centering
    \includegraphics[width=0.95\textwidth]{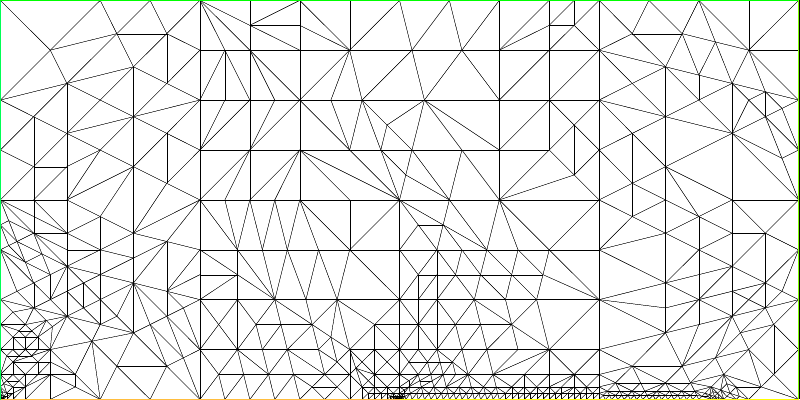}
    \vspace{1cm}
    \caption{Adaptevily refined mesh after 10 remeshing steps.}
    \label{fig:C-10th adaptive mesh}
  \end{subfigure}
  \caption{Contact region and adaptively refined mesh for the Coulomb test case of Section~\ref{sec:numerical results:coulomb rectangular domain}.}
\end{figure}

In this section, we consider again the configuration depicted in Figure~\ref{fig:domain illustration}, using the same parameters as in the previous numerical example but, this time, Coulomb friction conditions are enforced on $\gamC$:

\begin{equation}\label{eq:Coulomb friction conditions}
    \begin{aligned}
        \left[\VEC{P}^{\VEC{t}}_{1,\gamma}(\VEC{u}_h)\right]_{S_h(\VEC{u}_h)} 
        =& \begin{cases}
            \VEC{0} & \quad {\rm if}\ P^n_{1,\gamma}(\VEC{u}_h) > 0, \medskip\\
            \VEC{P}^{\VEC{t}}_{1,\gamma}(\VEC{u}_h) & \quad {\rm if}\ \bigl\lvert \VEC{P}^{\VEC{t}}_{1,\gamma}(\VEC{u}_h)\bigr\rvert \le -\fric\, P^n_{1,\gamma}(\VEC{u}_h), \medskip\\
            -\fric\, P^n_{1,\gamma}(\VEC{u}_h) \displaystyle\frac{\VEC{P}^{\VEC{t}}_{1,\gamma}(\VEC{u}_h)}{\bigl\lvert\VEC{P}^{\VEC{t}}_{1,\gamma}(\VEC{u}_h)\bigr\rvert} & \quad \rm otherwise,
        \end{cases}
    \end{aligned}
\end{equation}
with the friction parameter $\fric = 0.5$.
In Figure~\ref{fig:C-10th adaptive mesh}, we compare the profiles of the contact boundary $\gamC$ in the deformed configuration ({\it light blue}) with the one in the reference configuration ({\it black}) and with the one in the deformed configuration with the Tresca boundary conditions with $s = 5\cdot 10^{-3}$ as in the previous example ({\it blue}). With the selected choice of friction parameters, the opening is more significant in the Coulomb friction case.
Starting from the coarse mesh of Figure~\ref{fig:initial mesh}, we apply the same adaptive approach as in the previous example, and, after 10 remeshing steps, we obtain the mesh of Figure~\ref{fig:C-10th adaptive mesh}. Once again, the refinement concentrates near the endpoints of $\gamD$ and on the contact boundary $\gamC$, particularly along the actual contact interval $I_C \approx (0.035,0.802)$. 

\begin{figure}[!tb]
  \centering
  \begin{subfigure}{0.49\textwidth}
    \centering
    \includegraphics[scale=0.73]{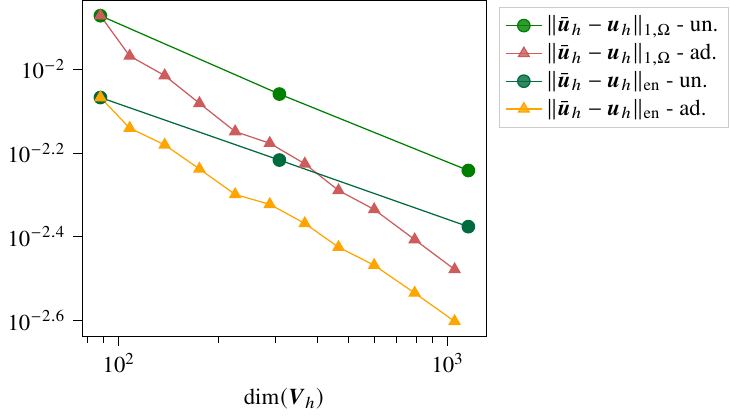}
    \subcaption{$H^1$-norm $\normHuno{\bar{\VEC{u}}-\VEC{u}_h}$ and energy norm $\energynorm{\bar{\VEC{u}}-\VEC{u}_h}$.}
  \end{subfigure}
  \vspace{6mm}
  \hfill \\
  \begin{subfigure}{0.49\textwidth}
    \centering
    \includegraphics[scale=0.73]{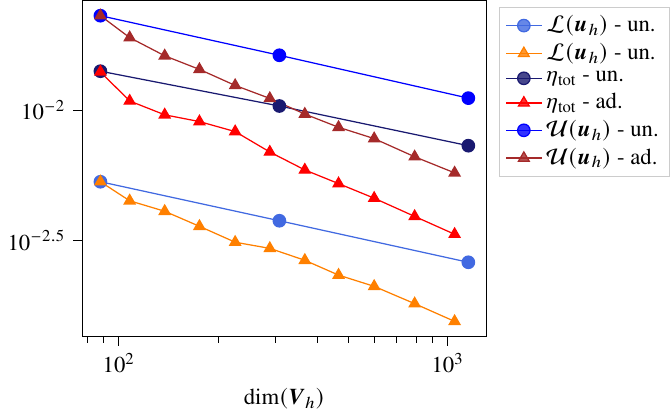}
    \subcaption{Global total estimator $\eta_{\rm tot}$, lower bound $\mathcal{L}(\VEC{u}_h)$, and upper bound $\mathcal{U}(\VEC{u}_h$).}
  \end{subfigure}
  \hfill
  \begin{subfigure}{0.49\textwidth}
    \centering
    \includegraphics[scale=0.73]{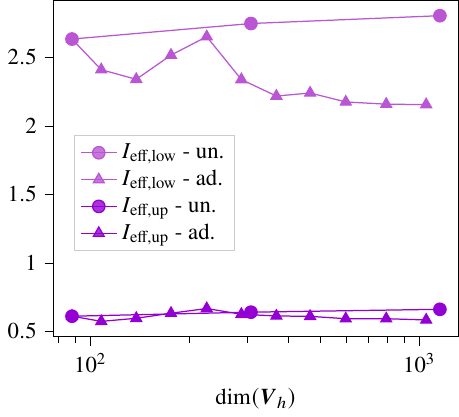}
    \subcaption{Effectivity indices of the lower bound $I_{\rm eff, low}$ and the upper bound $I_{\rm eff, up}$.}
  \end{subfigure}
  \caption{Comparison between uniform and adaptive refinement (circles and triangles, respectively) for the Coulomb test case of Section~\ref{sec:numerical results:coulomb rectangular domain}.}
  \label{fig:C-uniform vs adaptive}
\end{figure}

\begin{figure}[!tb]
  \centering
  \begin{subfigure}{0.49\textwidth}
    \centering
    \includegraphics[scale=0.8]{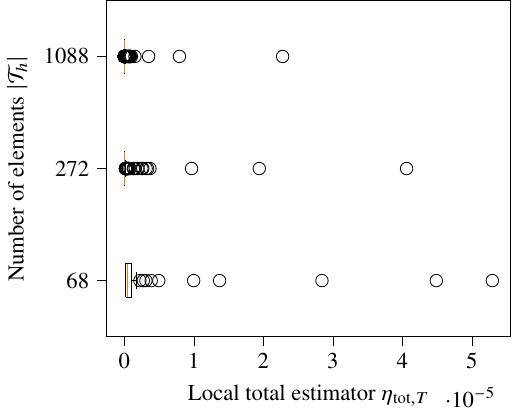}
  \end{subfigure}
  \hfill
  \begin{subfigure}{0.49\textwidth}
    \centering
    \includegraphics[scale=0.8]{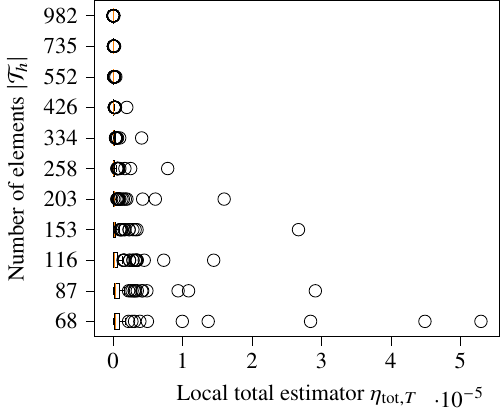}
  \end{subfigure}
  \caption{Distribution of the local total estimator $\eta_{\rm tot,T}$ for each spatial step with uniform ({\it left}) and adaptive ({\it right}) refinement approach for the Coulomb test case of Section~\ref{sec:numerical results:coulomb rectangular domain}.}
  \label{fig:C-distribution estimator}
\end{figure}

We then compare the results obtained with uniform and adaptive approaches: Figure~\ref{fig:C-uniform vs adaptive} showcases the convergence of the $H^1$-norm $\normHuno{\bar{\VEC{u}}-\VEC{u}_h}$ and energy norm $\energynorm{\bar{\VEC{u}}-\VEC{u}_h}$, along with the comparison of the global total estimator $\eta_{\rm tot}$ with the bounds $\mathcal{L}(\VEC{u}_h)$ and $\mathcal{U}(\VEC{u}_h)$ defined by \eqref{eq:lower bound} and \eqref{eq:upper bound}, as well as the corresponding effectivity indices \eqref{eq:effectivity indices}. Additionally, Figure~\ref{fig:C-distribution estimator} represents the distribution of the total local estimators $\{\eta_{{\rm tot},T}\}_{T\in\mathcal{T}_h}$. Notably, the results mirror those observed in the previous example. Specifically, in this case the asymptotic rates of convergence of $H^1$-norm and energy norm are approximately 0.317 and 0.277 for the uniform case, 0.496 and 0.513 for the adaptive one.

\subsection{A test case from literature}\label{eq:numerical results:literature}

\begin{figure}[!tb]
    \centering
    \begin{subfigure}{0.49\textwidth}
        \centering
        \includegraphics{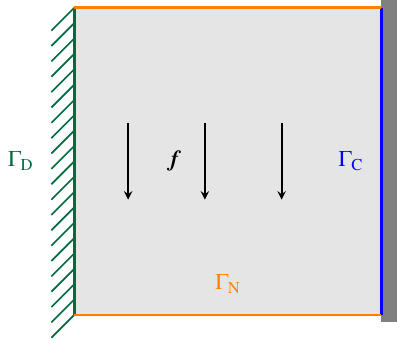}
        \subcaption{Configuration schematic representation.}
        \label{fig:S1 domain}
    \end{subfigure}
    \hfill
    \begin{subfigure}{0.49\textwidth}
        \centering
        \includegraphics[width=0.8\textwidth]{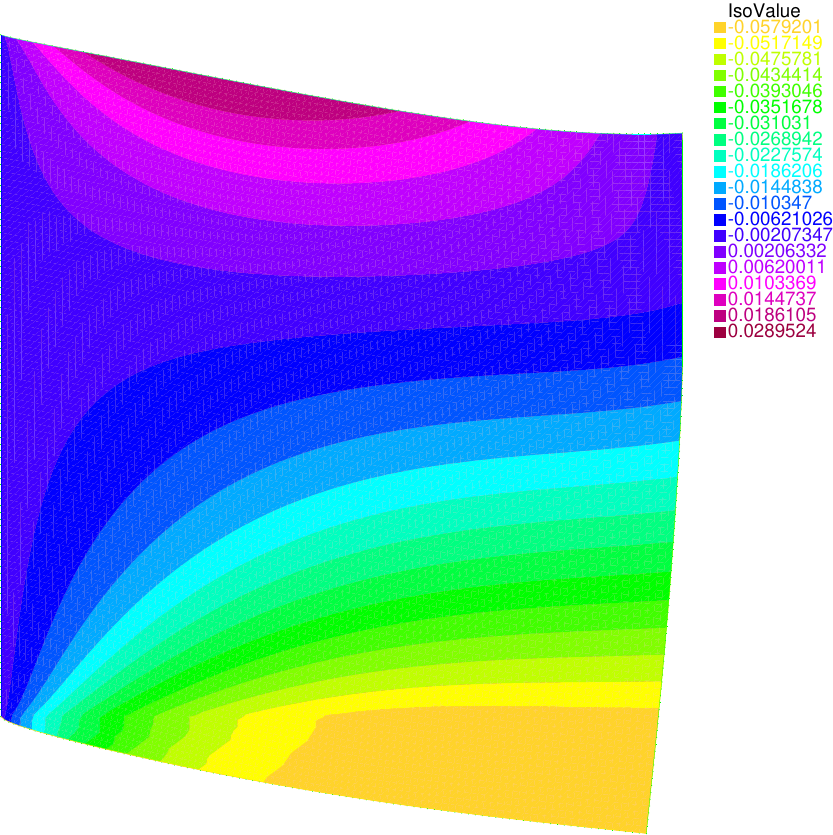}
        \subcaption{Horizontal displacement in the deformed configuration.}
        \label{fig:S1 horizontal displacement}
    \end{subfigure}
    \caption{Square domain for the test case of Section~\ref{eq:numerical results:literature} (configuration from \cite{Hild2009}) with representation of internal forces and division of the domain's boundary.
    In particular, homogeneous Neumann conditions are enforced on $\Gamma_{{\rm N}}$.
    The portion of the boundary $\gamD$ is fixed, while contact is possible on $\gamC$.
    }
    \label{fig:S1-domain+u1}
\end{figure}

We conclude this section by considering the setting of the numerical test investigated in \cite{Hild2009}.
We consider an elastic object represented by the square domain $\Omega = (0,1)^2$, with Young modulus $E = 10^6$ and Poisson ratio $\nu = 0.3$. This domain is subject to a vertical force $\VEC{f} = (0, - 76518)$, it is clamped on $\gamD = \{0\}\times (0,1)$, and no force is applied on $\gamN = (0,1) \times \{0\}\cap\{1\}$. On the contact boundary part $\gamC = \{1\}\times (0,1)$, Coulomb boundary conditions \eqref{eq:Coulomb friction conditions} are enforced with the friction parameter $\mu_{\rm Coul} = 0.2$. Additionally, the Nitsche parameter is set to $\gamma_0 = E$ as in~\cite{Chouly-Mlika2017}. Figure~\ref{fig:S1-domain+u1} shows this setting (\emph{left}) together with the horizontal displacement in the deformed configuration (\emph{right}).
For this configuration, as before, the reference solution $\bar{\VEC{u}}_h$ is computed solving the discrete problem \eqref{eq:Nitsche-based_method} using $\mathcal{P}^2$ Lagrange finite elements on a fine mesh with mesh size $h\approx 8.34\cdot 10^{-3}$, while the approximate solution $\VEC{u}_h$ is obtained using $\mathcal{P}^1$ Lagrange finite elements and the adaptive algorithm described in Subsection~\ref{subsec:a posteriori distinguishing}.
Also in this case, on the contact boundary $\gamC$ we observe both slip and separation.

\begin{figure}[!tb]
    \centering
    \begin{subfigure}{0.32\textwidth}
        \centering
        \includegraphics[width=0.95\textwidth]{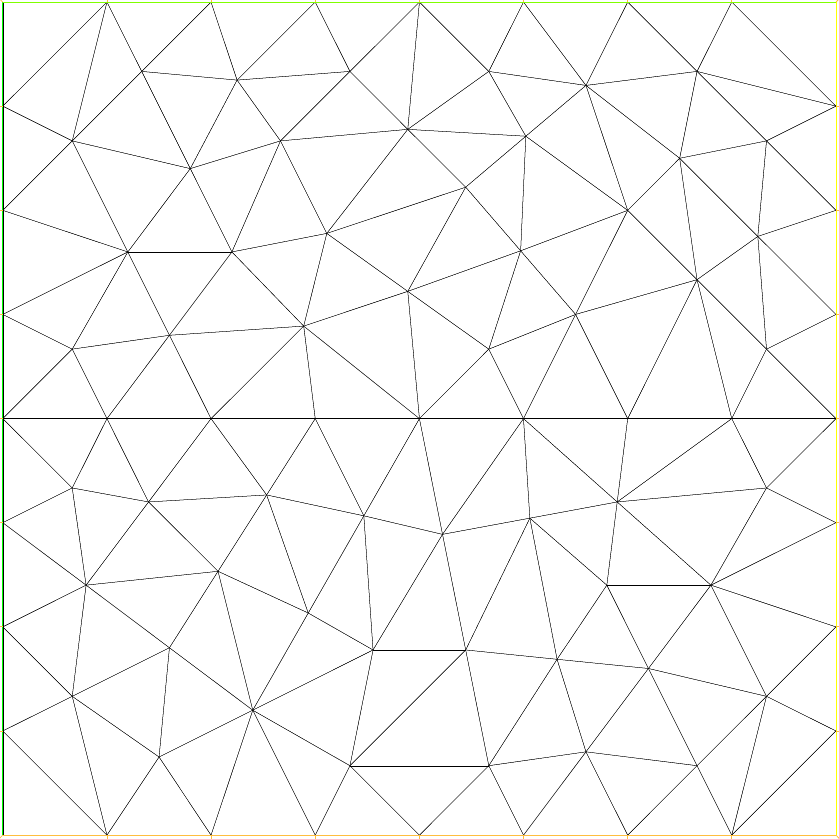}
        \subcaption{Initial mesh}
        \label{fig:S1-initial mesh}
    \end{subfigure}
    \hfill
    \begin{subfigure}{0.32\textwidth}
        \centering
        \includegraphics[width=0.95\textwidth]{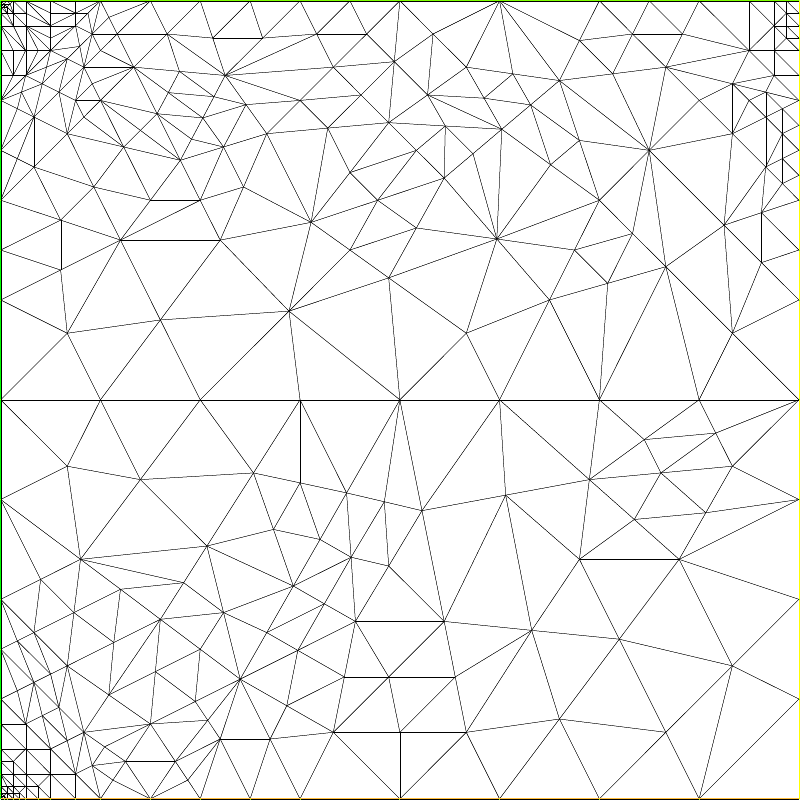}
        \subcaption{$5^{\rm th}$ mesh refinement iteration}
        \label{fig:S1-5th spatial iteration}
    \end{subfigure}
    \hfill
    \begin{subfigure}{0.32\textwidth}
        \centering
        \includegraphics[width=0.95\textwidth]{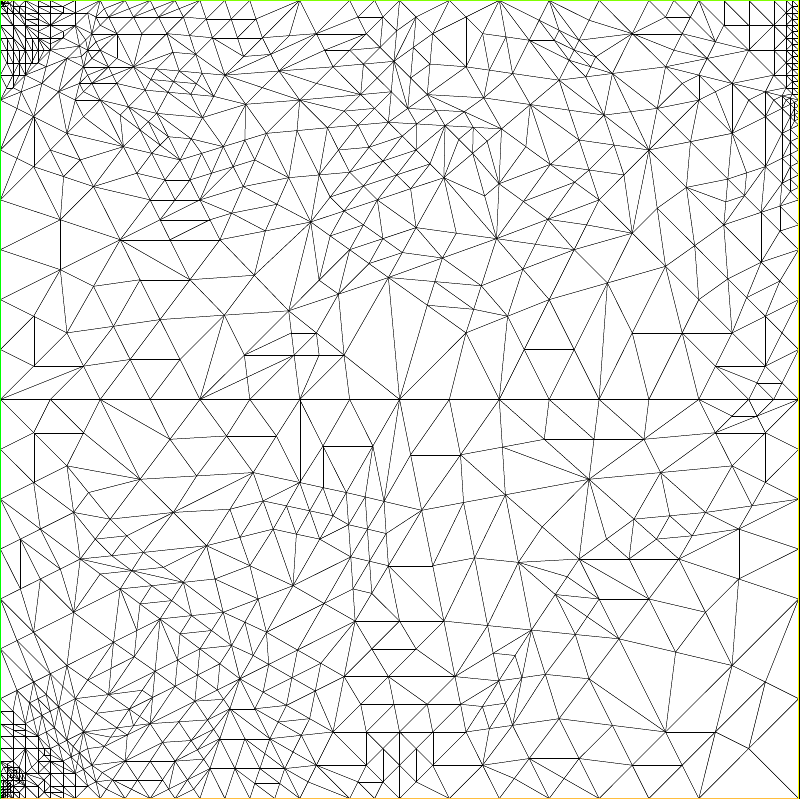}
        \subcaption{$9^{\rm th}$ mesh refinement iteration}
        \label{fig:S1-9th spatial iterations}
    \end{subfigure}
    \caption{Initial mesh and adaptively refined mesh after 5, and 9 remeshing steps, respectively, for the test case of Section \ref{eq:numerical results:literature}.}
    \label{fig:S1-mesh adaptive refinement}
\end{figure}

\begin{figure}[!tb]
	\centering
	\begin{subfigure}{0.49\textwidth}
		\centering
            \includegraphics[scale=0.73]{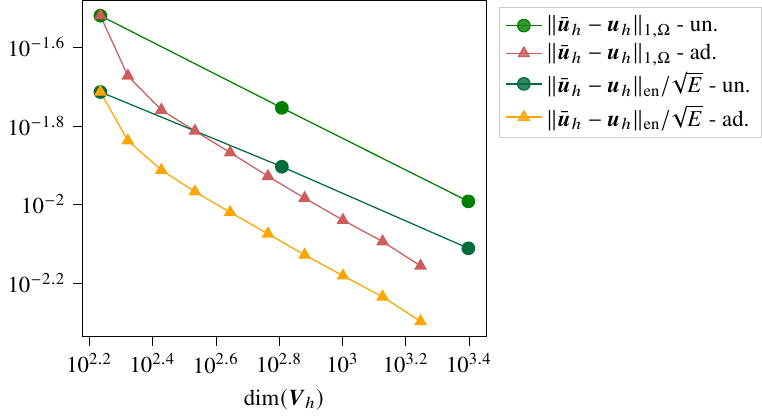}
            \subcaption{$H^1$-norm $\normHuno{\bar{\VEC{u}}-\VEC{u}_h}$ and energy norm $\energynorm{\bar{\VEC{u}}-\VEC{u}_h}/\sqrt{E}$.}
            \label{fig:S1-H1+Energy}
	\end{subfigure}
	\vspace{6mm}
        \hfill \\
	
	\begin{subfigure}{0.49\textwidth}
		\centering
            \includegraphics[scale=0.73]{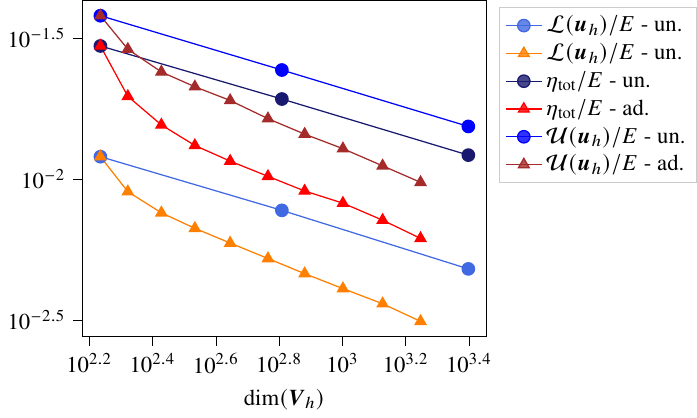}
            \subcaption{Global total estimator $\eta_{\rm tot}$, lower bound $\mathcal{L}(\VEC{u}_h)$, and upper bound $\mathcal{U}(\VEC{u}_h$), divided by $E$.}
            \label{fig:S1-Lower+TE+Upper}
	\end{subfigure}
	\hfill
	\begin{subfigure}{0.47\textwidth}
		\centering
            \includegraphics[scale=0.73]{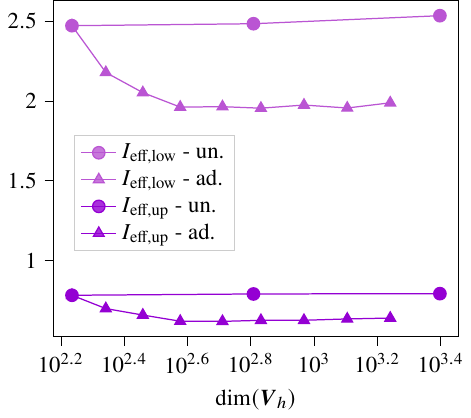}
            \subcaption{Effectivity indices of the lower bound $I_{\rm eff, low}$ and the upper bound $I_{\rm eff, up}$.}
            \label{fig:S1-effectivity}
	\end{subfigure}
        \caption{Comparison between uniform and adaptive refinement (circles and triangles, respectively) for the test case of Section~\ref{eq:numerical results:literature}.}
        \label{fig:S1-uniform vs adaptive}
\end{figure}

We start from the mesh depicted in Figure~\ref{fig:S1-initial mesh} and refine at each spatial iteration at least $6.2\%$ of the elements. After 5 and 9 steps, respectively, we obtain the meshes of Figures~\ref{fig:S1-5th spatial iteration} and \ref{fig:S1-9th spatial iterations}. Notice that the refinement is concentrated around the endpoints of $\gamD$ and the actual contact interval $I_{\rm C}\subset \gamC$. Additionally, Figure~\ref{fig:S1-H1+Energy} shows the evolution of the $H^1$-norm $\normHuno{\bar{\VEC{u}}-\VEC{u}_h}$ and energy norm $\energynorm{\bar{\VEC{u}}-\VEC{u}_h}$, which asymptotic rates are approximatively 0.404 and 0.353 in the uniform case, and 0.516 and 0.522 in the adaptive one. The total estimators $\eta_{\rm tot}$, the lower bound $\mathcal{L}(\VEC{u}_h)$, and the upper bound $\mathcal{U}(\VEC{u}_h)$ are displayed in Figure~\ref{fig:S1-Lower+TE+Upper}, and the corresponding effectivity indices in Figure~\ref{fig:S1-effectivity}. 
Results similar to the ones of Figure~\ref{fig:distribution total estimator} can be obtained for the distribution of the local total estimators.

\section{Efficiency of the estimators}\label{sec:efficiency}

In this section, we prove the efficiency of the estimators introduced in Section~\ref{subsec:a posteriori distinguishing} using the stress reconstruction $\VEC{\sigma}_h$ described in Section~\ref{sec:stress reconstruction}. As before, we will use the notation $a\lesssim b$, $a,b\in\MAT{R}$ when $a\leq C b$ where $C > 0$ is a constant independent of the mesh size $h$ and of the Nitsche parameter $\gamma_0$.
We start presenting the results of local and global efficiency, and then we explain the main idea of their proofs.

To show local efficiency, we work on local patches around elements of the mesh \cite{Verfurth1999}.
For any mesh element $T\in\mathcal{T}_h$, we introduce the patch $\tilde{\omega}_T$ defined as the union of all elements sharing at least one vertex with $T$ and denote by $\mathcal{T}_T$ the corresponding set of elements. Then, we define the {\em local residual operator} $\mathcal{R}_{\mathcal{T}_T}\colon \VEC{V}_h \to (\HunoD{\tilde{\omega}_T})^*$ by:
For all $\VEC{w}_h\in\VEC{V}_h$ and all $\VEC{v}\in\HunoD{\tilde{\omega}_T}$.
\begin{equation*}
  \begin{split}
    \left\langle\mathcal{R}_{\mathcal{T}_T}(\VEC{w}_h), \VEC{v}\right\rangle_{\tilde{\omega}_T} \coloneqq&\ (\VEC{f},\VEC{v})_{\tilde{\omega}_T} + (\gN,\VEC{v})_{\partial\tilde{\omega}_T\cap\gamN} - \bigl(\VEC{\sigma}(\VEC{w}_h),\VEC{\varepsilon}(\VEC{v})\bigr)_{\tilde{\omega}_T}\\
    &+ \left(\left[P_{1,\gamma}^n(\VEC{w}_h)\right]_{\mathbb{R}^-},v^n\right)_{\partial\tilde{\omega}_T\cap \gamC} + \left(\left[\VEC{P}_{1,\gamma}^{\VEC{t}}(\VEC{w}_h)\right]_{S_h(\VEC{w}_h)}, \VEC{v}_h^{\VEC{t}}\right)_{\partial\tilde{\omega}_T\cap \gamC}.
  \end{split}
\end{equation*}
Here, the space $\HunoD{\tilde{\omega}_T}$ is the natural restriction of $\HunoD{\Omega}$ to the patch $\tilde{\omega}_T$, i.e.,
\begin{equation*}
  \HunoD{\tilde{\omega}_T} \coloneqq \left\{\VEC{v}\in \VEC{H}^1(\tilde{\omega}_T)\ :\ \VEC{v}=\VEC{0}\ \text{on}\  \partial\tilde{\omega}_T\cap\gamD \ \text{and on}\ \partial\tilde{\omega}_T\cap\Omega\right\}.
\end{equation*}
Finally, letting
\begin{equation*}
  \norm{\VEC{v}}_{\tilde{\omega}_T}\coloneqq \left(\left\lVert\VEC{\nabla}\VEC{v}\right\rVert_{\tilde{\omega}_T}^2 + \left\lvert\VEC{v}\right\rvert^2_{C,\tilde{\omega}_T} \right)^{\nicefrac{1}{2}} = \biggl(\left\lVert\VEC{\nabla}\VEC{v}\right\rVert_{\tilde{\omega}_T}^2 + \sum_{F\in \faces{\mathcal{T}_T}{C}} \frac{1}{h_F} \lVert\VEC{v}\rVert_F^2 \biggr)^{\nicefrac{1}{2}},
\end{equation*}
with $\faces{\mathcal{T}_T}{C}$ denoting the (possibly empty) set of faces of $\mathcal{T}_T$ that lie on $\gamC$,
we get, for a function $\VEC{w}_h\in\VEC{V}_h$, the local residual norm
\begin{equation}\label{eq:dual norm of the local residual}
  \begin{split}
    \localdualnormresidual{\VEC{w}_h} = \sup_{\substack{\VEC{v}\in \HunoD{\tilde{\omega}_T},\ \norm{\VEC{v}}_{\tilde{\omega}_T} = 1}} \left\langle \mathcal{R}_{\mathcal{T}_T}(\VEC{w}_h), \VEC{v}\right\rangle_{\tilde{\omega}_T} .
  \end{split}
\end{equation}

\begin{theorem}[Local efficiency]\label{th:local efficiency}
    Assume $d=2$.
    Let $\VEC{u}_h^k\in\VEC{V}_h$ be the approximate solution obtained with Algorithm~\ref{algorithm} replacing Line~\ref{alg:global stopping criterion} with the local stopping criterion \eqref{eq:local stopping criterion}, and let $\VEC{\sigma}_h^{k}$ be the stress field resulting from Construction~\ref{construction sigmahk}.
    Then, for every element $T\in\mathcal{T}_h$, it holds
    \begin{equation}\label{eq:thesis local efficiency1}
        \eta_{\rm str, T}^k \lesssim \localdualnormresidual{\VEC{u}_h^k}
        + \eta_{\emph{osc},\mathcal{T}_T}^k + \eta_{\emph{Neu},\mathcal{T}_T}^k 
        + \eta_{\emph{cnt},\mathcal{T}_T}^k + \eta_{\emph{frc},\mathcal{T}_T}^k,
    \end{equation}
    and, as a consequence,
    \begin{multline}\label{eq:thesis local efficiency2}
    \eta_{\emph{osc},T}^k + \eta_{\emph{str},T}^k + \eta_{\emph{Neu},T}^k + \eta_{\emph{cnt},T}^k + \eta_{\emph{frc},T}^k + \eta_{\emph{lin},T}^k  
    \\ 
    \lesssim
    \localdualnormresidual{\VEC{u}_h^k}
    + \eta_{\emph{osc},\mathcal{T}_T}^k + \eta_{\emph{Neu},\mathcal{T}_T}^k 
    + \eta_{\emph{cnt},\mathcal{T}_T}^k + \eta_{\emph{frc},\mathcal{T}_T}^k,
    \end{multline}
    where
    \begin{equation}
      \eta_{\bullet,\mathcal{T}_T}^k \coloneqq \left[
        \sum_{T'\in\mathcal{T}_T} \left(
        \eta_{\bullet,T'}^k
        \right)^2
        \right]^{\nicefrac{1}{2}} \qquad \text{with}\ \bullet\in\{ \emph{osc}, \emph{Neu}, \emph{cnt}\}.
    \end{equation}
\end{theorem}

\begin{theorem}[Global efficiency]\label{th:global efficiency}
    Assume $d=2$.
    Let $\VEC{u}_h^k\in\VEC{V}_h$ the approximate solution obtained with Algorithm~\ref{algorithm}, and let $\VEC{\sigma}_h^{k}$ be the stress field resulting from Construction~\ref{construction sigmahk}.
    Then, it holds
    \begin{equation}\label{eq:thesis global efficiency1}
        \eta_{\rm str}^k \lesssim \dualnormresidual{\VEC{u}_h^k}
        + \eta_{\emph{osc}}^k + \eta_{\emph{Neu}}^k 
        + \eta_{\emph{cnt}}^k + \eta_{\emph{frc}}^k,
    \end{equation}
    and, as a consequence,
    \begin{equation}\label{eq:thesis global efficiency2}
        \eta_{\emph{osc}}^k + \eta_{\emph{str}}^k + \eta_{\emph{Neu}}^k + \eta_{\emph{cnt}}^k + \eta_{\emph{frc}}^k + \eta_{\emph{lin}}^k 
        \lesssim \dualnormresidual{\VEC{u}_h^k}
        + \eta_{\emph{osc}}^k + \eta_{\emph{Neu}}^k 
        + \eta_{\emph{cnt}}^k + \eta_{\emph{frc}}^k.
    \end{equation}
\end{theorem}

\subsection{Proof of the local efficiency}

We illustrate the main steps of the proof of Theorem~\ref{th:local efficiency}. To this end, we will need the notion of {\it bubble function} of an element $T\in\mathcal{T}_h$ and of a face $F\in\mathcal{F}_h$ defined starting from the hat function $\psi_{\VEC{a}}$:
\begin{equation*}
    \psi_T \coloneqq \alpha_T \prod_{\VEC{a}\in\mathcal{V}_T} \psi_{\VEC{a}} \in \mathcal{P}^{d+1}(T), \qquad\qquad\qquad \psi_F \coloneqq \alpha_F \prod_{\VEC{a}\in\mathcal{V}_F} \psi_{\VEC{a}} \in \mathcal{P}^{d}(\omega_F),
\end{equation*}
where the constants $\alpha_T$ and $\alpha_F$ are determined by the conditions $\max_{\VEC{x}\in T} \psi_T(\VEC{x}) = 1$ and $\max_{\VEC{x}\in F} \psi_F(\VEC{x}) = 1$. Additionally, we will need the following four properties of these bubble functions (see \cite[Section 3.1]{Verfurth1996}):
\begin{subequations}\label{eq:properties Verfurth}
    \begin{gather}
        \lVert \VEC{v}\rVert_T^2 \lesssim (\psi_T \VEC{v},\VEC{v})_T \leq \lVert \VEC{v}\rVert_T^2, \label{eq:first property Verfurth}\\
        \lVert \psi_T \VEC{v}\rVert_{1,T} \lesssim \frac{1}{h_T} \lVert \VEC{v}\rVert_T, \label{eq:second property Verfurth}\\
        \lVert \VEC{\varphi}\rVert_F^2 \lesssim (\psi_F \VEC{\varphi}, \VEC{\varphi})_F \leq \lVert \VEC{\varphi}\rVert_F^2, \label{eq:third property Verfurth}\\
        %% \lVert \psi_F \VEC{\varphi}\rVert_{1,\omega_F} \lesssim \frac{1}{h_F^{\nicefrac{1}{2}}} \lVert \VEC{\varphi}\rVert_F, \label{eq:fourth property Verfurth}\\
        %% \lVert \psi_F\VEC{\varphi}\rVert_{\omega_F} \lesssim h_F^{\nicefrac{1}{2}} \lVert \VEC{\varphi}\rVert_F, \label{eq:fifth property Verfurth}
        \lVert \psi_F\VEC{\varphi}\rVert_{\omega_F}
        + h_F \lVert \psi_F \VEC{\varphi}\rVert_{1,\omega_F} \lesssim h_F^{\nicefrac{1}{2}} \lVert \VEC{\varphi}\rVert_F, \label{eq:fourth property Verfurth}.
    \end{gather}
\end{subequations}
where $T\in\mathcal{T}_h$, $F\in\mathcal{F}_h$, $\VEC{v}$, and $\VEC{\varphi}$ are $d$-valued polynomials of degree at most $r$ defined on $T$ and $\omega_F$, respectively. The hidden constants depend only on the polynomial degree $r$ and on the shape regularity parameter of the mesh.

\begin{remark}[Extension of \eqref{eq:first property Verfurth} and \eqref{eq:third property Verfurth}]
  Following the path of \cite{ElAlaoui2011}, it is possible to show that for any $\mathcal{S}\subseteq \mathcal{T}_h$ and any $\VEC{v}\in \VEC{\mathcal{P}}^r(\mathcal{S})$
  \begin{equation}\label{eq:sup-inequality1}
    \left(\sum_{T\in\mathcal{S}} h_T^2 \lVert\VEC{v}\rVert_T^2\right)^{\nicefrac{1}{2}} \lesssim \sup_{\substack{\VEC{w}\in \VEC{\mathcal{P}}^r(\mathcal{S}),\\ \lVert\VEC{w}\rVert_{\omega_{\mathcal{S}}}=1}} \sum_{T\in\mathcal{S}} (\VEC{v}, h_T\psi_T\VEC{w})_T.
  \end{equation}
  where $\omega_{\mathcal{S}}\coloneqq \bigcup_{T\in \mathcal{S}} T$ and $\lVert\VEC{w}\rVert_{\omega_{\mathcal{S}}} \coloneqq \left(\sum_{T\in\mathcal{S}} \lVert\VEC{w}\rVert_{T}^2\right)^{\nicefrac{1}{2}}$.
  In a similar way, for any $\mathcal{E}\subseteq \mathcal{F}_h$ and any $\VEC{\varphi}\in \VEC{\mathcal{P}}^r(\mathcal{E})$
  \begin{equation}\label{eq:sup-inequality2}
    \left(\sum_{F\in\mathcal{E}} h_F\lVert\VEC{\varphi}\rVert_F^2 \right)^{\nicefrac{1}{2}} \lesssim \sup_{\substack{\VEC{\phi}\in \VEC{\mathcal{P}}^r(\mathcal{E}),\\ \lVert\VEC{\phi}\rVert_{\mathcal{E}}=1}} \sum_{F\in\mathcal{E}} (\VEC{\varphi}, h_F^{\nicefrac{1}{2}}\psi_F\VEC{\phi})_F
  \end{equation}
  where $\lVert\VEC{\phi}\rVert_{\mathcal{E}} \coloneqq \left(\sum_{F\in\mathcal{E}} \lVert\VEC{\phi}\rVert_F^2\right)^{\nicefrac{1}{2}}$.
\end{remark}

Following \cite{Verfurth1999}, for any element $T\in\mathcal{T}_h$ we introduce a local {\it residual based} estimator defined on the local patch $\tilde{\omega}_T$:
\begin{equation}\label{eq:definition eta_sharp}
  \begin{aligned}
    \eta_{\sharp,T}^k
    &\coloneqq
    \Bigg(
    \sum_{T'\in\mathcal{T}_T} h_{T'}^2 \left\lVert\VEC{\rm div}\, \VEC{\sigma}(\VEC{u}_h^k) + \VEC{\Pi}_{T'}^p \VEC{f}\right\rVert_{T'}^2
    \Bigg)^{\nicefrac{1}{2}}
    + \Bigg(
    \sum_{F\in\faces{\mathcal{T}_T}{i}} h_F\left\lVert \llbracket\VEC{\sigma}(\VEC{u}_h^k)\VEC{n}_F\rrbracket\right\rVert_F^2
    \Bigg)^{\nicefrac{1}{2}}
    \\
    &\qquad
    + \Bigg(
    \sum_{F\in\faces{\mathcal{T}_T}{N}} h_F\left\lVert \VEC{\sigma}(\VEC{u}_h^k)\VEC{n} - \VEC{\Pi}_F^{p+1}\gN \right\rVert_F^2
    \Bigg)^{\nicefrac{1}{2}}
    \\
    &\qquad\qquad
    + \Bigg(
    \sum_{F\in\faces{\mathcal{T}_T}{C}} h_F \left\lVert\sigma^n(\VEC{u}_h^k) - {\Pi}_F^{p+1} \left[P_{1,\gamma}^n(\VEC{u}_h^k)\right]_{\mathbb{R}^-} \right\rVert_F^2
    \Bigg)^{\nicefrac{1}{2}}
    \\
    &\qquad\qquad\qquad
    + \Bigg(
    \sum_{F\in\faces{\mathcal{T}_T}{C}} h_F \left\lVert\VEC{\sigma^t}(\VEC{u}_h^k) - \VEC{\Pi}_F^{p+1} \left[\VEC{P}_{1,\gamma}^{\VEC{t}}(\VEC{u}_h^k)\right]_{S_h(\VEC{u}_h^k)} \right\rVert_F^2
    \Bigg)^{\nicefrac{1}{2}}.
  \end{aligned}
\end{equation}

\begin{lemma}[Control of the residual-based estimator $\eta_{\sharp,T}^k$]\label{lem:local sharp estimator}
    Let $\VEC{u}_h^k\in\VEC{V}_h$, let $\VEC{\sigma}_h^k$ be the equilibrated stress defined by Construction \ref{construction sigmahk}, and let $\eta_{\sharp,T}^k$ be the local residual-based estimator defined by \eqref{eq:definition eta_sharp}.
    Then, for any element $T\in\mathcal{T}_h$,
    \begin{equation}\label{eq:upper bound sharp estimator}
      \eta_{\sharp,T}^k \lesssim \localdualnormresidual{\VEC{u}_h^k} + \eta_{\emph{osc},\mathcal{T}_T}^k + \eta_{\emph{Neu},\mathcal{T}_T}^k
      + \eta_{\emph{cnt},\mathcal{T}_T}^k
      + \eta_{\emph{frc},\mathcal{T}_T}^k.
    \end{equation}
\end{lemma}

\begin{proof}[Proof of Lemma \ref{lem:local sharp estimator}]
    Let us fix an element $T\in\mathcal{T}_h$. 
    We analyze each term on the right-hand side of \eqref{eq:definition eta_sharp} separately. For simplicity, we denote them with $\mathcal{J}_1$, $\mathcal{J}_2$, $\mathcal{J}_3$, $\mathcal{J}_4$, and $\mathcal{J}_5$, respectively. The key idea is to use the above-mentioned inequalities involving the bubble functions: \eqref{eq:sup-inequality1} with $\mathcal{S} = \mathcal{T}_T$ for $\mathcal{J}_1$ and \eqref{eq:sup-inequality2} with $\mathcal{E} = \faces{\mathcal{T}_T}{i}, \faces{\mathcal{T}_T}{N}, \faces{\mathcal{T}_T}{C}$ for $\mathcal{J}_2$, $\mathcal{J}_3$, $\mathcal{J}_4$, and $\mathcal{J}_5$. 

    Since $(\VEC{\nabla}\cdot\VEC{\sigma}(\VEC{u}_h^k))|_{T'} + \VEC{\Pi}_{T'}^p \VEC{f}|_{T'}\in \VEC{\mathcal{P}}^p(T')$ for every $T'\in\mathcal{T}_T$, applying \eqref{eq:sup-inequality1} we get
    \begin{equation}
        \mathcal{J}_1 \lesssim \sup_{\substack{\VEC{w}\in \VEC{\mathcal{P}}^p(\mathcal{T}_T),\\ \lVert\VEC{w}\rVert_{\tilde{\omega}_T} = 1}} \sum_{T'\in\mathcal{T}_T} (\VEC{\rm div}\,  \VEC{\sigma}(\VEC{u}_h^k) + \VEC{\Pi}_{T'}^p\VEC{f}, h_{T'} \psi_{T'} \VEC{w})_{T'}.
    \end{equation}
    Notice that here we simply write $\lVert\VEC{w}\rVert_{\tilde{\omega}_T}$ instead of $\lVert\VEC{w}\rVert_{\omega_{\mathcal{T}_T}}$.
    Fix $\VEC{w}\in \VEC{\mathcal{P}}^p(\mathcal{T}_T)$ with $\lVert\VEC{w}\rVert_{\tilde{\omega}_T}=1$, and define $\VEC{\lambda}|_{T'} \coloneqq h_{T'}\psi_{T'}\VEC{w}|_{T'}$ for every $T'\in\mathcal{T}_T$.
    Notice that $\VEC{\lambda}\in \VEC{\mathcal{P}}^{p+d+1}(\mathcal{T}_T) \cap \VEC{H}^1_D(\tilde{\omega}_T)$. Then, using an integration by parts on each element $T'\in\mathcal{T}_T$, the definition of the residual, and the Cauchy--Schwarz inequality we obtain
    \begin{equation}
        \begin{split}
          &\sum_{T'\in\mathcal{T}_T} (\VEC{\rm div}\, \VEC{\sigma}(\VEC{u}_h^k) + \VEC{\Pi}_{T'}^p\VEC{f}, h_{T'}\psi_{T'}\VEC{w})_{T'}
          \\
          &\quad \lesssim\localdualnormresidual{\VEC{u}_h^k} \norm{\VEC{\lambda}}_{\tilde{\omega}_T} + \left(\sum_{T'\in\mathcal{T}_T} h_{T'}^2 \left\lVert \VEC{f}-\VEC{\Pi}_{T'}^{p-1}\VEC{f}\right\rVert_{T'}^2\right)^{\nicefrac{1}{2}} \left(\sum_{T'\in\mathcal{T}_T} \left\lVert \psi_{T'}\VEC{w}\right\rVert_{T'}^2\right)^{\nicefrac{1}{2}}.
        \end{split}
    \end{equation}
    Here, we have also used the fact that $\lVert\VEC{f} - \VEC{\Pi}_{T'}^p \VEC{f}\rVert_{T'} \leq 2\lVert \VEC{f}- \VEC{\Pi}_{T'}^{p-1} \VEC{f}\rVert_{T'}$ for $p>0$.
    By the definition of $\VEC{\lambda}$, properties \eqref{eq:first property Verfurth} and \eqref{eq:second property Verfurth}, along with the fact that $\lVert\VEC{w}\rVert_{\tilde{\omega}_T} = 1$, it is possible to show that
    \begin{equation*}
        \norm{\VEC{\lambda}}_{\tilde{\omega}_T} \lesssim 1 \qquad\text{and}\qquad \left(\sum_{T'\in\mathcal{T}_T} \left\lVert \psi_{T'}\VEC{w}\right\rVert_{T'}^2\right)^{\nicefrac{1}{2}} \lesssim 1,
    \end{equation*}
    and, combining the above results with \eqref{eq:rewrite oscillation estimator}, we conclude
    \begin{equation}\label{J1final}
        \mathcal{J}_1 \lesssim \localdualnormresidual{\VEC{u}_h^k} + \eta_{\text{osc},\mathcal{T}_T}^k.
    \end{equation}
    
    Now, we analyze for instance the term $\mathcal{J}_5$: $\mathcal{J}_2$, $\mathcal{J}_3$ and $\mathcal{J}_4$ can be treated in a similar way.
    Using the fact that $\VEC{\sigma^t}(\VEC{u}_h^k) - \VEC{\Pi}_F^{p+1} \left[\VEC{P}_{1,\gamma}^{\VEC{t}}(\VEC{u}_h^k)\right]_{S_h(\VEC{u}_h^k)} \in \VEC{\mathcal{P}}^{p+1}(F)$ for every $F\in\faces{\mathcal{T}_T}{C}$ along with \eqref{eq:sup-inequality2}, we have
    \begin{equation}\label{eq:J4}
        \mathcal{J}_4 \lesssim \sup_{\substack{\VEC{\phi}\in \VEC{\mathcal{P}}^{p+1}(\faces{\mathcal{T}_T}{C}),\\ \lVert\VEC{\phi}\rVert_{\faces{\mathcal{T}_T}{C}} = 1}} \sum_{F\in\faces{\mathcal{T}_T}{C}} \left(\VEC{\sigma^t}(\VEC{u}_h^k) - \VEC{\Pi}_F^{p+1} \left[\VEC{P}_{1,\gamma}^{\VEC{t}}(\VEC{u}_h^k)\right]_{S_h(\VEC{u}_h^k)}, h_F^{\nicefrac{1}{2}} \psi_F\VEC{\phi}\right)_F.
    \end{equation}
    Fix $\VEC{\phi}\in \VEC{\mathcal{P}}^{p+1}(\faces{\mathcal{T}_T}{C})$ with $\lVert\VEC{\phi}\rVert_{\faces{\mathcal{T}_T}{C}} = 1$, and define $\VEC{\lambda}\in \VEC{\mathcal{P}}^{p+d+1}(\mathcal{T}_T) \cap \VEC{H}^1_D(\tilde{\omega}_T)$ satisfying $\VEC{\lambda}|_F = h_F^{\nicefrac{1}{2}}\psi_F \VEC{\phi}|_F$ for every $F\in\faces{\mathcal{T}_T}{C}$ and vanishing outside of $\bigcup_{F\in\faces{\mathcal{T}_T}{C}} \omega_F$. Then, using the Cauchy--Schwarz inequality together with \eqref{eq:rewrite oscillation estimator} and \eqref{eq:rewrite friction estimator} we get
    \begin{equation}\label{eq:J4 bis}
        \begin{split}
            &\sum_{F\in\faces{\mathcal{T}_T}{C}} \left(\VEC{\sigma^t}(\VEC{u}_h^k) - \VEC{\Pi}_F^{p+1} \left[\VEC{P}_{1,\gamma}^{\VEC{t}}(\VEC{u}_h^k)\right]_{S_h(\VEC{u}_h^k)}, h_F^{\nicefrac{1}{2}} \psi_F\VEC{\phi}\right)_F  \\
            &\hspace{3cm} = -\,\langle \mathcal{R}_{\mathcal{T}_T}(\VEC{u}_h^k), \VEC{\lambda}\rangle_{\tilde{\omega}_T} + \sum_{T'\in\mathcal{T}_T} (\VEC{\nabla}\cdot\VEC{\sigma} (\VEC{u}_h^k)+\VEC{f},\VEC{\lambda})_{T'} \\
            &\hspace{3cm}\quad + \sum_{F\in\faces{\mathcal{T}_T}{C}} \left( \left[\VEC{P}_{1,\gamma}^{\VEC{t}}(\VEC{u}_h^k)\right]_{S_h(\VEC{u}_h^k)} - \VEC{\Pi}_F^{p+1} \left[\VEC{P}_{1,\gamma}^{\VEC{t}}(\VEC{u}_h^k)\right]_{S_h(\VEC{u}_h^k)},\VEC{\lambda}\right)_F  \\
            &\hspace{3cm} \lesssim \localdualnormresidual{\VEC{u}_h^k} \norm{\VEC{\lambda}}_{\tilde{\omega}_T} + (\mathcal{J}_1 + \eta_{\text{osc},\mathcal{T}_T}^k) \left(\sum_{T'\in\mathcal{T}_T} \frac{1}{h_{T'}^2} \lVert\VEC{\lambda}\rVert_{T'}^2\right)^{\nicefrac{1}{2}} \\
            &\hspace{3cm}\quad + \eta_{\text{frc},\mathcal{T}_T}^k \Biggl(\sum_{F\in\faces{\mathcal{T}_T}{C}} \frac{1}{h_F} \lVert\VEC{\lambda}\rVert^2_F\Biggr)^{\nicefrac{1}{2}}.
        \end{split}
    \end{equation}
    Exploiting the properties \eqref{eq:third property Verfurth} and \eqref{eq:fourth property Verfurth}, it is possible to show that
    \begin{equation*}
        \norm{\VEC{\lambda}}_{\tilde{\omega}_T} \lesssim 1, \qquad 
        \left(\sum_{T'\in\mathcal{T}_T} \frac{1}{h_{T'}^2} \left\lVert \VEC{\lambda}\right\rVert_{T'}^2\right)^{\nicefrac{1}{2}} \lesssim 1 
        \qquad\text{and}\qquad 
        \Biggl(\sum_{F\in\faces{\mathcal{T}_T}{C}} \frac{1}{h_F} \lVert\VEC{\lambda}\rVert^2_F\Biggr)^{\nicefrac{1}{2}} \lesssim 1,
    \end{equation*}
    and, combining \eqref{eq:J4}, \eqref{eq:J4 bis}, and \eqref{J1final}, we conclude taht
    \begin{equation*}
        \mathcal{J}_4 \lesssim \localdualnormresidual{\VEC{u}_h^k} + \eta_{\text{osc},\mathcal{T}_T}^k + \eta_{\text{frc}, \mathcal{T}_T}^k.
    \end{equation*}
    
    Proceeding in a similar way, it is possible to obtain the following bounds:
    \begin{gather*}
        \mathcal{J}_2 \lesssim \localdualnormresidual{\VEC{u}_h^k} + \eta_{\text{osc},\mathcal{T}_T}^k, \\
        \mathcal{J}_3 \lesssim \localdualnormresidual{\VEC{u}_h^k} + \eta_{\text{osc},\mathcal{T}_T}^k + \eta_{\text{Neu}, \mathcal{T}_T}^k, \\
        \mathcal{J}_4 \lesssim \localdualnormresidual{\VEC{u}_h^k} + \eta_{\text{osc},\mathcal{T}_T}^k + \eta_{\text{cnt}, \mathcal{T}_T}^k.
    \end{gather*}
    Combining all the results obtained so far gives \eqref{eq:upper bound sharp estimator}.
\end{proof}

\begin{lemma}[Control of the local stress estimator]\label{lem:local stress estimator}
    Assume $d=2$.
    Let $\VEC{u}_h^k\in\VEC{V}_h$, let $\VEC{\sigma}_h^k$ be the equilibrated stress defined by Construction~\ref{construction sigmahk}, and let $\eta_{\sharp,T}^k$ be the local residual-based estimator defined by \eqref{eq:definition eta_sharp}. Then, for every element $T\in\mathcal{T}_h$,
    \begin{equation}\label{eq:thesis lemma local flux estimator}
        \eta_{\emph{str},T}^k \lesssim \eta_{\sharp,T}^k.
    \end{equation}
\end{lemma}

\begin{proof}
    Following the path of \cite{Botti2018}, for any element $T\in\mathcal{T}_h$ we introduce the following local nonconforming space \cite{Arbogast1995}:
    \begin{equation*}
        \VEC{M}_T \coloneq  \begin{cases}
            \{\VEC{m}\in \VEC{\mathcal{P}}^{p+2}(T)\ :\ \VEC{m}|_F\in \VEC{\mathcal{P}}^{p+1}(F) \ \text{for any}\ F\in\mathcal{F}_T\} & \quad \text{if}\ p\ \text{is even},\\
            \{\VEC{m}\in \VEC{\mathcal{P}}^{p+2}(T) \ :\ \VEC{m}|_F\in \VEC{\mathcal{P}}^p(F) \oplus \tilde{\VEC{\mathcal{P}}}^{p+2}(F) \ \text{for any}\ F\in\mathcal{F}_T\} & \quad \text{if}\ p\ \text{is odd},
        \end{cases}
    \end{equation*}
    where $\tilde{\VEC{\mathcal{P}}}^{p+2}(F)$ is the $L^2(F)$-orthogonal complement of $\VEC{\mathcal{P}}^{p+1}(F)$ in $\VEC{\mathcal{P}}^{p+2}(F)$.
    Then, for any vertex $\VEC{a}$, on the patch $\omega_{\VEC{a}}$ we define the spaces
    \begin{equation*}
        \begin{split}
            \VEC{M}_h(\omega_{\VEC{a}}) &\coloneq \bigl\{\VEC{m}_h\in \VEC{L}^2(\omega_{\VEC{a}})\ :\ \VEC{m}_h|_T\in\VEC{M}_T\ \text{for any}\ T\in\mathcal{T}_{\VEC{a}}, \\
            &\hspace{1.5cm} (\llbracket\VEC{m}_h\rrbracket,\VEC{v}_h)_F = 0 \ \text{for any}\ \VEC{v}_h\in \VEC{\mathcal{P}}^p(F)\ \text{and for any}\ F\in\mathcal{F}_{\VEC{a}}\setminus \mathcal{F}_h^b \bigr\},
        \end{split}
    \end{equation*}
    and
    \begin{equation*}
        \VEC{M}_h^{\VEC{a}} \coloneq \begin{cases}
            \left\{\VEC{m}_h\in\VEC{M}_h(\omega_{\VEC{a}}) \ :\ (\VEC{m}_h,\VEC{z})_{\omega_{\VEC{a}}} = 0\ \text{for any}\ \VEC{z}\in\VEC{RM}^2 \right\} & \quad \text{if}\ \VEC{a}\in\vertices{h}{i}\ \text{or}\ \VEC{a}\in\vertices{h}{b} \setminus \vertices{h}{D}, \smallskip\\
            \bigl\{\VEC{m}_h\in\VEC{M}_h(\omega_{\VEC{a}})\ :\ (\VEC{m}_h,\VEC{v})_F=0\  & \\
            \hspace{8mm}\text{for any}\ \VEC{v}\in \VEC{\mathcal{P}}^p(F)\ \text{and for any}\ F\in \mathcal{F}_{\VEC{a}}\cap\mathcal{F}_h^D\bigr\} & \quad \text{if}\ \VEC{a}\in\vertices{h}{D}.
        \end{cases}
    \end{equation*}

    Now, fix $T\in\mathcal{T}_h$.
    Combining the definition of the local stress estimator \eqref{eq:stress estimator uhk}, of $\VEC{\sigma}_{h, {\rm dis}}^k$ given by Construction~\ref{construction sigmahk}, of the hat function $\VEC{\psi}_{\VEC{a}}$ with the triangle inequality, we directly get
    \begin{equation}\label{eq:eff stress estimator}
        \eta_{\text{str},T}^k = \left\lVert \VEC{\sigma}_{h, {\rm dis}}^k - \VEC{\sigma}(\VEC{u}_h^k)\right\rVert_T 
        \leq \sum_{\VEC{a}\in\mathcal{V}_T} \left\lVert \VEC{\sigma}_{h, {\rm dis}}^{\VEC{a},k} - \psi_{\VEC{a}}\VEC{\sigma}(\VEC{u}_h^k) \right\rVert_{\omega_{\VEC{a}}}.
    \end{equation}
    Adapting the argument of \cite[Section 4.4]{Botti2018} to our problem with Neumann and frictional contact boundary conditions, it is possible to show that, for any $\VEC{a}\in\mathcal{V}_h$,
    \begin{equation}\label{eq:result of Botti2018}
        \left\lVert \VEC{\sigma}_{h, {\rm dis}}^{\VEC{a},k} - \psi_{\VEC{a}} \VEC{\sigma}(\VEC{u}_h^k) \right\rVert_{\omega_{\VEC{a}}} \lesssim \sup_{\substack{\VEC{m}_h\in\VEC{M}_h^{\VEC{a}}, \\ \lVert\VEC{\nabla}_h \VEC{m}_h\rVert_{\omega_{\VEC{a}}}=1}} \bigl(\VEC{\sigma}_{h, {\rm dis}}^{\VEC{a},k}-\psi_{\VEC{a}}\VEC{\sigma}(\VEC{u}_h^k),\VEC{\nabla}_h\VEC{m}_h \bigr)_{\omega_{\VEC{a}}}
    \end{equation}
    with $\VEC{\nabla}_h$ denoting the standard broken gradient.
    Then, %fixing a vertex $\VEC{a}\in\mathcal{V}_h$ and a function $\VEC{m}_h\in\VEC{M}_h^{\VEC{a}}$ such that $\lVert\VEC{\nabla}_h\VEC{m}_h\rVert_{\omega_{\VEC{a}}}=1$, we have
    applying an integration by parts, using the properties of $\VEC{M}_h^{\VEC{a}}$, and the fact that, by definition, $\VEC{\sigma}_{h, {\rm dis}}^{\VEC{a},k}\in\VEC{\Sigma}_{h, {\rm N}, {\rm C}, {\rm dis}}^{\VEC{a},k}$, we obtain
    
    \begin{equation*}
        \begin{split}
            (\VEC{\sigma}_{h,{\rm dis}}^{\VEC{a},k} - \psi_{\VEC{a}}\VEC{\sigma}(\VEC{u}_h^k), \VEC{\nabla}_h\VEC{m}_h)_{\omega_{\VEC{a}}} &= \sum_{T'\in\mathcal{T}_{\VEC{a}}} (\VEC{\sigma}_{h,{\rm dis}}^{\VEC{a},k}-\psi_{\VEC{a}}\VEC{\sigma}(\VEC{u}_h^k),\VEC{\nabla}_h\VEC{m}_h)_{T'}  \\
            &\hspace{-2cm}= \underbrace{-\sum_{T'\in\mathcal{T}_{\VEC{a}}} \bigl(\VEC{\rm div}\, (\VEC{\sigma}_{h,{\rm dis}}^{\VEC{a},k} - \psi_{\VEC{a}}\VEC{\sigma}(\VEC{u}_h^k)), \VEC{m}_h\bigr)_{T'}}_{=:\mathcal{I}_1} + \underbrace{\sum_{F\in\faces{\VEC{a}}{i}} \bigl(\llbracket \psi_{\VEC{a}}\VEC{\sigma}(\VEC{u}_h^k)\VEC{n}_F\rrbracket, \VEC{m}_h\bigr)_F}_{=:\mathcal{I}_2} + \\
            &\hspace{-1cm}+ \underbrace{\sum_{F\in\faces{\VEC{a}}{N}} \left(\VEC{\Pi}_F^p \left(\psi_{\VEC{a}}\VEC{g}_N\right) - \psi_{\VEC{a}}\VEC{\sigma}(\VEC{u}_h^k)\VEC{n}, \VEC{m}_h\right)_F}_{=:\mathcal{I}_3} +\\
            &+ \underbrace{\sum_{F\in\faces{\VEC{a}}{C}} \left({\Pi}_F^p \left(\psi_{\VEC{a}} \left[P_{1,\gamma}^n(\VEC{u}_h^k)\right]_{\MAT{R}^-}\right) - \psi_{\VEC{a}}{\sigma}^n(\VEC{u}_h^k), {m}_h^n\right)_F}_{=:\mathcal{I}_4} +\\
          &\hspace{1cm}+ \underbrace{\sum_{F\in\faces{\VEC{a}}{C}} \left(\VEC{\Pi}_F^p \left(\psi_{\VEC{a}} \left[\VEC{P}_{1,\gamma}^{\VEC{t}}(\VEC{u}_h^k)\right]_{S_h(\VEC{u}_h^k)}\right) - \psi_{\VEC{a}}\VEC{\sigma^t}(\VEC{u}_h^k), \VEC{m}_h^{\VEC{t}}\right)_F}_{=:\mathcal{I}_5}.
        \end{split}
    \end{equation*}
    The first two terms can be treated as in \cite[Proof of Theorem 4.7]{Botti2018}, obtaining
    \begin{gather*}
        \mathcal{I}_1 \lesssim \left[
          \sum_{T'\in\mathcal{T}_{\VEC{a}}} h_{T'}^2 \left\lVert\VEC{\rm div}\, \VEC{\sigma}(\VEC{u}_h^k) + \VEC{\Pi}_{T'}^p \VEC{f}\right\rVert_{T'}^2
          \right]^{\nicefrac{1}{2}} \left\lVert\VEC{\nabla}_h \VEC{m}_h\right\rVert_{\omega_{\VEC{a}}},
        \\
        \mathcal{I}_2 \lesssim \left[
          \sum_{F\in\faces{\VEC{a}}{i}} h_F\left\lVert \llbracket\VEC{\sigma}(\VEC{u}_h^k)\VEC{n}_F\rrbracket \right\rVert_F^2
          \right]^{\nicefrac{1}{2}} \left\lVert\VEC{\nabla}_h \VEC{m}_h\right\rVert_{\omega_{\VEC{a}}}.
    \end{gather*}
    In a similar way, using the Cauchy--Schwarz inequality, the discrete trace inequality $\lVert\VEC{m}_h\rVert_F \lesssim h_F^{-\nicefrac{1}{2}} \lVert\VEC{m}_h\rVert_{T'}$, and the discrete Poincaré inequality \cite{Vohralik2005} when $\VEC{a}\notin \vertices{h}{D}$ and the discrete Friedrichs inequality \cite{Vohralik2005} when $\VEC{a}\in\vertices{h}{D}$, together with the definition of $\VEC{M}_h^{\VEC{a}}$, we have
    \begin{equation*}
        \begin{split}
            \mathcal{I}_5 &= \sum_{F\in\faces{\VEC{a}}{C}} \left(\psi_{\VEC{a}} \left(\left[\VEC{P}_{1,\gamma}^{\VEC{t}}(\VEC{u}_h^k)\right]_{S_h(\VEC{u}_h^k)} - \VEC{\sigma^t}(\VEC{u}_h^k)\right), \VEC{\Pi}_F^p \VEC{m}_h^{\VEC{t}}\right)_F \\
            &= \sum_{F\in\faces{\VEC{a}}{C}} \left(\VEC{\Pi}_F^{p+1} \left[\VEC{P}_{1,\gamma}^{\VEC{t}}(\VEC{u}_h^k)\right]_{S_h(\VEC{u}_h^k)} - \VEC{\sigma^t}(\VEC{u}_h^k), \psi_{\VEC{a}} \VEC{\Pi}_F^p \VEC{m}_h^{\VEC{t}}\right)_F \\
            &\leq \left[\sum_{F\in\faces{\VEC{a}}{C}} h_F 
            \left\lVert \psi_{\VEC{a}} \left(\VEC{\Pi}_F^{p+1} \left[\VEC{P}_{1,\gamma}^{\VEC{t}}(\VEC{u}_h^k)\right]_{S_h(\VEC{u}_h^k)} - \VEC{\sigma^t}(\VEC{u}_h^k)\right) \right\rVert_F^2\right]^{\nicefrac{1}{2}} \left[\sum_{F\in\faces{\VEC{a}}{C}} \frac{1}{h_F} \left\lVert\VEC{m}_h\right\rVert_F^2 \right]^{\nicefrac{1}{2}} \\
            &\lesssim \left[\sum_{F\in\faces{\VEC{a}}{C}} h_F 
            \left\lVert \VEC{\Pi}_F^{p+1} \left[\VEC{P}_{1,\gamma}^{\VEC{t}}(\VEC{u}_h^k)\right]_{S_h(\VEC{u}_h^k)} - \VEC{\sigma^t}(\VEC{u}_h^k) \right\rVert_F^2\right]^{\nicefrac{1}{2}} \left\lVert\VEC{\nabla}_h\VEC{m}_h\right\rVert_{\omega_{\VEC{a}}},
        \end{split}
    \end{equation*}
    and also
    \begin{gather*}
        \mathcal{I}_3 \lesssim\left[\sum_{F\in\faces{\VEC{a}}{N}} h_F \left\lVert \VEC{\Pi}_F^{p+1} \VEC{g}_N - \VEC{\sigma}(\VEC{u}_h^k)\VEC{n}\right\rVert_F^2 \right]^{\nicefrac{1}{2}} \left\lVert\VEC{\nabla}_h\VEC{m}_h\right\rVert_{\omega_{\VEC{a}}} \\
        \mathcal{I}_4 \lesssim \left[\sum_{F\in\mathcal{F}_{\VEC{a}}^C} h_F \left\lVert \VEC{\Pi}_F^{p+1} \Bigl(\left[P_{1,\gamma}(\VEC{u}_h^k)\right]_{\mathbb{R}^-} \VEC{n}\Bigr) - \VEC{\sigma}(\VEC{u}_h^k)\VEC{n}\right\rVert_F^2 \right]^{\nicefrac{1}{2}} \left\lVert\VEC{\nabla}_h\VEC{m}_h\right\rVert_{\omega_{\VEC{a}}}
    \end{gather*}
    Combining all the above results with \eqref{eq:eff stress estimator} and \eqref{eq:result of Botti2018} yields \eqref{eq:thesis lemma local flux estimator}.
\end{proof}

\begin{remark}[Case $d=3$]
    In the case $d=3$, it becomes more challenging to identify a space similar to $\VEC{M}_h^{\VEC{a}}$ with the appropriate features to recover \eqref{eq:result of Botti2018} and do the subsequent analysis. For this reason, in this paper, we present the proof of Lemma~\ref{lem:local stress estimator} specifically for the case $d=2$.
\end{remark}

Finally, \eqref{eq:thesis local efficiency1} follows by combining the results of Lemmas~\ref{lem:local sharp estimator} and \ref{lem:local stress estimator}, and \eqref{eq:thesis local efficiency2} by using also the local stopping criterion \eqref{eq:local stopping criterion}.

\subsection{Proof of the global efficiency}

With the aim to prove \eqref{eq:thesis global efficiency1} and \eqref{eq:thesis global efficiency2} of Theorem~\ref{th:global efficiency} we introduce the global version of the local estimator $\eta_{\sharp,T}^k$:
\begin{equation}\label{eq:definition eta_sharp global}
  \eta_{\sharp}^k \coloneqq \left[
    \sum_{T\in\mathcal{T}_h} \left(\eta_{\sharp,T}^k\right)^2
    \right]^{\nicefrac{1}{2}}.
\end{equation}

\begin{lemma}[Control of the residual-based estimator $\eta_{\sharp}^k$]\label{lem:global sharp estimator}
  Let $\VEC{u}_h^k\in\VEC{V}_h$, let $\VEC{\sigma}_h^k$ be the equilibrated stress defined by Construction~\ref{construction sigmahk}, and let $\eta_{\sharp}^k$ be the global residual-based estimator defined by \eqref{eq:definition eta_sharp global}.
  Then, for any $k\geq 1$
  \begin{equation}\label{thesis-global sharp estimator}
    \eta_{\sharp}^k \lesssim \dualnormresidual{\VEC{u}_h^k} + \eta_{\emph{osc}}^k + \eta_{\emph{Neu}}^k + \eta_{\emph{cnt}}^k + \eta_{\emph{frc}}^k.
  \end{equation}
\end{lemma}

\begin{proof}
  We have
  \begin{equation*}
    \begin{split}
      \eta_{\sharp}^k \lesssim& \underbrace{\left(\sum_{T\in\mathcal{T}_h} h_T^2 \left\lVert\VEC{\rm div}\, \VEC{\sigma}(\VEC{u}_h^k) + \VEC{\Pi}_T^p \VEC{f}\right\rVert_T^2\right)^{\nicefrac{1}{2}}}_{=: \mathcal{L}_1} + \underbrace{\left(\sum_{F\in\faces{h}{i}} h_F\left\lVert \llbracket\VEC{\sigma}(\VEC{u}_h^k)\VEC{n}_F\rrbracket\right\rVert_F^2 \right)^{\nicefrac{1}{2}}}_{=:\mathcal{L}_2} + \\
      & \hspace{2cm}+ \underbrace{\left(\sum_{F\in\faces{h}{N}} h_F\left\lVert \VEC{\sigma}(\VEC{u}_h^k)\VEC{n} - \VEC{\Pi}_F^{p+1}\VEC{g}_N \right\rVert_F^2\right)^{\nicefrac{1}{2}}}_{=:\mathcal{L}_3} + \\
      & \hspace{3cm} + \underbrace{\left(\sum_{F\in\faces{h}{C}} h_F \left\lVert\sigma^n(\VEC{u}_h^k) - \Pi_F^{p+1} \left[P_{1,\gamma}^n(\VEC{u}_h^k)\right]_{\mathbb{R}^-} \right\rVert_F^2\right)^{\nicefrac{1}{2}}}_{=:\mathcal{L}_4} + \\
      & \hspace{4cm} + \underbrace{\left(\sum_{F\in\faces{h}{C}} h_F \left\lVert\VEC{\sigma}^{\VEC{t}} (\VEC{u}_h^k) - \VEC{\Pi}_F^{p+1} \left[\VEC{P}_{1,\gamma}^{\VEC{t}}(\VEC{u}_h^k)\right]_{S_h(\VEC{u}_h^k)} \right\rVert_F^2\right)^{\nicefrac{1}{2}}}_{=:\mathcal{L}_5}
    \end{split}
  \end{equation*}
  Proceeding as in the proof of Lemma~\ref{lem:local sharp estimator}, it is possible to show that
  \begin{gather*}
    \mathcal{L}_1 \lesssim \dualnormresidual{\VEC{u}_h^k} + \eta_{\text{osc}}^k, 
    \qquad\quad 
    \mathcal{L}_2 \lesssim \dualnormresidual{\VEC{u}_h^k} + \eta_{\text{osc}}^k,
    \qquad\quad
    \mathcal{L}_3 \lesssim \dualnormresidual{\VEC{u}_h^k} + \eta_{\text{osc}}^k + \eta_{\text{Neu}}^k, \\
    \mathcal{L}_4 \lesssim \dualnormresidual{\VEC{u}_h^k} + \eta_{\text{osc}}^k + \eta_{\text{cnt}}^k,
    \qquad\quad
    \mathcal{L}_4 \lesssim \dualnormresidual{\VEC{u}_h^k} + \eta_{\text{osc}}^k + \eta_{\text{frc}}^k.\qedhere
  \end{gather*}
\end{proof}

\begin{lemma}[Control of the global stress estimator]\label{lem:global stress estimator}
Let $\VEC{u}_h^k\in\VEC{V}_h$, let $\VEC{\sigma}_h^k$ be the equilibrated stress defined by Construction~\ref{construction sigmahk}, and let $\eta_{\sharp}^k$ be the global residual-based estimator defined by \eqref{eq:definition eta_sharp global}.
    Then, for any $k\geq 1$, 
    \begin{equation*}  
        \eta_{\emph{str}}^k \lesssim \eta_{\sharp}^k.
    \end{equation*}
\end{lemma}

\begin{proof}
    It is an immediate consequence of Lemma~\ref{lem:local stress estimator}.
\end{proof}

Finally, \eqref{eq:thesis global efficiency1} follows by combining the results of Lemmas~\ref{lem:global sharp estimator} and \ref{lem:global stress estimator}, and \eqref{eq:thesis global efficiency2} by using also the global stopping criterion of Line~\ref{alg:global stopping criterion} of Algorithm~\ref{algorithm}.

\section*{Acknowledgements}

Funded by the European Union (ERC Synergy, NEMESIS, project number 101115663).
Views and opinions expressed are however those of the author(s) only and do not necessarily reflect those of the European Union or the European Research Council Executive Agency. Neither the European Union nor the granting authority can be held responsible for them.

\bibliography{refs_tresca-apost}
\bibliographystyle{siam}
\nocite{*}

\end{document}